\documentclass[a4paper,11pt,fleqn]{amsart}
\usepackage[utf8]{inputenc}
\usepackage{epsfig,graphics}
\usepackage[all,knot,poly,color,curve]{xy}
\usepackage{array}
\usepackage{amsthm,amssymb,amsbsy,bbm,a4wide,verbatim,url}
\usepackage[marginpar]{todo}
\usepackage{longtable}
\usepackage{rotating}
\usepackage{color,lipsum}

\usepackage{fancybox}
\usepackage{enumitem}

\usepackage{tikz,tkz-euclide}
\newcounter{braid}
\newcounter{strands}
\pgfkeyssetvalue{/tikz/braid height}{1cm}
\pgfkeyssetvalue{/tikz/braid width}{1cm}
\pgfkeyssetvalue{/tikz/braid start}{(0,0)}
\pgfkeyssetvalue{/tikz/braid colour}{black}
\pgfkeys{/tikz/strands/.code={\setcounter{strands}{#1}}}

\makeatletter
\def\cross{%
  \@ifnextchar^{\message{Got sup}\cross@sup}{\cross@sub}}

\def\cross@sup^#1_#2{\render@cross{#2}{#1}}

\def\cross@sub_#1{\@ifnextchar^{\cross@@sub{#1}}{\render@cross{#1}{1}}}

\def\cross@@sub#1^#2{\render@cross{#1}{#2}}

\def\render@cross#1#2{
  \def\strand{#1}
  \def\crossing{#2}
  \pgfmathsetmacro{\cross@y}{-\value{braid}*\braid@h}
  \pgfmathtruncatemacro{\nextstrand}{#1+1}
  \foreach \thread in {1,...,\value{strands}}
  {
    \pgfmathsetmacro{\strand@x}{\thread * \braid@w}
    \ifnum\thread=\strand
    \pgfmathsetmacro{\over@x}{\strand * \braid@w + .5*(1 - \crossing) * \braid@w}
    \pgfmathsetmacro{\under@x}{\strand * \braid@w + .5*(1 + \crossing) * \braid@w}
    \draw[braid] \pgfkeysvalueof{/tikz/braid start} +(\under@x pt,\cross@y pt) to[out=-90,in=90] +(\over@x pt,\cross@y pt -\braid@h);
    \draw[braid] \pgfkeysvalueof{/tikz/braid start} +(\over@x pt,\cross@y pt) to[out=-90,in=90] +(\under@x pt,\cross@y pt -\braid@h);
    \else
    \ifnum\thread=\nextstrand
    \else
     \draw[braid] \pgfkeysvalueof{/tikz/braid start} ++(\strand@x pt,\cross@y pt) -- ++(0,-\braid@h);
    \fi
   \fi
  }
  \stepcounter{braid}
}

\tikzset{braid/.style={double=\pgfkeysvalueof{/tikz/braid colour},double distance=1pt,line width=2pt,white}}

\newcommand{\braid}[2][]{%
  \begingroup
  \pgfkeys{/tikz/strands=2}
  \tikzset{#1}
  \pgfkeysgetvalue{/tikz/braid width}{\braid@w}
  \pgfkeysgetvalue{/tikz/braid height}{\braid@h}
  \setcounter{braid}{0}
  \let\dsigma=\cross
  #2
  \endgroup
}
\makeatother

\def\GQ{\mathrm{Q}}

\def\into{\hookrightarrow}
\def\onto{\twoheadrightarrow}
\def\diag{\mathrm{diag}}

\def\kk{\mathbbm{k}}

\def\g{\mathfrak{g}}

\def\C{\ensuremath{\mathbbm{C}}}
\def\Q{\mathbbm{Q}}

\def\Z{\mathbbm{Z}}

\def\N{\mathbbm{N}}
\def\F{\mathbbm{F}}

\def\sl{\mathfrak{sl}}

\newcommand{\la}{\lambda}

\def\eps{\varepsilon}
\def\ii{\mathrm{i}}
\def\dd{\mathrm{d}}

\def\Ker{\mathrm{Ker}}

\def\End{\mathrm{End}}
\def\Hom{\mathrm{Hom}}
\def\Res{\mathrm{Res}}

\def\Ug{\mathsf{U}\g}

\newtheorem{theorem}{Theorem}[section]
\newtheorem{lemma}[theorem]{Lemma}

\newtheorem{remark}[theorem]{Remark}

\newtheorem{corollary}[theorem]{Corollary}

\newtheorem{lemme}[theorem]{Lemme}
\newtheorem{proposition}[theorem]{Proposition}

\numberwithin{equation}{section}

\date{November 9, 2018}

\author{Ivan Marin}
\title{A maximal  cubic quotient of the braid algebra}

\newlength\Colsep
\begin{document}

\maketitle

{\bf Abstract.} We study a quotient of the group algebra of the braid group in which the Artin generators satisfy a cubic relation. This quotient is maximal among the ones satisfying
such a cubic relation. It is finite-dimensional for at least $n \leq 5$ and we investigate its module structure in this range. 
We also investigate the proper quotients of it that appear in the realm of quantum groups, and describe another maximal quotient
related to the usual Hecke algebras. Finally, we describe the connection between this algebra and a quotient of the algebra of horizontal chord diagrams introduced by Vogel. We prove that these two are isomorphic for $n \leq 5$.

\tableofcontents

\section{Introduction}

Since its introduction 80 years ago by E. Artin, the braid
group and its representations have been recognized as object of great importance in a number of mathematical problems, in topology and elsewhere. There is one finite-dimensional quotient of the group algebra of the braid group -- braid algebra for short -- with a simple description that is well-understood. It is the Hecke algebra, where the additional
relation is a quadratic relation on any Artin generator
on 2 strands. It is thus worthwhile to get
similar quotients.

Replacing the quadratic relation by a generic cubic one,
one gets another quotient which deserves to be called in this context a cubic Hecke algebra, but which is not
finite-dimensional (for at least one specialization
of the parameters). However, a number of already defined finite-dimensional quotients of the braid algebra (some
of them originating from the quantum world and in the realm
of Vassiliev invariants) factor through this cubic Hecke
algebra. It is thus tempting to look for a new quotient,
covering the usual ones, and whose defining relations would
involve as few strands as possible.

Since there is no other possible relation on 2 strands
than the basic cubic one, the next possibility is to look
for relations on 3 strands. It turns out that the generic cubic Hecke algebras \emph{on at most 5 strands} are finite-dimensional and semisimple, and therefore there is
only a finite number of ideals by which it is possible to divide out. Therefore a `maximal cubic quotient' defined
on the fewest possible number of strands is uniquely
defined by one of (the isotypic component attached to) irreducible representations of the cubic Hecke algebra on 3 strands. This explains the term \emph{maximal} used in the title of the present paper.

It turns out that, up to some Galois symmetries, there
are only 3 possibilities for such a maximal quotient,
corresponding to representations of dimension 1, 2 and 3.
We prove below, elaborating on previous joint work
with M. Cabanes, that the one corresponding to the 3-dimensional representation is finite dimensional
but is closely related to the usual, quadratic Hecke algebra. The one related to the 1-dimensional one is still
mysterious, although we proved in \cite{CABANESMARIN} that the quotient related to the collection of all 1-dimensional representations collapses on 5 strands. In the present paper we study in detail the maximal quotient related to an arbitrary 2-dimensional irreducible representation.

This quotient is particularly interesting because all
the quantum constructions we know of which provide quotients of the cubic Hecke algebras factorize
through this quotient. Moreover, there are reasons to
hope that this quotient is itself finite-dimensional. Finally, it has the advantage of being defined simply by
\begin{itemize}
\item the generic cubic relation on 2 strands
\item a single, relatively simple, braid relation on 3 strands
\end{itemize}

Although we do not solve the question of its  finite dimensionality in this paper for $n \geq 6$,
we tried to provide a thorough algebraic study of
this quotient on at most 5 strands. This includes
\begin{itemize}
\item  the
comparaison with an infinitesimal algebra introduced by P. Vogel in \cite{VOGELEXP}
\item the determination of a convenient $\Z[a^{\pm 1},b^{\pm 1},c^{\pm 1}]$-module structure for it, which is
proved to be free on at most 4 strands, with an explicit basis.
\item determination of generators (for at most 5 strands)
as modules over the subalgebra on 1 strand less
\end{itemize} 

The $\Z[a^{\pm 1},b^{\pm 1},c^{\pm 1}]$-module structure we define is the quotient of the group algebra
of the braid group on $n$ strands by the cubic relation $(s_1-a)(s_1-b)(s_1-c) = 0$, where $s_1,\dots,s_{n-1}$ denote the Artin generators, together with the relation
$$
\begin{array}{lcl}
 s_1^{-1}s_2s_1 &=& s_1 s_2 s_1^{-1} - a^{-1} s_1 s_2 + a s_1 s_2^{-1} + a s_1^{-1}s_2 - a^3 s_1^{-1} s_2^{-1}
+ a^{-1} s_2 s_1 - a s_2 s_1^{-1} - a s_2^{-1} s_1 \\
 & & + a^3 s_2^{-1}s_1^{-1} + a^2 s_1^{-1} s_2^{-1} s_1
- a^2 s_1 s_2^{-1} s_1^{-1}  \\
\end{array}
$$
that is

\begin{center}

\begin{tikzpicture}
\begin{scope}[scale=.28]
\braid[braid colour=black,strands=3,braid start={(0,0)}]%
{ \dsigma _1^{-1} \dsigma_2 \dsigma_1 }
\end{scope}
\begin{scope}[shift={(1.1,-.5)}]
\draw (0,0) node {$=$};
\end{scope}
\begin{scope}[scale=.28,shift={(3.8,0)}]
\braid[braid colour=black,strands=3,braid start={(0,0)}]%
{ \dsigma _1 \dsigma_2 \dsigma_1^{-1} }
\end{scope}
\begin{scope}[shift={(2.4,-.5)}]
\draw (0,0) node {$-a^{-1}$};
\end{scope}
\begin{scope}[scale=.28,shift={(9.6,-.5)}]
\braid[braid colour=black,strands=3,braid start={(0,0)}]%
{ \dsigma _1 \dsigma_2  }
\end{scope}
\begin{scope}[shift={(3.8,-.5)}]
\draw (0,0) node {$+a$};
\end{scope}
\begin{scope}[scale=.28,shift={(14,-.5)}]
\braid[braid colour=black,strands=3,braid start={(0,0)}]%
{ \dsigma _1 \dsigma_2^{-1}  }
\end{scope}
\begin{scope}[shift={(5.1,-.5)}]
\draw (0,0) node {$+a$};
\end{scope}
\begin{scope}[scale=.28,shift={(18.5,-.5)}]
\braid[braid colour=black,strands=3,braid start={(0,0)}]%
{ \dsigma _1^{-1} \dsigma_2  }
\end{scope}
\begin{scope}[shift={(6.35,-.5)}]
\draw (0,0) node {$-a^3$};
\end{scope}
\begin{scope}[scale=.28,shift={(23,-.5)}]
\braid[braid colour=black,strands=3,braid start={(0,0)}]%
{ \dsigma _1^{-1} \dsigma_2^{-1}  }
\end{scope}
\begin{scope}[shift={(7.9,-.5)}]
\draw (0,0) node {$+a^{-1}$};
\end{scope}
\begin{scope}[scale=.28,shift={(29,-.5)}]
\braid[braid colour=black,strands=3,braid start={(0,0)}]%
{ \dsigma _2 \dsigma_1  }
\end{scope}
\begin{scope}[shift={(9.2,-.5)}]
\draw (0,0) node {$-a$};
\end{scope}
\begin{scope}[scale=.28,shift={(33,-.5)}]
\braid[braid colour=black,strands=3,braid start={(0,0)}]%
{ \dsigma _2 \dsigma_1^{-1}  }
\end{scope}
\begin{scope}[shift={(10.4,-.5)}]
\draw (0,0) node {$-a$};
\end{scope}
\begin{scope}[scale=.28,shift={(37.4,-.5)}]
\braid[braid colour=black,strands=3,braid start={(0,0)}]%
{ \dsigma _2^{-1} \dsigma_1  }
\end{scope}

\begin{scope}[shift={(11.65,-.5)}]
\draw (0,0) node {$+a^3$};
\end{scope}
\begin{scope}[scale=.28,shift={(42,-.5)}]
\braid[braid colour=black,strands=3,braid start={(0,0)}]%
{ \dsigma _2^{-1} \dsigma_1^{-1}  }
\end{scope}

\begin{scope}[shift={(12.9,-.5)}]
\draw (0,0) node {$+a^2$};
\end{scope}
\begin{scope}[scale=.28,shift={(46.5,-.5)}]
\braid[braid colour=black,strands=3,braid start={(0,0)}]%
{ \dsigma _1^{-1} \dsigma_2^{-1}\dsigma_1  }
\end{scope}

\begin{scope}[shift={(14.15,-.5)}]
\draw (0,0) node {$-a^2$};
\end{scope}
\begin{scope}[scale=.28,shift={(50.9,-.5)}]
\braid[braid colour=black,strands=3,braid start={(0,0)}]%
{ \dsigma _1 \dsigma_2^{-1}\dsigma_1^{-1}  }
\end{scope}
\end{tikzpicture}
\end{center}

This relation, to the best of our knowledge, has first been exhibited by Ishii, in his study of the Links-Gould polynomial
(see \cite{ISHII}).

We hope that this algebraic work will be useful, both
for the further study of this quotient on a higher number
of strands, and concerning the study of specializations of these algebras as well. In particular, it is important to have a well-defined algebra over a generic ring of the form
$\Z[a^{\pm 1},b^{\pm 1},c^{\pm 1}]$ in order to be able
to deal with specialization, and it is known in particular
that some of the Markov traces occuring on some of the cubic quotients cannot be defined over $\Q(a,b,c)$ -- notably the one producing the Kauffman polynomial.

\bigskip

The paper is organized as follows. In the preliminary section 2 we provide background on the cubic Hecke algebras,
and start by a definition of the quotient $Q_n$ over the field $\Q(a,b,c)$ -- that we denote $Q_n \otimes K$
for coherence with the sequence. This definition does not involve any braid formula, and is enough to justify that a few quantum quotients of the braid group factorize through it. With the exception of the BMW algebra, which is quite well known (see e.g.
\cite{BW,WENZL,TRBMW}), we describe them for small $n$.  We also prove here that the other 'maximal cubic quotient' mentionned above is actually a `tripled' version of the ordinary (quadratic) Hecke algebra.
Finally, we introduce Vogel's algebra and explain the connection with our quotient $Q_n$.

Section 3 investigates the structure of $Q_3$ on 3 strands,
notably from a computer algebra point of view. It also establishes some tools which are useful for the sequel.

The module structure of $Q_4$ is determined in section 4, as well as its structure as a $Q_3$-bimodule. Finally, we prove in section 5 that $Q_5$ is a finitely generated module, and investigate it as a $Q_4$-bimodule.

\section{Preliminaries}

\subsection{The cubic Hecke algebras}

Let us denote
$
R = \Z[a,a^{-1},b,b^{-1},c,c^{-1}] = \Z[a,b,c,(abc)^{-1}]$.
 We let $H_n$ denote
the $R$-algebra defined as the quotient of the group algebra $R B_n$
of the braid group on $n$ strands by the relations $(s_i-a)(s_i-b)(s_i-c) = 0$
for $1 \leq i \leq n-1$ or, equivalently -- since each $s_i$ is conjugated to $s_1$ --
by the relation $(s_1-a)(s_1-b)(s_1-c)= 0$. It is known that $H_n$ is a
free $R$-module of finite rank for $n \leq 5$ (see \cite{CUBIC5}) and that the specialization
of $H_n$ at $\{a,b,c\} = \mu_3(\C)$ -- that is, the group algebra of $B_n/s_1^3$ --
is infinite-dimensional for $n\geq 6$ by a theorem of Coxeter (see \cite{COXETER}). In the present state of knowledge it remains however possible that
$H_n$ is finite dimensional for other values of $n$ when extended over $\Q(a,b,c)$, although there is no evidence
in this direction so far.

\subsubsection{The cubic Hecke algebra for $n=3$}
\label{sect:H3}
More precisely, for $n =3$, one may excerpt from \cite{CUBIC5} the
following result (see also \cite{BROUEMALLE, FUNAR,THESE} for related statements).

\begin{proposition} {\ } \label{prop:H3libre}
\begin{enumerate}
\item The algebra $H_3$ is a free $H_2$-module of rank $8$, with basis the elements $1, s_2,s_2^{-1}$, $s_1^{\alpha} s_2^{\beta}$ for $\alpha,\beta \in \{ 1,-1 \}$,
$s_2 s_1^{-1} s_2$.
\item The algebra $H_3$ is a free $R$-module of rank 24, with basis the
elements 
$$
\begin{array}{lcl}
\mathcal{B}_1&=& (1,  s_1, s_1^{-1}, s_2, s_2^{-1}, s_1s_2, s_1s_2^{-1}, s_1^{-1}s_2, s_1^{-1}s_2^{-1}, s_1s_2s_1, s_1s_2s_1^{-1}, s_1^{-1}s_2s_1, 
  s_1^{-1}s_2s_1^{-1},\\ & & s_1s_2^{-1}s_1, s_1^{-1}s_2^{-1}s_1, s_2s_1, s_2^{-1}s_1, s_2s_1^{-1}, s_2^{-1}s_1^{-1}, s_1s_2^{-1}s_1^{-1}, 
  s_1^{-1}s_2^{-1}s_1^{-1}, s_2s_1^{-1}s_2,\\ & & s_1s_2s_1^{-1}s_2, s_1^{-1}s_2s_1^{-1}s_2 ).
  \end{array}
$$ {}
\end{enumerate}
\end{proposition}
\begin{proof}  From \cite{CUBIC5} theorem 3.2 we know that $H_3$ is generated as a $H_2$-module by the $8$ elements on the first statement. Since $H_2$
is spanned by $1,s_1, s_1^{-1}$ it follows that $H_3$ is generated as a $H_3$-module by the $24$ elements of the
second statement. Since $\Gamma_3$ has 24 elements and by an argument of \cite{BMR} (see also \cite{CYCLO}, proposition 2.4 (1)) it follows that these $24$ elements are 
a basis over $R$ of $H_3$. It readily follows that the $8$ original elements provide a basis of $H_3$ as a $H_2$-module.
\end{proof}

A consequence is that $H_3$ is a free deformation
of the group algebra $R \Gamma_3$,
where $\Gamma_n$ denotes the quotient of the braid group by the relations $s_i^3 = 1$,
and $H_3$ becomes isomorphic to it after extension of scalars to the algebraic
closure $\overline{K}$ of the field of fractions $K$ of $R$. Actually, one
has the stronger result $H_3 \otimes_R K \simeq K \Gamma_3$,
because the irreducible representations of $K H_3$ 
are absolutely irreducible.

We will use the following explicit matrix models for the representations, which
are basically the same which were obtained in \cite{BROUEMALLE}, \S 5B. We endow $\{a,b,c \}$
with the total ordering $a<b<c$. We denote
\begin{enumerate}
\item $S_{\alpha}$ for $\alpha \in \{a,b,c\}$ the 1-dimensional representation
$s_1,s_2 \mapsto \alpha$ 
\item $U_{\alpha,\beta}$ for $\alpha , \beta \in  \{a,b,c\}$ with $\alpha < \beta$ the 2-dimensional representation
$$
U_{\alpha,\beta} : 
s_1 \mapsto \left( \begin {array}{cc} \alpha&0\\ \noalign{\medskip}-\alpha&\beta\end {array} \right)
s_2 \mapsto \left( \begin {array}{cc} \beta&\beta\\ \noalign{\medskip}0&\alpha\end {array} \right)
$$
\item $V$ the $3$-dimensional irreducible representation 
$$
s_1 \mapsto \left( \begin {array}{ccc} c&0&0\\ \noalign{\medskip}ac+{b}^{2}&b&0\\ \noalign{\medskip}b&1&a\end {array} \right)
s_2 \mapsto \left( \begin {array}{ccc} a&-1&b\\ \noalign{\medskip}0&b&-ac-{b}^{2}\\ \noalign{\medskip}0&0&c\end {array} \right) 
$$
\end{enumerate}
We note the important feature that these representations are actually defined over $R$.
As a consequence, these formulas provide an explicit embedding
$$
\Phi_{H_3} : H_3 \into R^3 \oplus M_2(R)^3 \oplus M_3(R) 
$$
with the RHS being a free $R$-module of rank $24$.

Another interesting property that we have in $H_3$ is the following relation (see \cite{CUBIC5}, lemma 3.6).
\begin{equation} \label{eq2b12b1com}
s_2^{-1}s_1s_2^{-1}s_1 - s_1 s_2^{-1}s_1s_2^{-1} \in u_1u_2+u_2u_1
\end{equation}

Finally we recall from e.g. \cite{TRBMW} section 2 that specializations of $H_3$ remain semisimple as soon as 
$ (x-y)(x^2-xy+y^2)(xy+z^2) \neq 0
$ for $\{x,y,z \} = \{a,b,c \}$.

\subsubsection{The cubic Hecke algebra on 4 strands}
\label{sect:cubicHecke4}

A description of the irreducible representations of $H_4$ can
be found in \cite{THESE} and \cite{LG}. We use the same notation here.
There are
\begin{itemize}
\item three 1-dimensional representations $S_x$ for $x \in \{a,b,c\}$,
defined by $s_i \mapsto x$. 
\item three 2-dimensional representations $T_{x,y}$
indexed by the subsets $\{x,y \} \subset \{ a,b,c \}$ of cardinality 2,
which factorize through the special morphism $B_4 \to B_3$ (hence through $H_3$).
\item one 3-dimensional representation $V$, factorizing through $B_3$.
\item six 3-dimensional representations $U_{x,y}$ for each tuple $(x,y)$
with $x \neq y$ and $x,y \in \{a,b,c \}$. 
\item six 6-dimensional representations $V_{x,y,z}$ for each tuple
$(x,y,z)$ with $\{x,y,z \} = \{ a,b,c \}$
\item three 8-dimensional representations $W_x$ for $x \in \{ a,b,c \}$
\item two 9-dimensional representations $X$, $X'$.
\end{itemize}

Except for $X,X'$, they are uniquely determined by their restriction
to $B_3$. 
$$
\begin{array}{lcl}
\Res U_{x,y} & = & S_x + U_{x,y} \\
\Res V_{x,y,z} &=& S_x + U_{x,y} + V \\
\Res W_x &=& S_x + U_{x,y} + U_{x,z} + V \\  
\Res X &= & U_{x,y} + U_{x,z} + U_{y,z} + V
\end{array}
$$
The representations $U_{x,y}$ of $B_3$ are well-determined by their
restriction to $B_2$ : the restriction to $B_2$ of
$U_{x,y}$ is the sum of two 1-dimensional representations on which $s_1$ acts by $x$ and $y$, respectively.

A complete set of matrices for these representations was first found by Brou\'e and Malle in \cite{BROUEMALLE}.
Other constructions were subsequently given, in \cite{THESE} and \cite{MALLEMICHEL}.
The latter ones have been included in the development version of the CHEVIE package for GAP3,
and the order in which they are stored in this package at the present time is
$S_a,S_c,S_b,T_{b,c}, T_{a,b}, T_{a,c},V$, $U_{b,a},U_{a,c},U_{c,b},U_{c,a},
U_{a,b}, U_{b,c}$,
$V_{c,a,b}, V_{b,c,a}, V_{a,b,c},V_{b,a,c},V_{c,b,a},
V_{a,c,b}, W_a, W_c, W_b, X , X'.
 $

For a printed version of these models, we refer the reader to the tables of \cite{LG}.

A consequence of the trace conjecture of \cite{BMM}
for the complex reflection group $G_{25}$ would be that
there exists a symmetrizing trace on this algebra with
a unicity condition enabling one to compute the corresponding Schur elements. Under this conjecture, these
Schur elements have been computed (see \cite{MALLE,MARIATH}). Of interest
for us will be the following fact : none of these Schur
elements vanish in $R/(a^3+b^2c)$ (while some do inside
$R/(a^3-b^2c)$, though !). Therefore this conjecture implies that $H_4 \otimes L$ is semisimple for $L$ the fraction field
of $R/(a^3-b^2c)$. We check this as follows. In CHEVIE there
are matrix models of all irreducible representations of $H_4$. We check on these by direct (computer) computation
that for the specialization $a = -4$, $b =8$, $c = 1$
all these representations remain irreducible, and pairwise non
isomorphic.

\subsection{A first definition of $\GQ_n$, over $K$}

Our quotient $\GQ_n$ can be defined over $K$ as follows.
Consider the ideal of $H_n \otimes K$ generated by
the (images inside $H_n \otimes K$) of the ideal of
$H_3$ associated to the irreducible representation $U_{b,c}$. Then this version of $\GQ_n$ over $K$, that we
denote $\GQ_n \otimes K$ for compatibility reasons with the sequel, is the quotient of
$H_n \otimes K$ by this ideal.

From the description of $H_4\otimes K$ given here, it
follows that $\GQ_4\otimes K$ is semisimple, and that
its irreducible representations are (see \cite{LG})
the $S_x$ for $x \in \{a,b,c\}$ and $T_{a,b}$,
$T_{a,c}$,$V$, $U_{a,b}$,$U_{b,a}$,$U_{a,c}$, $U_{c,a}$,
$V_{a,b,c}$, $V_{b,a,c}$, $V_{a,c,b}$, $V_{c,a,b}$, $W_a$.

The above definition can already be slightly generalized to
the case where $L$ is the field and we consider a specialization $R \to L$ where the image of
$(x-y)(x^2-xy+y^2)(xy+z^2)$ is nonzero for any
$\{x,y,z \} = \{a,b,c \}$. In this case $\GQ_4 \otimes L$
is again the quotient of the semisimple algebra
$H_4 \otimes L$ by the ideal corresponding to $U_{b,c}$.

\subsection{Quantum cubic quotients}

Let $\g$ denote a (finite dimensional) semisimple Lie
algebra over the complex numbers, endowed with a $\g$-invariant non-degenerate bilinear form $<\cdot,\cdot>$
(usually the Killing form). We 
fix an arbitrary basis $e_1,\dots,e_m$ of $\g$, denote
$e^1,\dots,e^m$ its dual basis with respect to the
given form. Let $C = \sum_{i=1}^m e_me^m \in Z(\mathsf{U}\g)$. If our form is the Killing form,
then $C$ is the Casimir operator.

Let us fix an finite dimensional $\g$-module $U$, and
denote $\tau \in \End(U\otimes U)$ denote the
action of $\sum_{i=1}^m e_i \otimes e_i$ on $U \otimes U$.
It commutes with the flip $x \otimes y \mapsto y\otimes x$. Defining $\tau_{ij} \in \End(U^{\otimes n})$ for $1\leq i \neq j \leq n$
as in \cite{KASSEL} we get a linear representation $\mathfrak{B}_n \to \End(U^{\otimes n})$, induced by $t_{ij} \mapsto \tau_{ij}$ and extending the natural action of $\mathfrak{S}_n$ on $U^{\otimes n}$. This
action commutes with the action of the envelopping algebra $\mathsf{U}\g$. By the Drinfeld-Kohno theorem, for generic values of $q$, this $\mathfrak{B}_n \otimes \mathsf{U}\g$-module structure provides through the monodromy of the associated KZ-system
$$
\frac{h}{\ii \pi} \sum_{1 \leq i < j \leq n} \tau_{ij} \dd \log(z_i - z_j)
$$
the
$\C B_n \otimes \mathsf{U}_q \g$-action on
the quantized module also denoted $U^{\otimes n}$ (see e.g. \cite{KASSEL}), where $q = e^h$. 

\subsubsection{Action on $U^{\otimes 2}$}

In particular the action of the braid generator $s_1$ is conjugate (for generic values of $h$)
to $(1 \ 2)\exp(h \tau)$.  Since 
$$
2 \sum_{i=1}^m e_i \otimes e^i = \Delta(C) - C \otimes 1 - 1 \otimes C
$$
where $\Delta : \Ug \to \Ug \otimes \Ug$ is the coproduct, we know that $\tau$ acts
by a scalar on any simple component of $U^{\otimes 2}$.
Also recall that $\tau$ commutes with $(1 \ 2)$ and
that the value of the Casimir element $C \in Z(\mathsf{U}\g)$
on $V(\la)$ is equal to $<\la,\la + 2 \rho>$,
where $\rho$ is equal to the half-sum of the positive roots (\cite{FH}, (25.14)).

\subsubsection{Commutant algebra}

We set $\mathcal{C}_n = \End_{\Ug}(U^{\otimes n})$
and assume $U^{\otimes n}$ is semisimple as a $\Ug$-module for $n$ smaller
than some $n_{\infty} \in \N$. Let $P_+$ denote the
lattice of dominant weights for $\g$, and $V(\la)$
the (irreducible) highest weight module associated to it. We set $P_+(n)$ the set of all $\la \in P_+$
such that $U^{\otimes n}$ contains an irreducible component isomorphic to $V(\la)$. As a $\mathcal{C}_n \otimes \Ug$-module there is a canonical multiplicity-free
decomposition of $U^{\otimes n}$ of the form
$$
U^{\otimes n} = \bigoplus_{\la \in P_+(n)} \hat{M}_n(\la), \ \ \ \ \ \ \hat{M}_n(\la) \simeq V_n(\la) \otimes M_n(\la)
$$
with $M_n(\la) = \Hom_{\Ug}(V_n(\la),U^{\otimes n})
=\Hom_{\Ug}(V_n(\la),\hat{M}_n(\la))$ and
the isomorphism $V_n(\la) \otimes M_n(\la) \to
\hat{M}_n(\la)$ being given by the evaluation map.
Since, as a $\Ug$-module, $U^{\otimes n}$
is the direct sum of the $\hat{M}_{n-1}(\mu) \otimes U$
and this decomposition is stable under $\mathfrak{B}_{n-1}$, we get a canonical decomposition
$$
\Res_{\mathcal{C}_{n-1}} M_n(\la) = \bigoplus_{\mu \in P_{n-1}(\la)} \Hom_{\Ug}(V(\la),\hat{M}_{n-1}(\mu)\otimes U)= \bigoplus_{\mu \in P_{n-1}(\la)}
c(V(\la),V(\mu)\otimes U)  M_{n-1}(\mu)
$$
where $c(V(\la),V(\mu)\otimes U)$ denotes the multiplicity of the simple $\Ug$-module $V(\la)$
inside $V(\mu) \otimes U$.

When every such restriction is multiplicity free, up to the
restriction to $\mathfrak{B}_2$, we get a canonical
decomposition of $M_n(\la)$ as direct sum of lines. A basis obtained by chosing one nonzero vector per line
will be called a \emph{suitable basis}. It is naturally
indexed by paths in the following levelled graph.
If $U = V(\la_0)$, then there is only one
vertex of level $0$, corresponding to $\la_0$. In general,
the vertices of level $k$ are 1-1 correspondence with
the $\mu \in P_+(k)$, and there is an edge between
the level $k$ vertex attached to $\la \in P_+(k)
$ and $\mu \in P_+(k+1)$ if and only if $V(\la) \otimes U$ contains an irreducible component isomorphic to $V(\mu)$ (and there will be only one such component under our multplity free assumption, for the $\la,\mu$ we will consider). When needed, we denote $\mu^{(k)}$ the vertex of level $k$ attached to $\mu \in P_+(k)$.

The indexing paths are the following ones. Consider paths from $\la_0^{(1)}$ to $\la^{(n)}$ (always
passing from one level to the next). Every such path
$\la_0^{(1)} = \mu_1^{(1)} \to\mu_2^{(2)} \to \dots
\to \mu_n^{(n)} = \la^{(n)}$ 
is canonically associated to the only line corresponding
to the inclusions $M_1(\mu_1) \subset M_2(\mu_2) \subset
\dots \subset M_n(\mu_n) = M_n(\la)$ in the direct sum
decomposition above.

\subsubsection{Action of infinitesimal braids}

From now on we assume that the multiplicity free assumption is satisfied by $M_n(\la)$, and that
we have pick a suitable basis.

There is a natural morphism $\mathfrak{B}_n \to \mathcal{C}_n$, and we want to know whether the
restriction to $\mathfrak{B}_n$ of the modules
$M_n(\la)$ is irreducible. We introduce
the elements $Y_r = \sum_{i<r} t_{ir} \in \mathfrak{B}_r$. They
commute to $\mathfrak{B}_{r-1} \subset \mathfrak{B}_r$
and in particular they commute to each other. Under
our assumption it is readily checked that they act
diagonally on our suitable basis. The scalar by which 
$Y_r$ acts on the (basis element indexed by the) path
$\mu_1^{(1)} \to\mu_2^{(2)} \to \dots
\to \mu_n^{(n)}$ is equal to $\frac{1}{2}( C(\mu_r) - C(\mu_{r-1}) - C(\la_0))$ (see e.g. \cite{THESE} ch. 4 or \cite{LG} lemma 2.1). Setting $s_r = (r,r+1)$ we get
$s_r Y_r s_r = Y_{r+1} - t_{r,r+1}$
hence $s_r W_r s_r + W_r = 2 t_{r,r+1}$ where
$W_r = Y_{r+1} - Y_r$. Notice that $W_r, s_r$ and $t_{r,r+1}$ all commute with $\mathfrak{B}_{r-1}$
and that their action on a given basis element only
depends on the section of the path given by
$\mu_{(r-1)}^{(r-1)}\to \mu_{(r)}^{(r)} \to 
\mu_{(r+1)}^{(r+1)}$. We call \emph{brick} between
$\mu_{(r-1)}^{(r-1)}$ and $ 
\mu_{(r+1)}^{(r+1)}$ the vector space indexed
by the paths of the form
$\mu_{(r-1)}^{(r-1)}\to \mu_{(r)}^{(r)} \to 
\mu_{(r+1)}^{(r+1)}$, endowed with the $\mathcal{A}$-module structure $W \mapsto W_r$, $s \mapsto s_r$, $u \mapsto t_{r,r+1}$, where $\mathcal{A}$ is the
unital algebra defined by generators $W,s,u$ and relations
$$
W + sWs = 2u, su=us, s^2 = 1
$$

\begin{lemma} \label{lem:brickirred33}Let $E$ be a $\mathcal{A}$-module
on which $s$ or $-s$ acts as a reflection, and on which
$W,u$ act diagonally with a disjoint set of distinct eigenvalues.
Then $E$ is irreducible.
\end{lemma}
\begin{proof}
Since $\mathcal{A}$ admits an algebra automorphism exchanging $s \leftrightarrow -s$, we can assume $s$ acts
by a reflection.
We choose a basis $e_1,\dots,e_{m+1}$ on which $s,u$ both act diagonally,
$s = \diag(-1,1,\dots,1)$, $u = \diag(\alpha,\beta_1,\dots,\beta_{m})$. Then the equation $W + sWs = 2u$ implies that $W = u + F$ with $F$ mapping $Vect(e_2,\dots,e_{m+1})$ to $Vect(e_1)$ and $Vect(e_1)$ to
$Vect(e_2,\dots,e_{m+1})$. Since $u$ has distinct eigenvalues, $E$ is irreducible unless one of the lines
$Vect(e_i)$ is stable. But this is possible only if $F.e_i = 0$, and then $Vect(e_i)$ would be a common eigenspace for $u,W$ and for the same eigenvalue for both $u$ and $W$. This is excluded by assumption, hence $E$ is irreducible.
\end{proof}

In dimension 2, we have the following much stronger form.

\begin{lemma} \label{lem:brickirred22} Let $E$ be a 2-dimensional $\mathcal{A}$-module on which $s$ or $-s$ acts as a reflection, and on which 
$W,u$ act diagonally. If $u$ has two distinct eigenvalues and $Sp(u) \neq Sp(W)$ then
$E$ is irreducible.
\end{lemma}
\begin{proof}
The proof starts in the same way as in the previous lemma, but then $W + sWs = 2u$ implies that, either $W$ and $u$ have the same spectrum, or none of the two eigenspaces of $u$ is stable under $W$. This implies that $E$ is irreducible.
\end{proof}

Assume that the $M_{n-1}(\mu)$ are all pairwise non-isomorphic irreducible
$\mathfrak{B}_{n-1}$-modules
for all the $\mu \in P_+(n-1)$ for which $V(\la)$ is
an irreducible component of $V(\mu) \otimes U$. Then,
the \emph{irreducibility graph} is the unoriented
graph whose vertices are all such $\mu$, and there
is an edge between $\mu_1$ and $\mu_2$ if there
is an irreducible brick between some $\nu^{(n-2)}$ and $ 
\la^{(n)}$ and paths 
$\nu^{(n-2)} \to \mu_1^{(n-1)}\to 
\la^{(n)}$
 and $\nu^{(n-2)} \to \mu_2^{(n-1)}\to 
\la^{(n)}$. It is easily proved (see e.g. \cite{THESE}, proposition 17, p. 51)
that, if the irreducibility graph is connected, then
$M_n(\la)$ is irreducible.

\subsection{Quantum cubic quotients : the $\sl(V)$ modules $\Lambda^2 V$ and $S^2 V$}
\label{sect:quantumLS}

Let $\g = \sl_n(\C)$ for $n \geq 7$, $V = \C^n$
and $U_- = \Lambda^2 V$. We denote $E_{ij}$ the elementary
matrix containing $1$ in row $i$ and column $j$ and $0$ otherwise. We consider the standard generators $X_i = E_{i,i+1}$, $Y_i = E_{i+1,i}$, $H_i = E_{ii}-E_{i+1,i+1}$. A highest weight vector $v$ is characterised by $X_i.v = 0$ for all $i$.

We use the notation $V(\alpha)$ for the highest weight $\g$-module of highest weight $\alpha$. Denoting $\varpi_1,\dots,\varpi_{n-1}$ the fundamental weight of $\g$ we have
$U_- = V(\varpi_2)$. Let $(e_1,\dots,e_n)$ denote
the standard basis of $\C^n = V$. We have $X_i.e_j = \delta_{i+1,j}e_{j-1}$
hence
$e_1$ is a highest weight vector for $V$. Then a h.w. vector for $U_- = \Lambda^2 V$ is $e_1 \wedge e_2$.

A related module if $U_+ = S^2 V$. All the results below concerning the action
of $B_k$ or $\mathfrak{B}_k$ are equally valid with $U_-$ replaced by $U_+$, after exchanging $n$ with $-n$. This
is because these two modules can be considered
as modules under the Lie superalgebra $\sl_n = \sl(n|0)
= \sl(0|n)$ of the form $S^2 E$ where $E$ is a
super vector space of type $n|0$ for $U_+$
or $0|n$ for $U_-$, and of superdimension $n$ or $-n$.
We provide the details in the case of $U_-$, the
corresponding details for $U_+$ being left to
the reader.

\subsubsection{Action of $B_k$, $k \leq 3$}
From the fairly easy multiplication rule by $U_- = V(\varpi_2)$ (see \cite{FH} proposition 15.25) one gets that the Bratteli
diagram of the tensor powers $U_-,U_-^{\otimes 2},U_-^{\otimes 3}$ is as follows.

$$
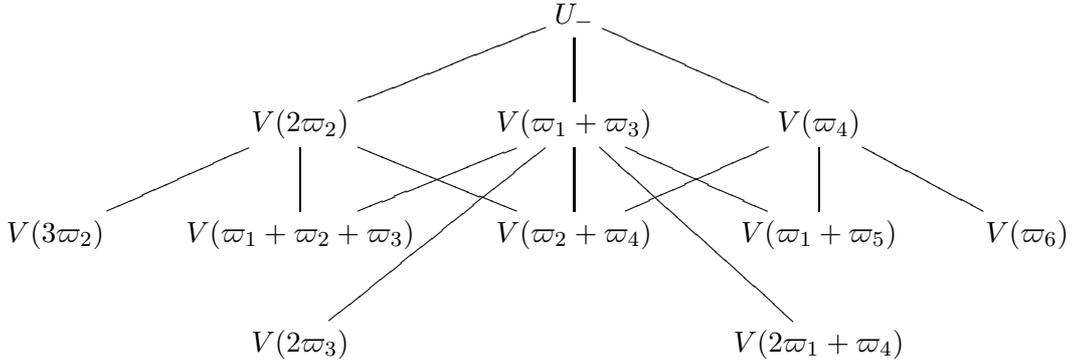
\xymatrix{
 & & U_- \ar@{-}[dl]\ar@{-}[d]\ar@{-}[dr] \\
& V(2 \varpi_2) \ar@{-}[dl]\ar@{-}[d] \ar@{-}[dr] &
V(\varpi_1+\varpi_3)\ar@{-}[ddl] \ar@{-}[ddr] \ar@{-}[dl] \ar@{-}[d]\ar@{-}[dr]
& V(\varpi_4) \ar@{-}[dl] \ar@{-}[d] \ar@{-}[dr] \\
V(3 \varpi_2) & V(\varpi_1+\varpi_2+\varpi_3) & V(\varpi_2+\varpi_4) & V(\varpi_1+\varpi_5) & V(\varpi_6) \\
 & V(2 \varpi_3) & & V(2 \varpi_1+\varpi_4) \\
} 
$$
We have $S^2U_- = V(2\varpi_2) + V(\varpi_4)$ (\cite{FH}, ex. 15.34) hence $\Lambda^2 U_- = V(\varpi_1+\varpi_3)$. Also,
$
\Lambda^3 U_- = V(2 \varpi_1 + \varpi_4)+V(2\varpi_3)$ and
$S^3 U_- = V(3 \varpi_2)+ V(\varpi_2+\varpi_4) + V(\varpi_6)
$.

The value of the Casimir element $C \in Z(\mathsf{U}\g)$
on $V(\la)$ is equal to $(\la,\la + 2 \rho)$,
where $\rho$ is equal to the half-sum of the positive roots (\cite{FH}, (25.14)).
We refer to Bourbaki (\cite{LIE456}, planche 1) for the basic datas involved in this case. Straightforward
calculations show $(\varpi_i,2 \rho) = i(n-i)$
and $(\varpi_i, \varpi_j) = \min(i,j) - ij/n$. From this the value of the Casimir is readily computed.

The value of $\tau$ on each of the irreducible components of $U_-^{\otimes 2}$ is then :
$$
 V(\varpi_1+\varpi_3): \frac{-4}{n}  \ \ \ \ \ \ V(2\varpi_2):\frac{2(n-2)}{n} \ \ \ \ \ \ V(\varpi_4):  \frac{-4(n+1)}{n}
$$
Therefore, the action of the braid group factorises through the cubic Hecke algebras, with 
$$
a = - \exp\left(\frac{-4}{n} h\right),
b =  \exp\left(\frac{2(n-2)}{n} h\right),
c =  \exp\left(\frac{-4(n+1)}{n} h\right),
$$
We have $a^3 + b^2c = 0$, so in this section we replace
$K$ by the fraction field $K_{\Lambda}$ of $R_{\Lambda}=R/(a^3 + b^2c)$.
Note that, inside $K_{\Lambda}$, $ (x-y)(x^2-xy+y^2)(xy+z^2) \neq 0
$ for $\{x,y,z \} = \{a,b,c \}$, so $\GQ_4\otimes K_{\Lambda}$ is well-defined, and semisimple by section \ref{sect:cubicHecke4}. Moreover,
$
R/(a^3 + b^2c) = \Z[a^{\pm 1},b^{\pm 1},c^{\pm 1}]/(a^3 + b^2c)$ is equal to
$$
\Z\left[a^{\pm 1},\left(\frac{b}{-a}\right)^{\pm 1},\left(\frac{c}{-a}\right)^{\pm 1}\right]/\left(1 - \left(\left(\frac{b}{a}\right)^2\left(\frac{c}{-a}\right)\right)\right)
\simeq \Z\left[a^{\pm 1},b^{\pm 1},c^{\pm 1}\right]/\left(1 - b^2c\right)
$$ which is isomorphic to 
$
(\Z[a^{\pm 1},b^{\pm 1}])[c^{\pm 1}]/\left(c- b^{-2}\right) \simeq \Z[a^{\pm 1},b^{\pm 1}]
$.

In order to check that the specialization
we are considering is generic inside $R_{\Lambda} = R/(a^3+b^2c)$, we need to check that the discriminant of
the corresponding algebra is nonzero. Since it is
homogeneous we can renormalize $s_i \mapsto -a^{-1} s_ i$, and assume $a = -1$, $b = e^{2h}$, $c = e^{4h}$.
Then, under the above isomorphism, this discriminant becomes a Laurent polynomial in $b$ specialized at $b = e^{2h}$,
which is nonzero for generic $h$ since the Laurent polynomial itself is nonzero by section \ref{sect:cubicHecke4}. Therefore, we will loosely work in the sequel as though we were working inside $\GQ_4 \otimes K_{\Lambda}$.

Similarly, the value of $Y_3$ can be computed easily,
and one checks that the irreducibility criteria
above are satisfied (for $n$ large enough), so that all the $M_3(\la)$
are irreducible as $\mathfrak{B}_3$-modules.
They are obviously pairwise non-isomorphic, except for the 1-dimensional representations. Actually, a 1-dimensional representation of $\mathfrak{B}_n$ has
the form $t_{ij} \mapsto \alpha$, $(i \ j) \mapsto \eps$ for some $\eps = \pm 1$ and $\alpha \in \C$. But since
here $(1 \ 2)$ acts as a polynomial of $t_{12}$ the
value of $\alpha$ determines $\eps$. And there
can be at most 4 values of $\alpha$, which are the
values of $\tau$ on $U_-^{\otimes 2}$. This
shows a priori that there cannot be 4 non-isomorphic 1-dimensional components for the action of $\mathfrak{B}_3$. Actually, the ones which are isomorphic
are the ones whose restriction to $\mathfrak{B}_2$
are isomorphic, and these are $M_3(2 \varpi_1+\varpi_4)$
and $M_3(2 \varpi_3)$. Another reason for this
is that $\Lambda^3 U_- = V(2 \varpi_1 + \varpi_4)+V(2\varpi_3)$ and that, for all $x \in U_-^{\otimes n}$,
we have $(i \ j).x = -x \Rightarrow t_{ij}.x = (-4/n)x$. Indeed, this property is true for $t_{12} = \tau$ on $U^{\otimes 2}$, therefore also on $U^{\otimes n}$ and through $\mathfrak{S}_n$-conjuguation one gets it holds for arbitrary $i,j$.

From this one gets the following decomposition of the
quantized module $U_-^{\otimes 3}$,
$$
\begin{array}{lcl}
U_-^{\otimes 3} &=& V(3 \varpi_2) \otimes S_ b  +V(\varpi_1+\varpi_2+\varpi_3) \otimes U_{a,b}
+V(\varpi_2+\varpi_4) \otimes V
+V(\varpi_1+\varpi_5) \otimes U_{a,c}\\ 
& & +V(\varpi_6) \otimes S_c
+(V(2\varpi_3) 
+V(2\varpi_1+\varpi_4)) \otimes S_ a
\end{array}
$$
This implies the following.
\begin{proposition} The action of $B_k$ on $U_-^{\otimes n}$ factorizes through $\GQ_k \otimes K$. Moreover, the
morphism $\GQ_3\otimes K \to \End(U_-^{\otimes 3})$
is injective and not surjective.
\end{proposition}
\subsubsection{The bimodule $U_-^{\otimes 4}$}
{\ } \\

The goal of this section is to prove the following.

\begin{proposition} For generic values of $q$, the
$(\mathsf{U}_q \g,B_4)$-bimodule $U_-^{\otimes 4}$ admits the following decomposition
$$
\begin{array}{lcl}
U_-^{\otimes 4} &=&  V(4\varpi_2)\otimes S_b +
V(\varpi_1+2\varpi_2+\varpi_3)\otimes  U_{b,a}  + V(2\varpi_1+2\varpi_3)\otimes T_{a,b} \\
& & + V(2\varpi_2+\varpi_4)\otimes  V_{b,a,c}  +V(\varpi_1+\varpi_3+\varpi_4)\otimes ( S_a + V_{a,b,c} )+V(2\varpi_4)\otimes V \\  
 & & +V(2\varpi_1+\varpi_2+\varpi_4)\otimes  U_{a,b} +V(\varpi_2+2\varpi_3)\otimes  U_{a,b} \\ & & 
+V(\varpi_1+\varpi_2+\varpi_5)\otimes  W_a + V(\varpi_3+\varpi_5)\otimes  V_{a,c,b}    +V(\varpi_2+\varpi_6)\otimes  V_{c,a,b} \\  
& & +V(3\varpi_1+\varpi_5)\otimes S_a 
+V(2\varpi_1+\varpi_6)\otimes  U_{a,c}  +V(\varpi_1+\varpi_7)\otimes U_{c,a}  +V(\varpi_8)\otimes S_c\\
\end{array}
$$
\end{proposition}
\begin{corollary}
The image of $\GQ_4 \otimes K$ inside $\End(U_-^{\otimes 4 })$ has dimension 260, and its irreducible
representations are $\mathrm{Irr}(\GQ_4 \otimes K) \setminus \{ T_{a,c}\}$.
\end{corollary}
In order to prove this decomposition, we need
to determine the restriction to $\mathfrak{B}_4$ of
the $\mathcal{C}_4$-modules $M_4(\la)$ for
$\la \in P_+(4)$. For this we will, in most cases,
first prove irreducibility, and then determine the
isomorphism type by looking at the restriction on 3 strands. Notice that all the representations of $\mathfrak{B}_4$ we are considering here are semisimple,
as they provide monodromy representations of $H_4 \otimes K_{\Lambda}$
and $H_4\otimes K_{\Lambda}$ is semisimple (see \cite{ASSOC} on a detail account on these monodromy properties). 

From the Littlewood-Richardson rule (see e.g. \cite{FH}, (15.23) and (A.8)) we
get the decompositions of the $V(\la) \otimes V(\varpi_2)$
for $\la \in P_3(U_-)$ (see table \ref{tab:timesVpi2}).

\begin{table}
$$
\begin{array}{|rcl|}
\hline
V(3 \varpi_2)  \otimes V(\varpi_2) & = & V(4\varpi_2) + V(\varpi_1+2\varpi_2+\varpi_3) + V(2 \varpi_2 + \varpi_4)\\
 V(\varpi_1+\varpi_2+\varpi_3) \otimes V(\varpi_2) & = & V(\varpi_1+2 \varpi_2 + \varpi_3)+V(2 \varpi_1 + 2\varpi_3)+V(2 \varpi_1 + \varpi_2+\varpi_4)\\
 & & +V( \varpi_2 + 2\varpi_3)+V(2 \varpi_2 + \varpi_4)
 +V(\varpi_1+ \varpi_3 + \varpi_4)\\ & & +V(\varpi_1+ \varpi_2 + \varpi_5) \\
 V(\varpi_2+\varpi_4) \otimes V(\varpi_2) & = & V(2 \varpi_2 + \varpi_4)+V( \varpi_1+\varpi_3 + \varpi_4)
 +V(\varpi_1+ \varpi_2 + \varpi_5)\\
 & & +V(2 \varpi_4)+V(\varpi_3 + \varpi_5)+V( \varpi_2 + \varpi_6)\\
 V(\varpi_1+\varpi_5) \otimes V(\varpi_2) & = &
 V(\varpi_1+ \varpi_2 + \varpi_5)+V(\varpi_3 + \varpi_5) +V(2 \varpi_1 + \varpi_6)\\ & & +V(\varpi_2 + \varpi_6)+V( \varpi_1 + \varpi_7)\\
 V(\varpi_6) \otimes V(\varpi_2) & = & V( \varpi_2 + \varpi_6)+V(\varpi_1 + \varpi_7)+V(\varpi_8) \\
 V(2 \varpi_3) \otimes V(\varpi_2) & = & V( \varpi_2 + 2\varpi_3)+V(\varpi_1+ \varpi_3 + \varpi_4)+V( \varpi_3 + \varpi_5)\\
  V(2 \varpi_1+\varpi_4) \otimes V(\varpi_2) & = & 
  V(2 \varpi_1+\varpi_2 + \varpi_4)+V(\varpi_1 +
  \varpi_3+ \varpi_4)+V(3 \varpi_1 + \varpi_5)\\ & & 
  +V(\varpi_1+ \varpi_2 + \varpi_5)+V(2 \varpi_1 + \varpi_6)\\
  \hline
\end{array}
$$
\caption{Decomposition of $V(\la) \otimes V(\varpi_2)$}
\label{tab:timesVpi2}
\end{table}

From these we readily get that $M_4(4 \varpi_2) \simeq S_b$, $\Res M_4(2 \varpi_1+2\varpi_3) = M_3(\varpi_1+\varpi_2+\varpi_3) \simeq U_{a,b}$
hence $M_4(2 \varpi_1+2\varpi_3)\simeq T_{a,b}$,
$\Res M_4(2\varpi_4) = M_3(\varpi_2 + \varpi_4) \simeq V$
hence $M_4(2\varpi_4) \simeq V$,
$\Res M_4(3 \varpi_1+\varpi_5) = M_3(2 \varpi_1 + \varpi_4) \simeq S_a$ hence $M_4(3 \varpi_1+\varpi_5)  \simeq S_a$,
$\Res M_4(\varpi_8) = M_3(\varpi_6) \simeq S_c$ hence $M_4(\varpi_8)  \simeq S_c$.

Again by using the Littlewood-Richardson rule we get
the decomposition of the $V(\la) \otimes V(\mu)$ for
$\la, \mu \in P_2(U_-)$ (see table \ref{tab:decU2laotmu}).
\begin{table}
$$
\begin{array}{|lcl|}
\hline
V(2 \varpi_2) \otimes V(2 \varpi_2) &=& V(4\varpi_2)+
V(\varpi_1+2\varpi_2+\varpi_3) + V(2\varpi_1+2\varpi_3)
+ V(2\varpi_2+\varpi_4)\\ & & +V(\varpi_1+\varpi_3+\varpi_4)+V(2\varpi_4) \\  
V(2 \varpi_2) \otimes V(\varpi_1+\varpi_3) &=& V(\varpi_1+2\varpi_2+\varpi_3)+V(2\varpi_1+\varpi_2+\varpi_4)+V(\varpi_2+2\varpi_3)\\ & & + V(2\varpi_2+\varpi_4)
+V(\varpi_1+\varpi_3+\varpi_4)+V(\varpi_1+\varpi_2+\varpi_5)\\ & & + V(\varpi_3+\varpi_5)  \\
V(2 \varpi_2) \otimes V(\varpi_4) &=& V(2\varpi_2+\varpi_4)+V(\varpi_1+\varpi_2+\varpi_5)+V(\varpi_2+\varpi_6)\\  
V(\varpi_1+\varpi_3) \otimes V(\varpi_1+\varpi_3) &=&
V(2\varpi_1+2\varpi_3)+V(2\varpi_1+\varpi_2+\varpi_4)
+V(\varpi_2+2\varpi_3)\\ & & +V(2\varpi_2+\varpi_4)+2V(\varpi_1+\varpi_3+\varpi_4)+V(3\varpi_1+\varpi_5)\\ & & +2V(\varpi_1+\varpi_2+\varpi_5)+V(2\varpi_4)
+V(\varpi_3+\varpi_5)+V(2\varpi_1+\varpi_6)\\ & & +
V(\varpi_2+\varpi_6) \\ 
V(\varpi_1+\varpi_3) \otimes V(\varpi_4) &=& V(\varpi_1+\varpi_3+\varpi_4) +V(\varpi_1+\varpi_2+\varpi_5)+V(\varpi_3+\varpi_5)\\ & & +V(2\varpi_1+\varpi_6)+V(\varpi_2+\varpi_6)+V(\varpi_1+\varpi_7)\\
V(\varpi_4) \otimes V(\varpi_4) &=& V(2\varpi_4)+V(\varpi_3+\varpi_5)+V(\varpi_2+\varpi_6)+V(\varpi_1+\varpi_7)+V(\varpi_8)\\
\hline 
\end{array}
$$
\caption{Decomposition of the $V(\la)\otimes V(\mu)$ for
$V(\la),V(\mu) \hookrightarrow (\Lambda^2 V)^{\otimes 2}$}
\label{tab:decU2laotmu}
\end{table}
These decompositions determine the
spectrum of $t_{34}$ (hence of $(3 \ 4)$) on any given brick.

We first deal with $2 \varpi_1+\varpi_2+\varpi_4 , \varpi_2+2\varpi_3 , 2 \varpi_1 +\varpi_6,\varpi_1+\varpi_7 , \varpi_1+2\varpi_2 +\varpi_3$. In all these
cases, the restriction to $\mathfrak{B}_3$ has 2 distinct irreducible components, and there are two bricks, one being 1-dimensional and the other being 2-dimensional. In order to prove irreducibility
we thus only need to prove irreducibility of the 2-dimensional brick under the action of $u$ and $W$, since
the irreducibility graph is then obviously connected.
By lemma \ref{lem:brickirred22} this only depends on the spectra of $u$ and $W$, provided $s$ is a reflection, which means here
that $\{ -4/n \} \subsetneq Sp(u)$. In table \ref{tab:LSirredU4part1}
we give the vertices of each brick as well as the
spectra, where $u' =  u$, $W' =  W$.
This proves irreducibility in all these cases.
Since the irreducible representations of $H_4$ are uniquely determined by their restriciton to $H_3$ and
since we know the restrictions of the $M_4(\la)$, we get
the conclusion for these weights.

\begin{table}
$$
\begin{array}{|c|c|c|c|c|}
\hline
U_-^{\otimes 2} & \varpi_1+\varpi_3 & \varpi_1+\varpi_3
& \varpi_1+\varpi_3
 \\
\hline
U_-^{\otimes 3} & \varpi_1+\varpi_2+\varpi_3;2\varpi_1+\varpi_4 & \varpi_1+\varpi_2+\varpi_3;2 \varpi_3 & 2 \varpi_1 + \varpi_4; \varpi_1+\varpi_5\\
\hline
U_-^{\otimes 4} &2 \varpi_1+\varpi_2+\varpi_4 & \varpi_2+2\varpi_3 & 2 \varpi_1 +\varpi_6 \\
\hline
Sp(u') &-4 ;2(n-2) &-4 ;2(n-2) & -4;-4(n+1)\\
\hline
Sp(W') & -2(n+2);4(n-1)& -2(n+2);4(n-1) & -4(2n+1) ; 4(n-1)\\
\hline
\hline
U_-^{\otimes 2} & \varpi_4 & 2 \varpi_2
&  \\
\hline
U_-^{\otimes 3} & \varpi_1+\varpi_5;\varpi_6 & 3\varpi_2;\varpi_1+\varpi_2+ \varpi_3 & \\
\hline
U_-^{\otimes 4} &\varpi_1+\varpi_7 & \varpi_1+2\varpi_2 +\varpi_3 &  \\
\hline
Sp(u') &-4 ;-4(n+1) &-4 ;2(n-2) & \\
\hline
Sp(W') & -4(2n+1);4(n-1)& -4(n+1);2(n-2) & \\
\hline
\end{array}
$$
\caption{Irreducibility of the 2-dimensional bricks.}
\label{tab:LSirredU4part1}
\end{table}

We now consider the weights $2 \varpi_2+\varpi_4$, $\varpi_3+\varpi_5$, $\varpi_2+\varpi_6$. In all the
cases the restriction of $M_ 4(\la)$ admits 3 pairwise non-isomorphic irreducible components, and there is 3-dimensional brick. Therefore, it remains to prove that
$s$ acts as a reflection and that $Sp(u') \cap Sp(W') = \emptyset$. From table \ref{tab:decU2laotmu} we get
that the spectrum of $u$ has 3 elements, equal to the
3 eigenvalues of $t_{12}$. In particular, $s$ acts as a reflection, and it thus sufficient to compute
the spectrum of $W$ to prove irreducibility, by lemma \ref{lem:brickirred33}. This is
done in table \ref{tab:LSirredU4part2}. Then, by considering the restriction, we similarly get
the conclusion for these weights.

\begin{table}
$$
\begin{array}{|c|c|c|c|c|}
\hline
U_-^{\otimes 2} & 2 \varpi_2 & \varpi_1+\varpi_3 
 \\
\hline
U_-^{\otimes 3} & 3 \varpi_2; \varpi_1+\varpi_2+\varpi_3;\varpi_2+\varpi_4 & 2 \varpi_3;\varpi_2+\varpi_4;\varpi_1+\varpi_5 \\
\hline
U_-^{\otimes 4} &2 \varpi_2+\varpi_4 & \varpi_3+\varpi_5 \\
\hline
Sp(W') & -4(2n+1);-2(n+2);4(2n-1) & -2(3n+2);-2(n+2) ; 2(3n-2)\\
\hline
\hline
U_-^{\otimes 2}  
& \varpi_4 & 
 \\
\hline
U_-^{\otimes 3}  & \varpi_2+\varpi_4; \varpi_1 + \varpi_5; \varpi_6 & \\
\hline
U_-^{\otimes 4} &  \varpi_2 +\varpi_6 & \\
\hline
Sp(W') & -2(5n+2) ; -2(n+2) ; 2(5n-2)& \\
\hline
\end{array}
$$
\caption{Irreducibility of some 3-dimensional bricks.}
\label{tab:LSirredU4part2}

\end{table}

\begin{table}
$$
\begin{array}{|lclcl|}
\hline
S^4 U_- &=& F_{[4]}(U_-) &=& V(4\varpi_2) + V(2\varpi_2+\varpi_4) + V(2 \varpi_4) + V(\varpi_2+\varpi_6) + V(\varpi_8) \\
& &F_{[3,1]}(U_-) &=& V(\varpi_1+2\varpi_2+\varpi_3)+V(2\varpi_2+\varpi_4)+V(\varpi_1+\varpi_3+\varpi_4)\\ 
& & &+ &  V(\varpi_1+\varpi_2+\varpi_5)+V(\varpi_3+\varpi_5)+V(\varpi_2+\varpi_6)+V(\varpi_1+\varpi_7)\\  
& &F_{[2,2]} &=&
V(2\varpi_1+2\varpi_3)+V(2\varpi_2+\varpi_4)
+V(\varpi_1+\varpi_2+\varpi_5)\\
& & &+& V(2\varpi_4) +V(\varpi_2+\varpi_6)\\ 
& &F_{[2,1,1]}(U_-) &= & V(2\varpi_1+\varpi_2+\varpi_4)+V(\varpi_2+2\varpi_3)+V(\varpi_1+\varpi_3+\varpi_4)\\
& & &+& V(\varpi_1+\varpi_2+\varpi_5)+V(\varpi_3+\varpi_5)+V(2\varpi_1+\varpi_6)\\
\Lambda^4 U_- &=& F_{[1,1,1,1]}(U_-) &=& V(\varpi_1+\varpi_3 + \varpi_4) + V(3\varpi_1 + \varpi_5) \\
\hline 
\end{array}
$$
\caption{Plethysm of $V(\varpi_2)^{\otimes 4}$ for
$\sl_9(\C)$.}
\label{tab:plethysm}
\end{table}

\paragraph{{\bf $M_4(\varpi_1+\varpi_2+\varpi_5)$ is irreducible } }

$$
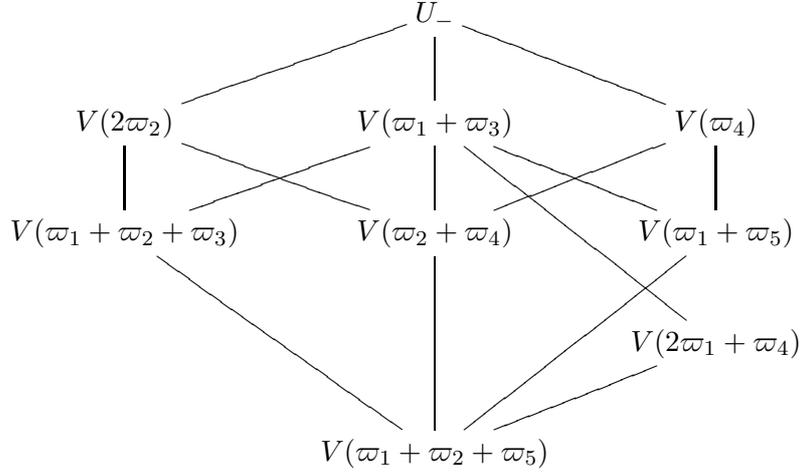
\xymatrix{
 & & U_- \ar@{-}[dl]\ar@{-}[d]\ar@{-}[dr] \\
& V(2 \varpi_2) \ar@{-}[d] \ar@{-}[dr] &
V(\varpi_1+\varpi_3) \ar@{-}[ddr] \ar@{-}[dl] \ar@{-}[d]\ar@{-}[dr]
& V(\varpi_4) \ar@{-}[dl] \ar@{-}[d]  \\
 & V(\varpi_1+\varpi_2+\varpi_3) \ar@{-}[ddr] & V(\varpi_2+\varpi_4) \ar@{-}[dd]& V(\varpi_1+\varpi_5) \ar@{-}[ddl]&  \\
 &  & & V(2 \varpi_1+\varpi_4) \ar@{-}[dl] \\
 & & V(\varpi_1+\varpi_2+\varpi_5) \\
} 
$$
We consider the two 2-dimensional bricks. For the one
based at $V(2 \varpi_2) \subset U_-^{\otimes 2}$,
$t_{34}$ has eigenvalues $-4/n$ and 
$\frac{-4(n+1)}{n}$.
For the one
based at $V(\varpi_4) \subset U_-^{\otimes 2}$,
$t_{34}$ has eigenvalues $-4/n$ and 
$\frac{2(n-2)}{n}$. Hence on both bricks
$s$ acts as a reflection. On the first one the
eigenvalues of $W$ are $-(7n+4)/n$ and $(3n-4)/n$, while on the second one the eigenvalues are $-(3n+4)/n$ and $(5n-4)/n$. In the irreducibility graph, we thus see that $\varpi_1+\varpi_2+\varpi_3$,
$\varpi_2 + \varpi_4$ and $\varpi_1 + \varpi_5$
belong to the same connected component. It follows
that either $M_4(\varpi_1+\varpi_2+\varpi_3)$
is irreducible, or it is the direct sum of two
irreducible representations, one 1-dimensional
and one 7-dimensional. Since there is no 7-dimensional
irreducible representation of $H_4$ the conclusion follows.
The
restriction to $B_3$ is isomorphic to $S_a + U_{a,c}+U_{a,b} + V$
hence $M_4(\varpi_1+\varpi_2+\varpi_5)\simeq W_a$.

\paragraph{{\bf The case  $M_4(\varpi_1+\varpi_3+\varpi_4)$ } }

$$
\xymatrix{
 & & U_- \ar@{-}[dl]\ar@{-}[d]\ar@{-}[dr] \\
& V(2 \varpi_2) \ar@{-}[d] \ar@{-}[dr] &
V(\varpi_1+\varpi_3)\ar@{-}[ddl] \ar@{-}[ddr] \ar@{-}[dl] \ar@{-}[d]
& V(\varpi_4) \ar@{-}[dl]   \\
 & V(\varpi_1+\varpi_2+\varpi_3)\ar@{-}[ddr] & V(\varpi_2+\varpi_4) \ar@{-}[dd]&  &  \\
 & V(2 \varpi_3) \ar@{-}[dr] & & V(2 \varpi_1+\varpi_4)\ar@{-}[dl]  \\
 & & V(\varpi_1 + \varpi_3 + \varpi_4) & } 
$$
Here the restriction to $\mathfrak{B}_3$ is \emph{not} multiplicity free, hence our criterion cannot be
applied directly. Consider however the 2-dimensional brick based at $V(2 \varpi_2)$. The eigenvalues of
$t_{34}$ are $-4/n$ and $2(n-2)/n$ (hence $s$ acts
by a reflection). The eigenvalues of $W$ are $-4(n+1)/n$ and $2(3n-2)/n$, hence the brick is irreducible by our criterion. Since the action of $\mathfrak{B}_4$ is semisimple, it easily follows that our
7-dimensional representation is a sum of an irreducible component of dimension at least $5$, and of
another component on which $t_{12}$ has for only
eigenvalue $-4/n$. Since there are no 5-dimensional
or 7-dimensional irreducible representation for $H_4$, the only possibility is that we have the sum of a 6-dimensional irreducible representations and a 1-dimensional one.
The
restriction to $B_3$ is isomorphic to $2S_a + U_{a,b}+V$
hence $M_4(\varpi_1+\varpi_3+\varpi_4)\simeq S_a + V_{a,b,c}$.

\begin{remark} We are not able to follow the same method to
elucidate the bimodule structure on $U_-^{\otimes 5 }$
because some of the
putative Schur elements of the cubic Hecke algebra on 5 strands vanish inside $R_{\Lambda}$, suggesting that its
representation-theoretic behavior is not generic. As a matter of fact, it can be checked that the restriction to $B_4$ of $M_5(\varpi_3+\varpi_7)$ cannot be obtained by
restriction of a representation of the generic Hecke algebra, thus proving that the discriminant of $H_5$
admits $a^3+b^2c$ as a factor. This has for consequence that we are not able to find the dimension by easy representation-theoretic arguments based
on the generic cubic Hecke algebra.
The question of whether the action of $B_5$ is semisimple over $K_{\Lambda}$ remains open, though.
\end{remark}

\subsection{Quantum cubic quotients : the Links-Gould invariant}

In \cite{LG} we investigated the quotient of the braid algebra involved in the Links-Gould polynomial invariant. Recall that this invariant arises through the consideration of a $1$-parameter family of 4-dimensional representations of the Lie superalgebra $\sl(2|1)$.
This invariant is stronger than the Alexander polynomial, and yet shares a number of properties with it (for instance, it vanishes on split links).

We proved in \cite{LG} that the centralizer algebra $LG_n$ involved in this construction is a quotient of the cubic Hecke algebra $H_n$, and even of $\GQ_n$.
We defined a quotient (denoted $A_n$ in \cite{LG}) of $\GQ_n$, proper when $n\geq 4$, as the quotient of
$K B_n$ by the ideal generated by $\Ker(K B_4 \to LG_4)$. Let us denote it
by $LG'_n$. We conjectured $LG_n \simeq LG'_n$. The
dimensions for $n=2,3,4,5$ are $3,20,175,1764$,
and conjecturally $\dim LG_{n+1}=(2n)!(2n+1)!/(n!(n+1)!)^2$. The description of the defining ideal $\Ker(K B_4 \to LG_4)$
given in \cite{LG} was representation-theoretic at first. In particular we got that 
$$
\mathrm{Irr}(LG_4) = \mathrm{Irr}(GQ_4 \otimes K) \setminus \{ T_{a,b},T_{a,c},V,V_{a,b,c},V_{a,c,b} \}
$$
and from this we get that the quotient map $\GQ_4\otimes K \onto LG_4$ factorizes through the algebra $ \Lambda S_4 = \mathrm{Im}(KB_4 \to \End(U_-^{\otimes 4}))$ described above.

From this representation-theoretic description we got
in particular some remarkable properties that we recall here :
$$
\begin{array}{lcll}
s_1^{-1} (s_3^{-1} s_2 s_3^{-1}) &\equiv& (s_3^{-1} s_2 s_3^{-1}) s_1^{-1} & \mod LG_3 s_3 LG_3 + LG_3 s_3^{-1} LG_3 + LG_3 \\
s_1 (s_3^{-1} s_2 s_3^{-1}) &\equiv &(s_3^{-1} s_2 s_3^{-1}) s_1 & \mod LG_3 s_3 LG_3 + LG_3 s_3^{-1} LG_3 + LG_3
\end{array}
$$

\subsection{The tripled quadratic Hecke algebra}

In this section $R$ denotes an arbitrary domain with $x,y \in R^{\times}$. We let
$H_n(x,y)$ denote the (ordinary) Hecke algebra with
these parameters, namely the quotient of the group algebra $R B_n$ by the quadratic relations $(s_i - x)(s_i-y) = 0$ for $1 \leq i \leq n-1$, or equivalently by
the relation $(s_1 - x)(s_1-y) = 0$. We denote $J_n(x,y) \subset R B_n$ the (twosided) ideal generated by these
relations, so that $H_n(x,y) = R B_n/J_n(x,y)$.

Now assume $a,b,c \in R^{\times}$ and assume the
additional condition that $(a-b)(a-c)(b-c) \in R^{\times}$. Then, we define the \emph{tripled quadratic algebra} as
$$
\mathcal{H}_n = \mathcal{H}_n(a,b,c) = \frac{R B_n}{J_n(a,b) \cap J_n(a,c) \cap J_n(b,c)}
$$
We studied this algebra in detail in \cite{CABANESMARIN}, in
the special case where $R$ was a field of characteristic $2$ containing $\mathbbm{F}_4$, and $\{ a,b, c \} = \mathbbm{F}_4 \setminus \{ 0 \}$. It turns out that most results of \cite{CABANESMARIN} are also valid in the present setting, as we explain now.

 Notice first that $\mathcal{H}_n$ obviously projects onto $H_n(x,y)$
for all $x \neq y$ with $\{x,y \} \subset  \{a,b,c\}$, and in particular there is a natural morphism
$\mathcal{H}_n \to H_n(a,b) \oplus H_n(a,c) \oplus H_n(b,c)$. We denote $q_x : H_n(x,y) \to R$ characterized by
$s_i \mapsto x$. 
$$
\xymatrix{
H_n(b,c) \ar[d]_{q_{c}} \ar[dr]^{q_{b}} & H_n(a,c) \ar[dl]^{q_{c}} \ar[dr]_{q_{a}} & H_n(a,b) \ar[dl]_{q_{b}} \ar[d]^{q_a} \\
R & R & R 
}
$$

\begin{proposition} The natural morphism
$\mathcal{H}_n \to H_n(a,b) \oplus H_n(a,c) \oplus H_n(b,c)$ is injective, and its image is made of the triples $(z_c,z_b,z_a)$ 
such that $q_{\alpha}(z_{\alpha'}) =q_{\alpha}(z_{\alpha''})$ whenever $\{\alpha,\alpha',\alpha'' \} = \{a,b,c\}$. If $R$ is a field, then $\mathcal{H}_n$ has dimension $3(n!-1)$.
\end{proposition}
\begin{proof}
This proposition follows from general arguments as in \cite{CABANESMARIN}, proposition 5.7, as soon as we know that, for $\{x,y,z \} = \{a,b,c\}$, the (twosided) ideal $J_n(x,y) + J_n(x,z)$ is generated by $s_1 -x$.
This follows from the fact that $(s_1-x)(s_1-y) -
(s_1-x)(s_1-z) = (z-y)(s_1 -x)$ and that $(z-y) \in R^{\times}$. Indeed, this imply immediately that $J_n(a,b)+J_n(b,c)+J_n(a,c) = R B_n$ again because $(a-b)(a-c)(b-c) \in R^{\times}$.
\end{proof}

\begin{lemma} \label{lem:relsternary} The following equalities hold inside $\mathcal{H}_3$, where $\Sigma_1 = a+b+c$, $\Sigma_2 = ab+bc+ac$, $\Sigma_3 = abc$.
\begin{enumerate}
\item $(s_i - a)(s_i -b)(s_i -c)=0$, $i\in \{1,2\}$
\item $[s_2^2,s_1] = [s_2,s_1^2]$, that is $s_2^2 s_1 = s_2 s_1^2 - s_1^2 s_2+s_1s_2^2$
\item $s_2 s_1^2s_2 = -\Sigma_3 s_1 + \Sigma_2 s_1s_2 + s_1 s_2 s_1^2 - \Sigma_1 s_1 s_2^2 + s_1^2s_2^2$
\end{enumerate}
The following equalities hold inside $\mathcal{H}_4$, \begin{enumerate}
\item $s_2 s_3^2 = -s_1^2 s_3 + s_1^2s_2 + s_2^2s_3 - s_1 s_2^2 + s_1 s_3^2$
\item $s_2^2s_3s_1 = s_2 s_1^2s_3 - s_1^2 s_2 s_3 + s_1 s_2^2 s_3$
\item $s_2^2s_3^2 = -\Sigma_3 s_1 + \Sigma_2 s_1s_3 - \Sigma_1 s_1^2s_3 + \Sigma_3 s_2 - \Sigma_2 s_2s_3 + \Sigma_1 s_1^2s_2 - s_1^2s_2s_3 + \Sigma_1 s_2^2 s_3 - \Sigma_1 s_1 s_2^2 + s_1 s_2^2 s_3 + s_1^2 s_3^2$
\item $s_2^2s_3^2 =-\Sigma_3 s_1 + \Sigma_2 s_1s_3 - \Sigma_1 s_1^2s_3 + \Sigma_3 s_2 - \Sigma_2 s_2s_3 + \Sigma_1 s_1^2s_2 + \Sigma_1 s_2^2 s_3 - \Sigma_1 s_1 s_2^2+s_1^2s_3^2 + s_2^2s_1 s_3 - s_2 s_1^2s_3$
\end{enumerate}
\end{lemma}
\begin{proof}
By the previous proposition it is enough to prove these equalities inside each of the $H_n(x,y)$ for $\{x,y\} \subset \{a,b,c\}$. Depending on taste, this can be done (by hand or by computer) either by using the natural bases of the Hecke algebras or models over $\Z[a,b,c]$ of their simple modules.
\end{proof}

From now on we assume that $R=K$ is a field of characteristic $0$, with $a,b,c \in K$ being generic
(e.g., algebraic independent over $\Q \subset K$. Actually, the genericity condition under which the remaining part of the section is valid is precisely the following. One needs $a,b,c$ as well as $(a-b)(a-c)(b-c)$ to be nonzero, and also that the algebras $H_3(a,b,c)$ and $H_4(a,b,c)$ are split semisimple. This last condition can be made quite explicit under the trace conjecture of Broué, Malle and Michel (see \cite{BMM} \S 2), which implies that the algebras $H_3$ and $H_4$ should be symmetric algebras over $\Z[a^{\pm},b^{\pm},c^{\pm}]$, this conjecture being known to be true for $H_3$ (see  \cite{MALLEMICHEL,LG}). Under this conjecture the condition of being semisimple amounts to
the nonvanishing of a collection known as the Schur elements of these algebras, and they have been determined in \cite{TRBMW} for $H_3$.
For $H_3$ this condition implies in particular $ac+b^2 \neq 0$.

We consider the element
$$
\mathbf{b}=[s_2^2,s_1] - [s_2,s_1^2] = s_2^2 s_1 - s_1 s_2^2 - s_2 s_1^2 + s_1^2 s_2
$$
and define $\mathcal{K}_n = \mathcal{K}_n(a,b,c) =  H_n(a,b,c)/(\mathbf{b})$. Note that, because of the cubic relation on $s_1$ and $s_2$ one can express inside
$H_n(a,b,c)$ each $s_i^2$ as a linear combination of $s_i^{-1}$, $s_i$ and $1$. By direct calculation one gets that $\mathcal{K}_n$ is equivalently defined as the quotient of $H_n(a,b,c)$ by the relation
$[s_2^{-1},s_1] = [s_2,s_1^{-1}]$, that is 

\begin{center}

\begin{tikzpicture}[scale=.6]
\braid[braid colour=red,strands=3,braid start={(0,0)}]%
{ \dsigma _2^{-1} \dsigma_1 }
\node[font=\Large] at (3.5,-1) {\( - \)};
\braid[braid colour = red,strands=3,braid start={(3.5,0)}]
{\dsigma_1 \dsigma_2^{-1} }
\node[font=\Large] at (7,-1) {\(- \)};
\braid[braid colour = red,strands=3,braid start={(7,0)}]
{\dsigma_2 \dsigma_1^{-1}}
\node[font=\Large] at (10.5,-1) {\(+\)};
\braid[braid colour = red,strands=3,braid start={(10.5,0)}]
{\dsigma_1^{-1} \dsigma_2 }
\node[font=\Large] at (14.5,-1) {\( = 0\)};
\end{tikzpicture}

\end{center}

The image of $\mathbf{b}=[s_2^2,s_1] - [s_2,s_1^2]$
inside the 3-dimensional representation $V$ of $H_3(a,b,c)$ (using the matrix model of section \ref{sect:H3}) is
$$
\left(\begin{array}{ccc}
-(a-c)(ac+b^2) & 2c(a+b)-ab-c^2 & (a-c)(ac+b^2) \\ 
-(ac+b^2)(a^2+bc-2a(b+c)) & 2 (a-c)(ac+b^2) & (ac+b^2)(a^2+bc -2a(b+c)) \\
(a-c)(ac+b^2) & ab+c^2-2c(a+b) & -(a-c)(ac+b^2)
\end{array}
\right)
$$

\begin{proposition} \label{prop:KegalHtriple} Assume that $a,b,c$ are generic.
The twosided ideal of $H_3$ generated by $\mathbf{b}$ is
the indecomposable ideal attached to the representation $V$. Moreover, the natural morphism
$\mathcal{K}_n \onto \mathcal{H}_n$ is an isomorphism for $n \leq 4$. In particular the relations of lemma \ref{lem:relsternary} hold true inside $\mathcal{K}_3$ and $\mathcal{K}_4$, respectively.
\end{proposition}
\begin{proof}
Under the genericity assumption, the first assertion is equivalent to the non-vanishing of $\mathbf{b}$ under $V$, which has been established above, and its vanishing under the other irreducible representations of $H_3$. But each one of these factorizes through one of the (quadratic)
Hecke algebras $H_3(x,y)$, and $\mathbf{b}$ maps to $0$ in each one of them. This proves the first claim. Since this quotient has dimension $24 - 3^2 = 15 = 3(3! -1)$
this also proves that the surjective morphism $\mathcal{K}_3 \onto \mathcal{H}_3$ is an isomorphism.
The ideal generated by $\mathbf{b}$ inside the semisimple algebra $H_4(a,b,c)$ is attached, up to possibling extending the scalars to an algebraic closure of $K$, to the irreducible representations of $H_4(a,b,c)$ whose restriction to $H_3(a,b,c)$ contains
$V$ as a constituent. From the restriction rules recalled in section \ref{sect:cubicHecke4} we get
that these irreducible representations are exactly the
ones which do \emph{not} factorize through one of the
$H_4(x,y)$ and we get $\dim \mathcal{K}_4 = 3 \times 1^2 + 6 \times 3^2 + 3 \times 2^2 = 69 = 3(4!-1) = \dim \mathcal{H}_4$ which proves the claim. 
\end{proof}

The next propositions are similar to the ones established in \cite{CABANESMARIN}. Moreover, the proofs
are most of the time exactly the same. Therefore, we only provide precise references to them as well as the small changes to make when they are needed. Notice that the results below do not use the genericity condition (except for corollary \ref{cor:surjternary}).

\begin{proposition} For $n\geq 2$ one has 
\begin{enumerate}
\item $\mathcal{K}_{n+1} = \mathcal{K}_{n} + \mathcal{K}_{n}s_n \mathcal{K}_{n} + \mathcal{K}_{n} s_n^2 \mathcal{K}_{n}$
\item $\mathcal{K}_{n+1} = \mathcal{K}_{n} + \mathcal{K}_{n}s_n \mathcal{K}_{n} + \mathcal{K}_{n} s_n^2$
\item If $k<n$, $r,t \in \{ 0,1,2 \}$ we have $s_k^r s_1^t s_n^2 \equiv s_1^{r+t} s_n^2 \mod \mathcal{K}_n + \mathcal{K}_n s_n$
\item $\mathcal{K}_{n+1} = \mathcal{K}_{n} + \mathcal{K}_{n}s_n \mathcal{K}_{n} + \mathcal{K}_{2} s_n^2$
\end{enumerate}
\end{proposition}
\begin{proof}
Inside $\mathcal{K}_3$, from $\mathbf{b}s_2 = 0$
we get $s_2 s_1^2 s_2 = s_2^2 s_1 - s_1 s_2^3 + s_1^2 s_2^2$. Using the cubic relation to expand $s_2^3$ one
gets $s_2 s_1^2 s_2 \in \mathcal{K}_2 + \mathcal{K}_2 s_2 \mathcal{K}_2 + \mathcal{K}_2 s_2^2 \mathcal{K}_2$
hence $s_n s_{n-1}^2 s_n \in \mathcal{K}_{n} + \mathcal{K}_{n}s_n \mathcal{K}_{n} + \mathcal{K}_{n} s_n^2 \mathcal{K}_{n}$. From this one gets by induction
(1) following the proof of \cite{CABANESMARIN}, proposition 4.2. Then (2) follows from (1) and $\mathbf{b} = 0$ with the same proof as in \cite{CABANESMARIN}, lemma 5.11. When $n=2$ (3) is trivial so we can assume $n \geq 3$. By proposition \ref{prop:KegalHtriple} we know that the identities of lemma \ref{lem:relsternary} are valid inside $\mathcal{K}_{n+1}$ hence in particular
$s_2 s_3^2 \equiv s_1 s_3$ and $s_2^2 s_3^2 \equiv s_1^2 s_3^2$
modulo $\mathcal{K}_3 + \mathcal{K}_3 s_3$. This implies
that $s_{n-1} s_n^2 \equiv s_{n-2} s_n$ and $s_{n-1}^2 s_n^2 \equiv s_{n-2}^2 s_n$
modulo $\mathcal{K}_n + \mathcal{K}_n s_n.$ From this the arguments for \cite{CABANESMARIN} lemma 5.11  can be applied directly and to prove (3) and then (4).
\end{proof}

For $0 \leq k \leq n$, we let $s_{n,k} = s_n s_{n-1} \dots s_{n-k+1}$
with the convention that $s_{n,0} = 1$ and $s_{n,1} = s_n$.
We let $\mathcal{K}_n^k = \mathcal{K}_n s_{n,k}$ (hence $\mathcal{K}_n^0 = \mathcal{K}_n$). Similarly,
we let $x_{n,k} = s_n s_{n-1} \dots s_{n-k+2} s_{n-k+1}^2$
for $1 \leq k \leq n$,
with the convention $x_{n,1} = s_n^2$.

\begin{lemma} \label{lem:s1Un} \ 
\begin{enumerate}
\item If $r \leq n-1$, $1 \leq k \leq n$ and $c \in \{ 0, 1 , 2 \}$, then $s_r s_1^c x_{n,k} \in
s_1^{c+1} x_{n,k} + \mathcal{K}_n^0 + \dots + \mathcal{K}_n^k$.
\item $\mathcal{K}_n x_{n,k} \subset \mathcal{K}_2x_{n,k} + \mathcal{K}_n^0 + \dots + \mathcal{K}_n^n$.
\end{enumerate}
\end{lemma}
\begin{proof}
The proof of lemma 5.12 in \cite{CABANESMARIN} can be applied directly, as it only uses the previous result as well as $\mathbf{b}=0$.

\end{proof}

Finally, all these partial results imply the following one.

\begin{proposition}
Let $n \geq 2$. Then $\dim \mathcal{K}_n = 3(n!-1)$ and
$$
\mathcal{K}_{n+1} = \mathcal{K}_n \oplus \mathcal{K}_n^1 \oplus \dots \oplus \mathcal{K}_n^n \oplus \mathcal{K}_2 x_{n,1} \oplus \dots \oplus \mathcal{K}_2 x_{n,n}
$$
\end{proposition}
\begin{proof} The proof is the same as the one of proposition 5.13 in \cite{CABANESMARIN}.
\end{proof}
\begin{corollary} \label{cor:surjternary} For generic $a,b,c$, and all $n \geq 3$, the natural morphism $\mathcal{K}_n \onto \mathcal{H}_n$ is an isomorphism. 
\end{corollary}
\begin{corollary} Let $x_{n,k}' =   s_n s_{n-1} \dots s_{n-k+2} s_{n-k+1}^{-1}$. Then
$$
\mathcal{K}_{n+1} = \mathcal{K}_n \oplus \mathcal{K}_n^1 \oplus \dots \oplus \mathcal{K}_n^n \oplus \mathcal{K}_2 x'_{n,1} \oplus \dots \oplus \mathcal{K}_2 x'_{n,n}
$$ 
\end{corollary}
\begin{proof}
From $s_i^3 - \Sigma_1 s_i^2 + \Sigma_2 s_i - \Sigma_3 = 0$ one gets
$s_i^2 - \Sigma_1 s_i + \Sigma_2  = \Sigma_3 s_i^{-1}$
hence
$x_{n,k} \in \Sigma_3 x'_{n,k} + \mathcal{K}_n \oplus \mathcal{K}_n^1 \oplus \dots \oplus \mathcal{K}_n^n$.
Since $\Sigma_3=abc$ is invertible this proves the claim.
\end{proof}

Consider the shift morphism $\mathcal{K}_n \to \mathcal{K}_{n+1}$ given by $s_i \mapsto s_{i+1}$.
We define a basis of $\mathcal{K}_n$ inductively
by choosing $\mathcal{B}_1=\{ 1\}$ as a basis of $\mathcal{K}_1$,
$\mathcal{B}_2 = \{ 1, s_1, s_1^{-1}\}$ as a basis of $\mathcal{K}_2$,
and 
$$
\mathcal{B}_{n+1} =\left( \bigcup_{k=0}^n \mathcal{B}_n s_{n,k} \right) \cup \left( \bigcup_{k=1}^n \mathcal{B}_2 x'_{n,k} \right) 
$$

\subsection{Vogel's algebra : definition}

\subsubsection{Trivalent diagrams}

We recall here basic material from the theory
of Vassiliev invariants of knots and links. We let $\mathbf{D}$ denote the category whose objects are the
$[n]=\{ 1,\dots,n\}$ for $n \in \N = \Z_{\geq 0}$
with the convention $[0]=\emptyset$, and $\Hom_{\mathbf{D}}([p],[q])$ is made of the linear combinations of the set of couples
$(\Gamma,\alpha)$ where $\Gamma$ is a graph with vertices of degrees either $1$ or $3$, together
with a cyclic ordering $\alpha$ of the neighbours of any vertex of degree $3$, such that the set of vertices
of degree $1$ is $[p]\sqcup [q]$, modulo the relations that
\begin{itemize}
\item changing $\alpha$ to $\alpha'$ where the cyclic ordering of one single vertex has been changed yields
$(\Gamma,\alpha') \equiv -(\Gamma,\alpha)$
\item the following local `IHX' relation, where the orientation of the plane determines the chosen cyclic orderings
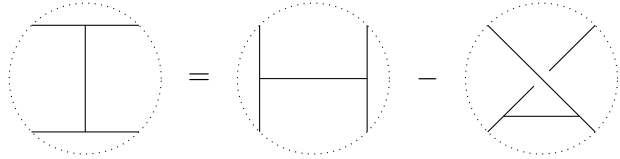
\begin{center}
\begin{tikzpicture}
\begin{scope}
\draw[dotted] (0,0) circle (1);
\draw (0.7071067812,0.7071067812) --(-0.7071067812,0.7071067812);
\draw (0.7071067812,-0.7071067812) -- (-0.7071067812,-0.7071067812);
\draw (0,0.7071067812) -- (0,-0.7071067812);
\draw (1.5,0) node {$=$};
\end{scope}
\begin{scope}[shift={(3,0)}]
\draw[dotted] (0,0) circle (1);
\draw (0.7071067812,0.7071067812) --(0.7071067812,-0.7071067812);
\draw (-0.7071067812,0.7071067812) -- (-0.7071067812,-0.7071067812);
\draw (0.7071067812,0) -- (-0.7071067812,0);
\draw (1.5,0) node {$-$};
\end{scope}
\begin{scope}[shift={(6,0)}]
\draw[dotted] (0,0) circle (1);
\draw (-0.7071067812,0.7071067812) --(0.7071067812,-0.7071067812);
\draw (0.7071067812,0.7071067812) -- (0.1,0.1);
\draw (-0.1,-0.1) -- (-0.7071067812,-0.7071067812);
\draw (0.5,-0.5) -- (-0.5,-0.5);
\end{scope}
\end{tikzpicture}
\end{center}
\end{itemize}

If $\g$ is a `quadratic' Lie algebra in the sense of \cite{VOGELALGSTR}, for example a semisimple Lie algebra endowed with its Killing form, there is a well-defined functor $\Phi_{\g} : \mathbf{D} \to \g{-}\mathrm{mod}$,
where $\g$-mod is the category of (finite-dimensional) $\g$-modules, such that $\Phi_{\g}([n]) = \g^{\otimes n}$ is the $n$-th tensor power of the adjoint representation of $\g$. Moreover, in case it is a simple Lie algebra endowed with its Killing form,
\begin{center}
\begin{tikzpicture}
\draw (-2,0) node {$\Phi_{\g} : $};
\draw (0.7071067812,0.7071067812) --(-0.7071067812,0.7071067812);
\draw (0.7071067812,-0.7071067812) -- (-0.7071067812,-0.7071067812);
\draw (0,0.7071067812) -- (0,-0.7071067812);
\draw (0.7071067812,0.7071067812) node {$\bullet$};
\draw (0.7071067812,-0.7071067812) node {$\bullet$};
\draw (-0.7071067812,0.7071067812) node {$\bullet$};
\draw (-0.7071067812,-0.7071067812) node {$\bullet$};
\draw (1.5,0) node[right] {$\mapsto \ \ \ \psi_{\g}=\sum_{i=1}^n \eps_i \otimes \eps^i \in \End(\g)^{\otimes 2} =\End(\g^{\otimes 2})$};
\end{tikzpicture}
\end{center}
where $e_1,\dots,e_n$ is a fixed arbitrary basis of $\g$,
$e^1,\dots,e^n$ is its dual basis with respect to the
Killing form, and $\eps_i$, $\eps^i$ denote the
images of $e_i, e^i \in \g \subset \mathsf{U}\g$
under the natural map $\mathsf{U}\g \to \End(\g)$.
Note that  
$$
2 \sum_{i=1}^n e_i \otimes e^i = \Delta(C) - C \otimes 1 - 1 \otimes C
$$
where $C = \sum_{i=1}^n e_ie^i \in Z(\mathsf{U}\g)$ is
the Casimir operator. Therefore $\psi_{\g}$ acts
by a scalar on any simple component of $\g^{\otimes 2}$.

The action of $\psi_{\g}$ on $\Lambda^2 \g$ has 2 eigenvalues, $0$ and $t \neq 0$. Moreover, $K= \Ker (\psi_{\g})_{\Lambda^2 \g}) = \Ker([\ , \ ] : \g^{\otimes 2} \to \g)$
and $\Lambda^2\g /K \simeq \g$.
For most Cartan types,
$\psi_{\g}$ acts on $S^2\g \subset \g \otimes \g$ with 4 eigenvalues $2t, \alpha,\beta,\gamma$.

\subsubsection{The case of exceptional Lie algebras}

Assume now that $\g$ is a (complex) simple Lie algebra
of exceptional Cartan type $E_6$,$E_7$, $E_8$, $F_4$ or $G_2$. Then $\psi_{\g}$ acts on $S^2\g \subset \g \otimes \g$ with only 3 eigenvalues $2t, \alpha,\beta$
with the relation $\alpha + \beta =t/3$. With the
purpose of uniformizing the results in the spirit of
\cite{VOGELLIEUNIV}, another parameter $\gamma = 2t/3$ is introduced, so that $\alpha +\beta+\gamma = t$.
The prefered choice of ordering for $\alpha,\beta$ is as 
as follows (see \cite{VOGELLIEUNIV}).
$$
\begin{array}{|c||c|c|c|}
\hline
 & \alpha & \beta & \gamma \\
\hline
E_6 & 3 & -1 & 4 \\
\hline
E_7 & 4 & -1 & 6 \\
\hline
E_8 & 6 & -1 & 10 \\
\hline
F_4& 5 & -2 & 6 \\
\hline
G_2 & 5 & -3 & 4 \\
\hline
\end{array}
$$

Assume that such a Lie algebra $\g$ is fixed, together
with the corresponding choice of $\alpha,\beta$.
Then the functor $\Phi_{\g}$ factors through
the category $\mathbf{D}^{exc}$, quotient of $\mathbf{D}$ by the local relation
\begin{center}
\begin{tikzpicture}
\begin{scope}[scale=.5]
\draw[dotted] (0,0) circle (2);
\draw (0,0) circle (1);
\draw (0.7071067812,0.7071067812) --(2*0.7071067812,2*0.7071067812);
\draw (0.7071067812,-0.7071067812) --(2*0.7071067812,-2*0.7071067812);
\draw (-0.7071067812,0.7071067812) --(-2*0.7071067812,2*0.7071067812);
\draw (-0.7071067812,-0.7071067812) --(-2*0.7071067812,-2*0.7071067812);
\draw (3,0) node[right] {$=(\alpha+\beta) \left( \ \right. $};
\end{scope}
\begin{scope}[scale=.5,shift={(10,0)}]
\draw[dotted] (0,0) circle (2);
\draw (-2*0.7071067812,2*0.7071067812) --(2*0.7071067812,2*0.7071067812);
\draw (-2*0.7071067812,-2*0.7071067812) --(2*0.7071067812,-2*0.7071067812);
\draw (0,2*0.7071067812) --(0,-2*0.7071067812);
\draw (2.5,0) node[right] {$+ $};
\end{scope}
\begin{scope}[scale=.5,shift={(17,0)}]
\draw[dotted] (0,0) circle (2);
\draw (-2*0.7071067812,2*0.7071067812) --(-2*0.7071067812,-2*0.7071067812);
\draw (2*0.7071067812,2*0.7071067812) --(2*0.7071067812,-2*0.7071067812);
\draw (2*0.7071067812,0) --(-2*0.7071067812,0);
\draw (2.5,0) node[right] {$\left. \right)$};
\end{scope}
\begin{scope}[scale=.5,shift={(2,-5)}]
\draw (-4,0) node {$-\frac{\alpha\beta}{2} \left( \right.$};
\draw[dotted] (0,0) circle (2);
\draw (-2*0.7071067812,2*0.7071067812) --(-2*0.7071067812,-2*0.7071067812);
\draw (2*0.7071067812,2*0.7071067812) --(2*0.7071067812,-2*0.7071067812);
\draw (2.5,0) node[right] {$+$};
\end{scope}
\begin{scope}[scale=.5,shift={(9,-5)}]
\draw[dotted] (0,0) circle (2);
\draw (-2*0.7071067812,2*0.7071067812) --(2*0.7071067812,2*0.7071067812);
\draw (-2*0.7071067812,-2*0.7071067812) --(2*0.7071067812,-2*0.7071067812);
\draw (2.5,0) node[right] {$+$};
\end{scope}
\begin{scope}[scale=.5,shift={(16,-5)}]
\draw[dotted] (0,0) circle (2);
\draw (-2*0.7071067812,2*0.7071067812) --(2*0.7071067812,-2*0.7071067812);
\draw (2*0.7071067812,2*0.7071067812) -- (0.2,0.2);
\draw (-0.2,-0.2) -- (-2*0.7071067812,-2*0.7071067812);
\draw (2.5,0) node[right] {$\left. \right)$};
\end{scope}

\end{tikzpicture}
\end{center}

Let us now consider the algebra
$$
D^{exc}(n) = \frac{\mathbf{D}^{exc}([n],[n])}{\sum_{k<n} \mathbf{D}^{exc}([n],[k])\circ
\mathbf{D}^{exc}([k],[n])}
$$
Inside this quotient, the following local relation obviously holds.

\begin{center}
\begin{tikzpicture}
\begin{scope}[scale=.5]
\draw[dotted] (0,0) circle (2);
\draw (0,0) circle (1);
\draw (0.7071067812,0.7071067812) --(2*0.7071067812,2*0.7071067812);
\draw (0.7071067812,-0.7071067812) --(2*0.7071067812,-2*0.7071067812);
\draw (-0.7071067812,0.7071067812) --(-2*0.7071067812,2*0.7071067812);
\draw (-0.7071067812,-0.7071067812) --(-2*0.7071067812,-2*0.7071067812);
\draw (3,0) node[right] {$=(\alpha+\beta)$};
\end{scope}
\begin{scope}[scale=.5,shift={(9,0)}]
\draw[dotted] (0,0) circle (2);
\draw (-2*0.7071067812,2*0.7071067812) --(2*0.7071067812,2*0.7071067812);
\draw (-2*0.7071067812,-2*0.7071067812) --(2*0.7071067812,-2*0.7071067812);
\draw (0,2*0.7071067812) --(0,-2*0.7071067812);
\end{scope}
\begin{scope}[scale=.5,shift={(16,0)}]
\draw (-3.5,0) node {$-\frac{\alpha\beta}{2} \left( \right.$};
\draw[dotted] (0,0) circle (2);
\draw (-2*0.7071067812,2*0.7071067812) --(-2*0.7071067812,-2*0.7071067812);
\draw (2*0.7071067812,2*0.7071067812) --(2*0.7071067812,-2*0.7071067812);
\draw (2.5,0) node[right] {$+$};
\end{scope}
\begin{scope}[scale=.5,shift={(22,0)}]
\draw[dotted] (0,0) circle (2);
\draw (-2*0.7071067812,2*0.7071067812) --(2*0.7071067812,-2*0.7071067812);
\draw (2*0.7071067812,2*0.7071067812) -- (0.2,0.2);
\draw (-0.2,-0.2) -- (-2*0.7071067812,-2*0.7071067812);
\draw (2.5,0) node[right] {$\left. \right)$};
\end{scope}
\end{tikzpicture}
\end{center}

\subsubsection{Trivalent diagrams and infinitesimal braids}

\begin{center}
\begin{tikzpicture}
\begin{scope}[scale=.5]
\draw[dotted] (0,0) circle (2);
\draw[thick] (2*0.8660254037,2*0.5) arc (30:55:2);
\draw[thick] (2*0.8660254037,2*0.5) arc (30:5:2);
\draw[thick] (-2*0.8660254037,2*0.5) arc (150:175:2);
\draw[thick] (-2*0.8660254037,2*0.5) arc (150:125:2);
\draw[thick] (0,-2) arc (-90:-115:2);
\draw[thick] (0,-2) arc (-90:-65:2);
\draw (0,0)++(30:2.2) node (A) {$$};
\draw (0,0)++(150:2.2) node (B) {$$};
\draw[thick] (0,0)++(-80:2) -- (A);
\draw[thick] (0,0)++(-100:2) -- (B);
\draw (3,0) node[right] {$-$};
\end{scope}
\begin{scope}[scale=.5,shift={(7,0)}]
\draw[dotted] (0,0) circle (2);
\draw[thick] (2*0.8660254037,2*0.5) arc (30:55:2);
\draw[thick] (2*0.8660254037,2*0.5) arc (30:5:2);
\draw[thick] (-2*0.8660254037,2*0.5) arc (150:175:2);
\draw[thick] (-2*0.8660254037,2*0.5) arc (150:125:2);
\draw[thick] (0,-2) arc (-90:-115:2);
\draw[thick] (0,-2) arc (-90:-65:2);
\draw (0,0)++(30:2.2) node (A) {$$};
\draw (0,0)++(150:2.2) node (B) {$$};
\draw[thick] (0,0)++(-80:2) -- (B);
\draw[thick] (0,0)++(-100:2) -- (A);
\draw (3,0) node[right] {$=$};
\end{scope}
\begin{scope}[scale=.5,shift={(14,0)}]
\draw[dotted] (0,0) circle (2);
\draw[thick] (2*0.8660254037,2*0.5) arc (30:55:2);
\draw[thick] (2*0.8660254037,2*0.5) arc (30:5:2);
\draw[thick] (-2*0.8660254037,2*0.5) arc (150:175:2);
\draw[thick] (-2*0.8660254037,2*0.5) arc (150:125:2);
\draw[thick] (0,-2) arc (-90:-115:2);
\draw[thick] (0,-2) arc (-90:-65:2);
\draw (0,0)++(-90:2.2) node (B) {$$};
\draw (0,0)++(150:2.2) node (C) {$$};
\draw[thick] (0,0)++(20:2) -- (B);
\draw[thick] (0,0)++(40:2) -- (C);
\draw (3,0) node[right] {$-$};
\end{scope}
\begin{scope}[scale=.5,shift={(21,0)}]
\draw[dotted] (0,0) circle (2);
\draw[thick] (2*0.8660254037,2*0.5) arc (30:55:2);
\draw[thick] (2*0.8660254037,2*0.5) arc (30:5:2);
\draw[thick] (-2*0.8660254037,2*0.5) arc (150:175:2);
\draw[thick] (-2*0.8660254037,2*0.5) arc (150:125:2);
\draw[thick] (0,-2) arc (-90:-115:2);
\draw[thick] (0,-2) arc (-90:-65:2);
\draw (0,0)++(-90:2.2) node (B) {$$};
\draw (0,0)++(150:2.2) node (C) {$$};
\draw[thick] (0,0)++(20:2) -- (C);
\draw[thick] (0,0)++(40:2) -- (B);
\end{scope}
\end{tikzpicture}
\end{center}

\begin{center}
\begin{tikzpicture}
\begin{scope}[scale=.5]
\draw[dotted] (0,0) circle (2);
\draw[thick] (2*0.8660254037,2*0.5) arc (30:55:2);
\draw[thick] (2*0.8660254037,2*0.5) arc (30:5:2);
\draw[thick] (-2*0.8660254037,2*0.5) arc (150:175:2);
\draw[thick] (-2*0.8660254037,2*0.5) arc (150:125:2);
\draw[thick] (0,-2) arc (-90:-115:2);
\draw[thick] (0,-2) arc (-90:-65:2);
\draw (0,0)++(30:2.2) node (A) {$$};
\draw (0,0)++(150:2.2) node (B) {$$};
\draw[thick] (0,0)++(-80:2) -- (A);
\draw[thick] (0,0)++(-100:2) -- (B);
\draw[red] (0,-2) circle (1);
\draw (3,0) node[right] {$-$};
\end{scope}
\begin{scope}[scale=.5,shift={(7,0)}]
\draw[dotted] (0,0) circle (2);
\draw[thick] (2*0.8660254037,2*0.5) arc (30:55:2);
\draw[thick] (2*0.8660254037,2*0.5) arc (30:5:2);
\draw[thick] (-2*0.8660254037,2*0.5) arc (150:175:2);
\draw[thick] (-2*0.8660254037,2*0.5) arc (150:125:2);
\draw[thick] (0,-2) arc (-90:-115:2);
\draw[thick] (0,-2) arc (-90:-65:2);
\draw (0,0)++(30:2.2) node (A) {$$};
\draw (0,0)++(150:2.2) node (B) {$$};
\draw[thick] (0,0)++(-80:2) -- (B);
\draw[thick] (0,0)++(-100:2) -- (A);
\draw (3,0) node[right] {$=$};
\draw[red] (0,-2) circle (1);
\end{scope}
\begin{scope}[scale=.5,shift={(14,0)}]
\draw[dotted] (0,0) circle (2);
\draw[thick] (2*0.8660254037,2*0.5) arc (30:55:2);
\draw[thick] (2*0.8660254037,2*0.5) arc (30:5:2);
\draw[thick] (-2*0.8660254037,2*0.5) arc (150:175:2);
\draw[thick] (-2*0.8660254037,2*0.5) arc (150:125:2);
\draw[thick] (0,-2) arc (-90:-115:2);
\draw[thick] (0,-2) arc (-90:-65:2);
\draw (0,0)++(-90:2.2) node (B) {$$};
\draw (0,0)++(150:2.2) node (C) {$$};
\draw[thick] (0,0)++(20:2) -- (B);
\draw[thick] (0,0)++(40:2) -- (C);
\draw (3,0) node[right] {$-$};
\draw[red] (2*0.8660254037,2*0.5) circle (1);
\end{scope}
\begin{scope}[scale=.5,shift={(21,0)}]
\draw[dotted] (0,0) circle (2);
\draw[thick] (2*0.8660254037,2*0.5) arc (30:55:2);
\draw[thick] (2*0.8660254037,2*0.5) arc (30:5:2);
\draw[thick] (-2*0.8660254037,2*0.5) arc (150:175:2);
\draw[thick] (-2*0.8660254037,2*0.5) arc (150:125:2);
\draw[thick] (0,-2) arc (-90:-115:2);
\draw[thick] (0,-2) arc (-90:-65:2);
\draw (0,0)++(-90:2.2) node (B) {$$};
\draw (0,0)++(150:2.2) node (C) {$$};
\draw[thick] (0,0)++(20:2) -- (C);
\draw[thick] (0,0)++(40:2) -- (B);
\draw[red] (2*0.8660254037,2*0.5) circle (1);
\end{scope}
\end{tikzpicture}
\end{center}
The IHX relations imply the above so-called 4T relations (apply the IHX relation inside the red circles).
These can be rewritten `horizontally' as
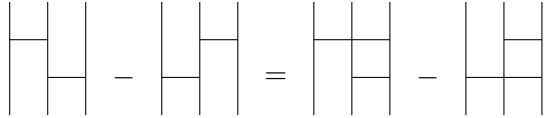
\begin{center}
\begin{tikzpicture}
\begin{scope}[scale=.5]
\draw (0,2) -- (0,-1);
\draw (1,2) -- (1,-1);
\draw (2,2) -- (2,-1);
\draw (0,1) -- (1,1);
\draw (1,0) -- (2,0);
\draw (3,0) node {$-$};
\end{scope}
\begin{scope}[scale=.5,shift={(4,0)}]
\draw (0,2) -- (0,-1);
\draw (1,2) -- (1,-1);
\draw (2,2) -- (2,-1);
\draw (1,1) -- (2,1);
\draw (0,0) -- (1,0);
\draw (3,0) node {$=$};
\end{scope}
\begin{scope}[scale=.5,shift={(8,0)}]
\draw (0,2) -- (0,-1);
\draw (1,2) -- (1,-1);
\draw (2,2) -- (2,-1);
\draw (0,1) -- (2,1);
\draw (1,0) -- (2,0);
\draw (3,0) node {$-$};
\end{scope}
\begin{scope}[scale=.5,shift={(12,0)}]
\draw (0,2) -- (0,-1);
\draw (1,2) -- (1,-1);
\draw (2,2) -- (2,-1);
\draw (1,1) -- (2,1);
\draw (0,0) -- (2,0);
\end{scope}
\end{tikzpicture}
\end{center}
and can be viewed as a relation inside $\mathbf{D}([3],[3])$. Let $n \geq 2$. For $1 \leq i<j \leq n$ we let
$d_{ij} \in \mathbf{D}([n],[n])$ denote the diagram
that differs from the identity diagram only in that
the $i$-th and $j$-th strands are connected by an additional straight arc. Then the above relation reads
$[d_{ij},d_{ik}+d_{kj}]=0$, with the convention
that $d_{ij}=d_{ji}$ when $i > j$. Obviously one also
has the relations $[d_{ij},d_{rs}]=0$ whenever $\#\{i,j,r,s\}=4$.

Moreover, there exists a natural embedding
$\mathfrak{S}_n \subset \mathbf{D}([n],[n])$. We
denote $s_{ij}$ the (image of) the transposition $(i \ j)$. Clearly, for $w \in \mathfrak{S}_n$, one has
$w d_{ij} w^{-1}= d_{w(i),w(j)}$.

Recall that the holonomy Lie algebra $\mathcal{T}_n$ of the ordered configuration space of $n$ points in the plane
admits a presentation by generators $t_{ij}=t_{ji}$
for $1 \leq i\neq j \leq n$ and relations
$[t_{ij},t_{rs}]=0$ whenever $\#\{i,j,r,s\}=4$
and $[t_{ij},t_{ik}+t_{kj}]=0$. It is acted upon by
$\mathfrak{S}_n$ such that $w.t_{ij}=t_{w(i),w(j)}$.
It follows that
there is a natural algebra morphism
$\mathfrak{B}_n \to \mathbf{D}([n],[n])$
inducing the identity on $\mathfrak{S}_n$
and mapping $t_{ij}$ to $d_{ij}$,
where $\mathfrak{B}_n = \kk \mathfrak{S}_n\ltimes \mathsf{U}\mathcal{T}_n$.

\subsubsection{Vogel's algebra}

The image of $ \mathfrak{B}_n$ under the composite
map $\mathfrak{B}_n \to \mathbf{D}([n],[n]) \to \mathbf{D}^{exc}([n],[n]) \to D^{exc}(n)$
is an algebra generated by $\mathfrak{S}_n$
and the (images of the) $t_{ij}$'s, that satisfies the defining relations of $ \mathfrak{B}_n$ together
with the following
additional ones.
\begin{enumerate}
\item For all $i,j$, $t_{ij}(i,j)=(i,j)t_{ij}=t_{ij}$
\item For all $i,j$,
$$
t_{ij}^2 - (\alpha+\beta) t_{ij} + \frac{\alpha \beta}{2} \left(  1 + (i,j) \right) = 0
$$
\end{enumerate}
Indeed, the first one is a consequence of the image of the IHX relation inside $D^{exc}(n)$, and the second one is the image of the `exceptional' relation of $\mathbf{D}^{exc}$.
In this paper we call the corresponding quotient of $\mathfrak{B}_n$ by these abstract relations \emph{Vogel's algebra} and denote it by $\mathfrak{V}_n$. P. Vogel communicated to me (\cite{VOGELPRIV}) that he computed
the dimension over $\Q(\alpha,\beta)$ of these algebras for small $n$ under the assumption that they are finite-dimensional, and
that he got that they are semisimple of dimensions
$3, 20, 264, 6490, 141824, 6799151$
for $n =2,3,4,5,6,7$. He conjectured this algebra
to be finite-dimensional and semisimple in general.

\subsection{Vogel's algebra : representations}

\subsubsection{Quotients of Vogel's algebra}

We now assume that $\kk$ is a field of characteristic not $2$.
For all $u,v  \in \kk$ there exists a well-defined (surjective) morphism
$\varphi_{u,v} :\mathfrak{B}_n \to \kk \mathfrak{S}_n$
restricting to the identity on $\kk \mathfrak{S}_n$
and mapping $t_{ij}$ to $u.1 +v.(i \ j)$. It is immediately checked that it factorizes through $\mathfrak{V}_n$ iff $u=v \in \{ \alpha/2,\beta/2\}$.
It follows that the composite of natural maps
$\kk \mathfrak{S}_n \to \mathfrak{B}_n \to \mathfrak{V}_n$ is injective.
We also note that the morphism $\kk \mathfrak{B}_n \to \kk$ mapping each $w \in \mathfrak{S}_n$ to its sign
and $t_{ij}$ to $0$ also factorizes through $\mathfrak{V}_n$. We denote it by $\eps : \mathfrak{V}_n \to \kk$.

We now introduce the algebra $Br_n(m)$ of Brauer diagrams. Its natural basis is made of Brauer diagrams,
a Brauer diagram being a collection of matchings
between $2n$ points which is depicted with $n$ points on the top and $n$ points on the bottom, so that they
can be composed to the expense of possibly making a
circle appear. In this case this circle is converted
into the scalar $m$. We refer to e.g. \cite{GOODMANWALLACH} for more details on
this object, and recall that is admits a natural
subalgebra isomorphic to $\kk \mathfrak{S}_n$. By
abuse of notations we identify the transposition $(i,j)$ with its image. We denote $p_{ij}$ the diagram
matching $i$ with $j$, $n+i$ with $n+j$ and $k$
with $n+k$ for $k \geq n$ and $k \not\in \{ i,j \}$.
For instance the following picture depicts $p_{13} \in Br_4(m)$.
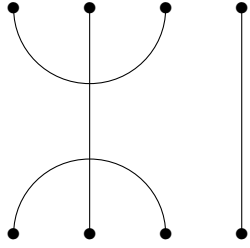
\begin{center}
\begin{tikzpicture}
\draw (0,1.5) node {$\bullet$};
\draw (1,1.5) node {$\bullet$};
\draw (2,1.5) node {$\bullet$};
\draw (3,1.5) node {$\bullet$};
\draw (0,-1.5) node {$\bullet$};
\draw (1,-1.5) node {$\bullet$};
\draw (2,-1.5) node {$\bullet$};
\draw (3,-1.5) node {$\bullet$};
\draw (0,1.5) arc (-180:0:1);
\draw (1,1.5) -- (1,-1.5);
\draw (3,1.5) -- (3,-1.5);
\draw (0,-1.5) arc (180:0:1);
\end{tikzpicture}
\end{center}
There exists a 2-parameters family of morphisms
$\psi_{u,v} : \mathfrak{B}_n \to Br_n(m)$ restricting to the identity on $\kk \mathfrak{S}_n$
and mapping $t_{ij}$ to $u.1 +v.((i,j) - p_{ij})$
(see \cite{QUOTINF}). It factors through the relation
$t_{ij}(i,j) = t_{ij}$ only if $u=v$, and finally
factors through $\mathfrak{V}_n$ iff $u \in \{ \alpha/2,\beta/2\}$ and $u(m-4) = -(\alpha + \beta)$. We thus
get the following commutative diagram, with $m_x = 4- 2(1+x^{-1})$ :
$$
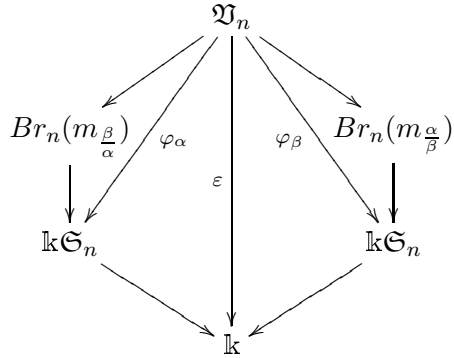
\xymatrix{
 & \mathfrak{V}_n \ar[dr] \ar[dl] \ar[ddl]^{\varphi_{\alpha}}
\ar[ddr]_{\varphi_{\beta}} \ar[ddd]_{\eps}
 & \\
 Br_n(m_{\frac{\beta}{\alpha}}) \ar[d] & & Br_n(m_{\frac{\alpha}{\beta}}) \ar[d] \\
 \kk \mathfrak{S}_n \ar[dr] & &  \kk \mathfrak{S}_n \ar[dl] \\
 & \kk & 
}
$$ 

\subsubsection{Vogel's algebra for $n=2,3$}
First assume $n=2$, $\alpha\beta \neq 0$ and $\alpha \neq \beta$. Also assume that $\kk$ is a field.
From the relations it is clear that $\mathfrak{V}_2$
is spanned by $1,(1,2), t_{12}$ and also
by $1,t_{12},t_{12}^2$. Moreover,
we easily get from the relations that
$t_{ij}^3 = (\alpha+\beta) t_{ij}^2-\alpha\beta t_{ij}$
which can be rewritten $t_{ij}(t_{ij}- \alpha)(t_{ij}-\beta)=0$. Therefore 
$$
\mathfrak{V}_2\simeq \kk[X]/X(X-\alpha)(X-\beta)\simeq
\frac{\kk[X]}{X}\oplus \frac{\kk[X]}{X-\alpha} \oplus \frac{\kk[X]}{X-\beta} \simeq \kk \oplus \kk \oplus \kk
$$
and, under this isomorphism, we have $t_{12} \mapsto (0,\alpha,\beta)$ while $(1,2) \mapsto (-1,1,1)$.

We now consider the case $n=3$, and assume that $char. \kk \not\in \{ 2,3 \}$. In particular the algebra $\kk \mathfrak{S}_3$ is split semisimple. The two surjective morphisms $\varphi_{\alpha}$ and $\varphi_{\beta}$
together with the irreducible representations of $\mathfrak{S}_3$ provide 
irreducible representations
of dimension 1 and 2
 of $\mathfrak{V}_3$. They can be distinguished up to isomorphism from the spectrum of $t_{1,2}$, which is
$
\{ \alpha, 0 \},
\{ \beta, 0 \},\{ \alpha \},\{ \beta \},\{ 0 \}
$. Moreover, since the Brauer algebra admits an irreducible representation of dimension $3$, it provides generically at least one 3-dimensional generically irreducible representation for $\mathfrak{V}_3$. Actually, the irreducible representations of dimension $3$ of $\mathfrak{B}_3$ have been classified in \cite{THESE} (see proposition 10). There is up to isomorphism only one such irreducible representation such that the spectrum of $t_{12}$ is (contained in) $\{ 0, \alpha,\beta \}$, provided that $(\alpha\beta)/2 \neq ((\alpha+\beta)/3)^2$. One matrix model for it is
$$
(1,2)=
\left(
\begin{array}{ccc}
-1 & 0 & 0 \\
0 & 1 & 0 \\
0 & 0 & 1
\end{array}
\right) \ \ 
(2,3)=\frac{1}{2}
\left(
\begin{array}{ccc}
1 & 1 & 0 \\
3 & -1 & 0 \\
0 & 0 & 2
\end{array}
\right) \ \ 
t_{12}=
\left(
\begin{array}{ccc}
0 & 0 & 0 \\
0 & 2d & b \\
0 & c & d
\end{array}
\right)  
$$
with $d = (\alpha+\beta)/3$ and $bc = 2d^2 - \alpha\beta$. It is straightforward to check that it defines indeed
a representation of $\mathfrak{B}_3$, which is irreducible under the above condition, and that the
defining relations of $\mathfrak{V}_3$ are indeed satisfied. This proves that $\mathfrak{V}_3$ has dimension at least $3 \times 1^2 + 2 \times 2^2 + 3^2 = 20$.

\subsubsection{Braids and infinitesimal braids}

In this section, we let $R = \C[[h]]$, and
assume $\alpha,\beta \in \kk\setminus \{ 0\} \subset \C$. We also assume $\alpha \neq \beta$. We identify
$B_n$ with the fundamental group of $\C_*^n/\mathfrak{S}_n$ with $\C_*^n= \{ (z_1,\dots,z_n) \in \C^n | z_i \neq z_j \}$ with respect with some arbitrarily chosen base point (alternatively, we could choose an arbitrary associator). From the $\mathfrak{V}_n\otimes R$-valued
1-form 
$$
\frac{h}{\ii \pi} \sum_{1 \leq i < j \leq n} t_{ij} \dd \log(z_i - z_j)
$$
we get by monodromy a morphism $R B_n \to \mathfrak{V}_n\otimes R$ that can be extended to $K B_n \to \mathfrak{V}_n\otimes K$ with $K = \C((h))$.
The image of an arbitrary Artin generator $s_i$
is then conjugated to $(i,i+1)e^{h t_{i,i+1}}$ and therefore to $(1,2)e^{h t_{1,2}} \in \mathfrak{V}_2\otimes R$. It follows that it is annihilated
by the polynomial $(X+1)(X-e^{h\alpha})(X-e^{h\beta})$.
Therefore the morphism $R B_n \to \mathfrak{V}_n\otimes R$ factorizes through the cubic Hecke algebra.

We now prove that the morphism $K B_n \to \mathfrak{V}_n\otimes K$ is surjective, following arguments of \cite{ASSOC}. Indeed, The image of $s_i$ is congruent to $(i,i+1)$ modulo $h$,
and the image of $h^{-1}(s_i^2 -1)$ is congruent to $2 t_{i,i+1}$ modulo $h$. Let $C$ be the $R$-subalgebra of $K B_n$ generated by the $\sigma_i$ and the 
$h^{-1}(s_i^2 -1)$ for $1 \leq i \leq n-1$. The monodromy morphism $R B_n \to \mathfrak{V}_n \otimes R$
can be extended to a morphism $C \to \mathfrak{V}_n \otimes R$. Since $\mathfrak{S}_n$
is generated by the $(i,i+1)$ and since $\mathfrak{V}_n$ is generated by $\mathfrak{S}_n$ and $t_{12}$,
we get that that reduction mod $h$ of this morphism is surjective, and therefore the morphism $C \to \mathfrak{V}_n \otimes R$ is surjective by Nakayama's lemma. It follows immediately that the morphism
$K B_n \to \mathfrak{V}_n\otimes K$ is surjective.

In the cases $n \leq 3,4,5$, since the cubic Hecke algebras are finite dimensional and semisimple, it follows that $K \mathfrak{V}_n$ is also finite-dimensional and semisimple.

We now restrict to the case $n=3$. We know that 
$K \mathfrak{V}_3$ is isomorphic to a quotient of the
cubic Hecke algebra of dimension at least $20$, with
at least 3 irreducible 1-dimensional representations,
2 of dimension 2, 1 of dimension 3. It follows that either $K \mathfrak{V}_3 \simeq \GQ_3$, or
 $K \mathfrak{V}_3 \simeq H_3$ and there exists a 2-dimensional irreducible representation of 
 $ \mathfrak{V}_3$ such that the spectrum of $\sigma_i$ under the monodromy morphism 
 is $\{ e^{h\alpha},e^{h\beta} \}$. This is possible only if there exists a 2-dimensional irreducible representation of $\mathfrak{V}_3$ such that $Sp(t_{ij}) = \{ \alpha,\beta \}$. But since $\alpha\beta \neq 0$
 this forces $t_{ij}$ to be invertible, whence by the relation $t_{ij}(i,j)=t_{ij}$ we would have $(i,j)=1$ for all $i,j$. But then the action of $\mathfrak{S}_n$ would be trivial, and $\Ker(t_{12}-\alpha)$ would
 be proper stable subspace of the representation. This
 contradiction proves that the irreducible representations of $\mathfrak{V}_3$ determined above are the only ones,
 that $\dim \mathfrak{V}_3 = 20$ and that $\GQ_3 \simeq K \mathfrak{V}_3$.

\subsubsection{Representations}

If $E$ be a representation of $\mathfrak{S}_n$ it
can be expanded as a representation of $\mathfrak{V}_n$
by letting $t_{ij} = \frac{\alpha}{2}(1 + (i \ j))$
or $t_{ij} = \frac{\beta}{2}(1 + (i \ j))$.

For $E$ associated to $[3,1]$ we denote these representations $\mathcal{U}_{b,a}$ and $\mathcal{U}_{c,a}$, 
while for $E$ associated to $[2,1,1]$ we denote these representations $\mathcal{U}_{a,b}$ and $\mathcal{U}_{a,c}$. One readily checks that the monodromy representation associated to $\mathcal{U}_{x,y}$ is $U_{x,y}$.

There is an algebra morphism $\mathfrak{V}_4 \to \mathfrak{V}_3$ induced by the special morphism $\mathfrak{S}_4 \onto \mathfrak{S}_3$ and $t_{ij} \mapsto t_{rs}$ with $\{r,s \} = \{i,j\}$ if $1 \leq i,j \leq 3 $ and $\{r,s,i,j \} = \{1,2,3,4 \}$ otherwise. From this the irreducible representations of $\mathfrak{V}_3$ can be expanded to $\mathfrak{V}_4$. The monodromy of
these representations yields the representations $V,T_{a,b},T_{a,c},S_a,S_b,S_c$.

Let $E = \C \mathcal{R}$ be a vector space spanned
by the set $\mathcal{R}$ of transpositions of $\mathfrak{S}_n$. We denote $v_s = v_{ij}$ the basis
vector corresponding to $s = (i \ j) \in \mathcal{R}$. We endow it with the $\mathfrak{S}_n$-permutation module structure associated to the action by 
conjugation on the transpositions, $w. v_s = v_{wsw^{-1}}$ and we fix $\la\in \C^{\times}, m,x \in \C$,  

We set $t_s.v_s = \la (m+x) v_s$, $t_s.v_u =x \la v_u+ \la v_{sus}$
if $s \neq u$, $su = us$, and
$t_s.v_u = x \la v_u+ \la v_{sus} - \la v_s$ otherwise. It is known
(see \cite{KRAMMINF}, \cite{KRAMCRG}) that this defines a representation of $\mathfrak{B}_n$ for $x =0$, which is irreducible for generic values of $m$.
Since $t_{ij} \mapsto \la(t_{ij} + x)$ defines an automorphism of $\mathfrak{B}_n$ these formulas provide a representation of $\mathfrak{B}_n$ for arbitrary values of $x$, irreducible for generic values of $m$.
 Moreover,
$st_s.v_s = \la (m+x) v_s = t_s.v_s$ ; if $s \neq u$, $su = us$, then 
 $st_s.v_u =x \la v_{sus}+ \la v_{u} = t_s.v_u$ iff $x = 1$ ; 
otherwise $st_s.v_u = x \la v_{sus}+ \la v_{u} -\la  v_s
= t_s.v_u$ iff $x = 1$.

Consider $s = (1 \ 2)$. 
The eigenspaces for $t_{12}$ are $\C v_s$ with eigenvalue $\la(m+1)$ ;
the subspace spanned by the $v_u$ for $su = us$, $s \neq u$ and the $v_{1k}+v_{2,k}+(2/(m-1))v_s$ for $3 \leq k \leq n$ with eigenvalue $2\la$, which has dimension $(n-1)(n-2)/2$ ; the subspace spanned by the $v_{1k}-v_{2k}$ with eigenvalue $0$ has dimension $0$, which has
dimension $n-2$. This proves that $t_{12}$ is diagonizable and satisfies the
relation $t_{12}(t_{12} - \la(m+1))(t_{12} - 2 \la) =0$.
Choosing $2\la \in \{ \alpha,\beta \}$ and $m = (\alpha+\beta)/\la -3$ we get that $t_{12}$ satisfies the relation $t_{12}(t_{12}^2 - (\alpha+\beta)t_{12} + \alpha \beta) = 0$ whence we
get two irreducible representations of $\mathfrak{V}_n$.

We now construct matrix models for the remaining
representations. From the study of the quantum
construction in section \ref{sect:quantumLS} we get an hint
about their restriction to $\mathfrak{S}_4$. Computing
the plethysm of $V(\varpi_2)^{\otimes 4}$ for e.g. type $A_8$  we get table \ref{tab:plethysm} 
which implies that the restriction of $M(\varpi_3+\varpi_5)$ to $\mathfrak{S}_4$ has isomorphism type $[3,1]+[2,1,1]$ while
the restriction of $M(\varpi_1+\varpi_2+\varpi_5)$ has type $[3,1]+[2,2]+[2,1,1]$ (at least for $n=9$).

Let us then introduce the matrices

$$
A_1 =
\left(\begin{array}{ccc}
-1 & 0 & 0 \\ 1 & 1 & 0 \\ 0 & 0 & 1
\end{array}\right)
A_2 =
\left(\begin{array}{ccc}
1 & 1 & 0 \\ 0 & -1 & 0 \\ 0 & 1 & 1
\end{array}\right)
A_3 =
\left(\begin{array}{ccc}
1 & 0 & 0 \\ 0 & 1 & 1 \\ 0 & 0 & -1
\end{array}\right)
$$
and
$$
B_1 =\diag(-1,1) = B_3,
B_2 = \frac{1}{2} \left(\begin{array}{cc}
1 & 3  \\ 1 & -1\\
\end{array}\right)
$$
Then $s_i \mapsto A_i$ defines the representation of $\mathfrak{S}_4$ associated with the partition $[3,1]$,
while $s_i \mapsto B_i$ defines the representation associated with the partition $[2,2]$,

We first consider the representation of $\mathfrak{S}_4$ associated to $[3,1]+[2,1,1]$, given by $s_i \mapsto \diag(A_i,-A_i)$. We set $
t_{12} \mapsto \frac{\alpha}{2} (s_1+1) +p_{12}$
with
$$p_{12} =    \left(\begin{array}{cccccc}
0 & 0 & \frac{2\alpha-\beta}{4}& \frac{\beta(\beta-2\alpha)}{16} & 0 & 0 \\
0 & 0 & \frac{\beta-2 \alpha}{2}& \frac{\beta(2\alpha-\beta)}{8} & 0 & 0 \\
0 & 0 & -2 & \beta/2 & 0 & 0 \\
0 & 0 & 1 & -\beta/4 & 0 & 0 \\
0 & 0 & 0 & 0 & 0 & 0  \\
\end{array} \right)
$$
The endomorphism $p_{12}$ has rank $1$ and nontrivial
eigenvalue $\beta - \alpha$. This provides an irreducible representation of dimension $6$ of $\mathfrak{V}_4$, for which $t_{12}$ has eigenvalues $0,0,0,\alpha,\alpha,\beta$. Another one is obtained by exchanging $\alpha$ and $\beta$. \footnote{This representation of $\mathfrak{S}_4$
is its hyperoctahedral representation as a Coxeter group of type $A_3$, see \cite{ARRREFL}. However, on the corresponding easy matrix model we were unable to find a nicer description of $t_{12}$, so we prefer this one where irreducibility is obvious.}

We now construct a 8-dimensional representation. We set $s_i \mapsto \diag(A_i,B_i,-A_i)$ and
$$
t_{12} \mapsto 
\left(\begin{array}{cccccccc}
0 & 0 & 0 & 0 & 0 & 0 & 0 & 0 \\
2c & 4c & c & 0 & 2c^2+2a &a&0&0 \\
0 & 0 & 2c & 0 &0 & -2a & 0 & 0 \\
0 & 0 & 0 & 0 & 0 & 0 & 0 & 0 \\
1 & 2 & 1 & 0 & 4c & 0 & 0 & 0 \\
0 & 0 & -2 & 0 & 0 & 6c & 0 & 0 \\
0 & 0 & 1 & 0 & 0 & -3c& 0 & 0 \\
0 & 0 & 0 & 0 & 0 & 0 & 0 & 0 \\
\end{array}\right)$$
It is readily checked that, $t_{ij} = w t_{12}w^{-1}$
does not depend on the choice of $w \in \mathfrak{S}_4$ such that $w(\{1,2\}) = \{i,j\}$,
and that this provides a representation of $\mathfrak{B}_4$. One then checks that the other relations of $\mathfrak{V}_4$ are satisfied for $\alpha+\beta = 8c$
and $\alpha\beta = 12c^2-4a$,
that is $\alpha =4c+2\sqrt{c^2+a}$
and $\beta =4c-2\sqrt{c^2+a}$. The representation
is clearly irreducible (look at the restriction
to $\mathfrak{S}_4$ and check that the isotopic components
are mapped to each other by $t_{12}$). The spectrum is
easily checked to be $0,0,0,0,\alpha,\alpha,\beta,\beta$, and the representation is clearly irreducible for generic values of $a,c \neq 0$.

\subsection{Freeness for the case $n=3$}

We set $x_i = t_{jk}$ and $z_i = (j,  k)$ whenever $\{i,j,k\}=\{1,2,3\}$. From the defining relation $x_1(x_2+x_3)=(x_2+x_3)x_1$ we get
$x_1(x_2+x_3)z_3=(x_2+x_3)x_1z_3$
hence
$$
\begin{array}{clcl}
&x_1x_2z_3+ x_1 (x_3z_3) &=& x_2x_1z_3 + x_3(x_1z_3) \\
\Leftrightarrow &x_1x_2z_3+ x_1 x_3 &=& x_2x_1z_3 + (x_3z_3)x_2 \\
\Leftrightarrow &x_1x_2z_3+ x_1 x_3 &=& x_2x_1z_3 + x_3x_2 \\
\Leftrightarrow &(x_1x_2-x_2x_1)z_3 &=&  x_3x_2- x_1 x_3 \\
\end{array}$$
that is $[x_1,x_2]z_3 =  x_3x_2- x_1 x_3$. Multiplying on the right by $x_3$ yields 
$[x_1,x_2]x_3 =  x_3x_2x_3- x_1 x_3^2$
hence $x_1x_2x_3 - x_2x_1x_3 = x_3x_2x_3 - x_1x_3^2$.
and

$[x_1,x_2]x_3 =  x_3x_2x_3- x_1 x_3^2 = x_3(x_2+x_1)x_3 - x_3x_1x_3 - x_1 x_3^2 
= (x_2+x_1)x_3^2 - x_3x_1x_3 - x_1 x_3^2 
= x_2x_3^2 - x_3x_1x_3  = (\alpha+\beta)x_2x_3 - \frac{\alpha \beta}{2}x_2(1+z_3)-x_3x_1x_3  $

Also note that $x_1x_2x_1 =x_1x_2z_1x_1 = x_1 z_1 x_3 x_1 = x_1x_3x_1$.

 Let
$$\mathcal{B}(3) = \{ 1,
x_1,
x_2,
x_3,
z_1,
z_2,
z_3,
x_1x_2,
x_1x_3,
x_1z_2,
x_1z_3,
x_2x_1,
x_2x_3,
x_2z_1,
x_2z_3,
x_3z_1,
x_3z_2,
z_1z_2,
z_1z_3,
x_1x_2x_3\}$$
We let
$\mathcal{B}_0(3) = \mathcal{B}(3)\setminus \{ x_1x_2x_3 \}$.
We need to prove $b x_i \subset R \mathcal{B}(3)$ for $i\in \{ 1,2,3 \}$ and $b \in \mathcal{B}(3)$. For $b =1$ this clear. Let us consider the case $b = x_j$. The case $i=j$ is clear, so we assume $i \neq j$. We have $b x_i \in \mathcal{B}(3)$ for all $i$, except for $j = 3$. From $(x_3+x_2)x_1 = x_1 (x_2+x_3)$ we get
$x_3 x_1 = x_1x_2 + x_1 x_3 - x_2 x_1 \in R \mathcal{B}_0(3)$ and similarly from $(x_3+x_1)x_2=x_2(x_1+x_3$ we get $x_3x_2 \in R \mathcal{B}_0(3)$. The case $b = z_j$
is clear since $\mathfrak{S}_3 \in \mathcal{B}(3)$
and from the defining relations $z_j x_i = x_k z_j$ and $z_j x_j = x_j z_j= x_j$ for $\{i,j,k\}=\{1,2,3\}$. Before the next case we first prove the following lemma.

\begin{lemma} \ 
\begin{enumerate}
\item For all $i,j$, $x_ix_j \in R \mathcal{B}_0(3)$,
$x_iz_j \in R \mathcal{B}_0(3)$ and $\mathfrak{S}_3 \subset R\mathcal{B}_0(3)$.
\item For all $i,j$ we have $2 x_i x_j x_i \in R\mathcal{B}_0(3)$
\item For all $\sigma \in \mathfrak{S}_3$ we have $2 x_{\sigma(1)}x_{\sigma(2)}x_{\sigma(3)} \equiv 2 x_1x_2x_3 \mod R \mathcal{B}_0(3)$
\item $(\alpha+\beta) x_1x_2z_3 \in 2 x_1x_2x_3 + R \mathcal{B}_0(3)$
\end{enumerate}
\end{lemma}
\begin{proof} (1) is a direct consequence of the definition of $\mathcal{B}_0(3)$ and of the defining relations, except for $x_3 x_i \in R\mathcal{B}_0(3)$
for $i =1,2$. From $(x_3+x_2)x_1 = x_1 (x_2+x_3)$ we get
$x_3 x_1 = x_1x_2 + x_1 x_3 - x_2 x_1 \in R \mathcal{B}_0(3)$ and similarly from $(x_3+x_1)x_2=x_2(x_1+x_3$ we get $x_3x_2 \in R \mathcal{B}_0(3)$. 
 In order to prove (2) we establish the identity
$ 2x_1x_2x_1 + (\alpha\beta/2)x_3z_1 + (\alpha\beta/2)x_2z_1 - (\alpha+\beta)x_1x_3 - (\alpha+\beta)x_1x_2 + (\alpha\beta/2)x_3 + (\alpha\beta/2)x_2 $. 
From $x_1(x_2-x_3)x_1=0$ we get $0=2 x_1x_2x_1 -x_1(x_2+x_3)x_1=2 x_1x_2x_1 -x_1^2(x_2+x_3)$ and the identity follows by expanding $x_1^2$ and easy defining relations. Therefore $2x_1x_2x_1 \in \mathcal{B}_0(3)$.
Since the statement is $\mathfrak{S}_3$-symmetric this
proves (2). Now notice
$x_2x_1x_3 = x_2(x_1+x_3)x_3 - x_2x_3^2 =
(x_1+x_3)x_2x_3 - x_2x_3^2 = x_1x_2x_3 + x_3x_2x_3 - x_2x_3^2$ hence $2 x_2x_1x_3 \equiv 2 x_1x_2x_3 + R \mathcal{B}_0(3)$. Similarly we get $2 x_1x_3x_2 \equiv 2 x_1x_2x_3 + R \mathcal{B}_0(3)$ and this proves (3).
From the defining relations we get
$(\alpha+\beta)x_1x_2z_3 = x_1^2 x_2z_3 + (\alpha\beta/2)(1+z_1)x_2z_3$. Since $z_1x_2z_3 = z_1 z_3 x_2 = z_3z_2 x_2 = z_3 x_2 = x_1 z_3 R\mathcal{B}_0(3)$
we only need to prove $x_1^2 x_2z_3\in 2 x_1x_2x_3 + R\mathcal{B}_0(3)$. Now
$x_1^2x_2 z_3 = x_1^2(x_2+x_3)z_3 - x_1^2x_3z_3 
=x_1(x_2+x_3)x_1z_3 - x_1^2x_3$.
But we get $x_1^2 x_3 \in R\mathcal{B}_0(3)$ by expanding $x_1^2$ and we have $x_1(x_2+x_3)x_1 = x_1x_2x_1+x_1x_3x_1 = 2 x_1x_3x_1$ hence $x_1(x_2+x_3)x_1z_3 = 2 x_1x_3x_1z_3 = 2 x_1x_3z_3x_2 = 2 x_1x_3x_2 \equiv 1 x_1x_2x_3 \mod R\mathcal{B}_0(3)$
and this proves (4). 
\end{proof}

Now consider the case $b = x_j x_k$ with $j\neq k$.
If $i \in \{j,k\}$ then we get $2bx_i \in R\mathcal{B}_0(3)$ either by expanding $x_k^2$ or by part (2) of the lemma. If not, then there exists $\sigma\in \mathfrak{S}_3$
such that $x_j x_k x_i = x_{\sigma(1)}x_{\sigma(2)}x_{\sigma(3)}$ hence $2 bx_i \in 2 x_1x_2x_3 + 
R\mathcal{B}_0(3)$ by part (3) of the lemma.
Now consider the case $b= x_j z_k$. If $i =k$ it is clear, and if $i\not\in \{k,j\}$ we get
$bx_i = x_j z_k x_i = x_j^2z_k \in R\mathcal{B}_0(3)$
by expanding $x_j^2$. Therefore we can assupe $i=j$.
Then $bx_i = x_i z_k x_i = x_{\sigma(1)}x_{\sigma(2)}z_{\sigma(3)}$ for some $\sigma \in \mathfrak{S}_3$.
Therefore $(\alpha+\beta) b x_i \in 2 x_1x_2x_3 + R\mathcal{B}_0(3)$ by the lemma. Finally, if $b  \in \mathfrak{S}_3$ is a 3-cycle, then $b$ can be written
as $z_j z_i$ for some $j$ hence $bx_i = z_j z_i x_i = z_j x_i \in R\mathcal{B}_0(3)$. This proves $R\mathcal{B}_0(3) x_i \subset R\mathcal{B}(3)$ for all $i$.

Assume now that $2$ and $(\alpha+\beta)$ are invertible.
We need to prove $x_1x_2x_3.x_i \in R\mathcal{B}(3)$
for all $i$. By part (3) of the lemma we get that $x_1x_2x_3 \equiv x_k x_j x_i \mod R\mathcal{B}_0(3)$ for some $k,j$. Since $R\mathcal{B}_0(3)x_i \subset R\mathcal{B}(3)$ this implies
$x_1x_2x_3x_i \in x_k x_j x_i^2 + R\mathcal{B}(3)$.
But expanding $x_i^2$   
we get that $x_kx_jx_i^2 \in R\mathcal{B}(3)$ and this proves the claim.

\section{Structure on 3 strands and general properties}

In this section we first define $\GQ_n$ over $R=\Z[a^{\pm 1},b^{\pm 1},c^{\pm 1}]$, as in the introduction.
By definition, $\GQ_n$ is the quotient of $H_n$ by one
of the following two relations, that we will prove to be equivalent.

$$
\begin{array}{clcl}
(1) & s_1^{-1}s_2s_1 &=& s_1 s_2 s_1^{-1} - a^{-1} s_1 s_2 + a s_1 s_2^{-1} + a s_1^{-1}s_2 - a^3 s_1^{-1} s_2^{-1}
+ a^{-1} s_2 s_1 - a s_2 s_1^{-1} - a s_2^{-1} s_1 \\
& & & + a^3 s_2^{-1}s_1^{-1} + a^2 s_1^{-1} s_2^{-1} s_1
- a^2 s_1 s_2^{-1} s_1^{-1}  \\
(2) & s_1s_2s_1^2 s_2 &=& s_1^2 s_2 s_1^2 + a s_2 s_1^2 s_2 - a s_1 s_2^2 s_1
+ a^2 s_2^2 s_1 - a^2 s_2 s_1^2 + a^2 s_1 s_2^2 - a^2 s_1^2s_2 \\ & & & 
- a^3 s_2^2 + a^3 s_1^2 + a^4 s_2 - a^4 s_1
\end{array}
$$

In other words, it is defined as the quotient of $R B_n$
for $n \geq 3$ by the cubic relation $(s_1 -a)(s_1 -b)(s_1 -c)$ and either (1) or (2). The proof that these two relations are equivalent will be given below.

In this section, we prove the following.

\begin{theorem} {\ } \label{theo:structmodGQ3}
The natural morphism $Q_2 \to Q_3$ is injective, and
\begin{enumerate}
\item $\GQ_3$ is a free $R$-module of rank $20$.
\item We have
$\GQ_3 = \GQ_2 + \GQ_2 s_2 \GQ_2 +  \GQ_2 s_2^{-1} \GQ_2 + R s_2 s_1^{-1} s_2$
\item We have
$\GQ_3 = \GQ_2 + \GQ_2 s_2 \GQ_2 +  \GQ_2 s_2^{-1} \GQ_2 + R s_2^{-1} s_1 s_2^{-1}$
\end{enumerate}
where $Q_2$ is identified with the subalgebra of $Q_3$ generated by $s_1$.
\end{theorem}

Moreover we provide two $R$-bases, one of them being made of positive braids,
and both of them originating from explicit Gröbner bases of the defining
ideal that we determine. 

Notice that it is known (see \cite{LG}) that $\GQ_3 \otimes_R K$ has dimension $20$,
where $K$ is the field of fractions of $R$. Therefore, in order for
a collection of elements spanning $\GQ_3$ to be a $R$-basis, it is necessary and sufficient
that it has exactly 20 elements.

\subsection{The defining relation(s)}

In order to compare the two possible defining relations, notice that $s_1s_2s_1^2 s_2 = s_2 s_1 s_2 s_1 s_2$.

We let $\Phi^a : R B_3 \to M_2(R)$ denote the representation $U_{b,c}$ of section \ref{sect:H3},
and compute the image of each of these relations under $\Phi^a$, and
also the $R B_3$-module they generate inside $M_2(R)$. Note that $R$ is noetherian,
and therefore every $R$-submodule of $M_2(R) \simeq R^4$ has finite type.

We check that all the other components of $\Phi_{H_3}(r)$ are $0$ for $r \in \{ r_1,r_2 \}$.
We
get
$$
\Phi^a(r1) = \frac{(a-c)(a-b)(a^2+bc)}{abc} \left( \begin{array}{cc} 1 & \frac{b-c}{b} \\
\frac{b-c}{c} & -1 \\ \end{array} \right)
$$
{}
$$
\Phi^a(r2) = (a-b)(a-c)(a^2+bc) \left( \begin{array}{cc} b-c & -c \\
-b & c-b \\ \end{array} \right)
$$
We consider the natural projection $p : H_3 \to \mathcal{H}_3(b,c)$, where
$\mathcal{H}_3(b,c)$ is the Iwahori-Hecke algebra of type $A_2$ defined over $R$
with parameters $b,c$, that is
the quotient of the group algebra $R B_3$ by the relations $(s_i-b)(s_i-c)=0$.
A straightforward computation in the standard basis of $\mathcal{H}_3(b,c)$ shows that
$$
p(r_1) = \frac{(a-c)(a-b)(a^2+bc)}{ab^2c^2}(s_1s_2 - s_2 s_1)
$$
{}
$$
p(r_2) = (a-c)(a-b)(a^2+bc)(s_1 - s_2)
$$
Now, note that $s_2(s_1 s_2 - s_2 s_1)s_1^{-1}s_2^{-1} = (s_2s_1s_2)s_1^{-1}s_2^{-1} - s_2s_2 s_1 s_1^{-1}s_2^{-1}
= s_1s_2s_1s_1^{-1}s_2^{-1} - s_2 = s_1 - s_2$. This proves that $p(r_1)$ and $p(r_2)$ generate the same ideal
inside $\mathcal{H}_3(b,c)$. From the following commutative diagram and the fact that $\Phi_{H_3}(r_1)$ and 
$\Phi_{H_3}(r_2)$ both belong to the ideal $Mat_2(R) \times R^2$ we get that $r_1$ and $r_2$ generate the same ideal.
$$
\xymatrix{
H_3 \ar@{^(->} [r] \ar@{->>}[d] & Mat_3(R)\times Mat_2(R)^3 \times R^3 \ar@<1ex>@{->>}[d]  \\
\mathcal{H}_3(b,c) \ar@{^(->} [r] & Mat_2(R) \times R^2  \ar@{^(->} [u] 
}
$$
Actually, by the same argument we get that the relation between $p(r_1)$ and $p(r_2)$ inside $\mathcal{H}_3(b,c)$
has to lift to the equality 
$$
s_2r_1 s_1^{-1}s_2^{-1} = ab^2c^2 r_2.
$$ 
Once this equality is guessed, the proof of its validity
can also be obtained computationally by applying directly $\Phi_{H_3}$ to it.

More generally, the membership problem for an element to belong to the defining ideal
of $\GQ_3$ is easily reduced, by this method to
\begin{itemize}
\item first checking that its image under $\Phi_{H_3}$ belongs to $(a-c)(a-b)(a^2+bc) M_2(R)$
\item and, if yes, whether its image inside $\mathcal{H}_3(b,c)$ divided by $(a-c)(a-b)(a^2+bc)$
belongs to the ideal of $\mathcal{H}_3(b,c)$ generated by $s_1 - s_2$.
\end{itemize}
For this last step, since $\mathcal{H}_3(b,c) = R B_3/(s_1 - b)(s_1-c)$ and the
quotient of $B_3$ by the relation $s_1s_2^{-1}$ is $B_3^{ab} \simeq \Z$,
we get  $\mathcal{H}_3(b,c)/(s_1 - s_2) = R[s]/(s-b)(s-c)$,
and this last membership problem is very easy to solve.

\subsection{Automorphisms of $\GQ_n$}

The braid group $B_n$ admits a group automorphism characterized by
the property that each generator $s_i$ is mapped to its inverse $s_i^{-1}$.
The image of a given braid is usually called its mirror image. It induces an automorphism of
$H_n$ as a $\Z$-algebra through via $s_i \mapsto s_i^{-1}$, $a \mapsto a^{-1}$,
$b \mapsto b^{-1}$, $c \mapsto c^{-1}$. We denote this automorphism by $\phi$. 

Now, the braid group $B_n$, as any other group, admits a group \emph{anti}automorphism
mapping every element to its inverse. It induces an \emph{anti}automorphism of $H_n$
as a $\Z$-algebra that maps $s_i, a,b,c$ to the same images as $\phi$. We denote this
antiautomorphism by $\psi$. Note that $\phi \circ \psi = \psi \circ \phi$ is a \emph{$R$-algebra} antiautomorphism
of $H_n$ which maps $s_i$ to $s_i$.

Direct computation (either by hand or using a computer implementation
of the injective morphism $\Phi_{H_3}$) proves that 
$\phi(r1) = a^{-2} r1$
and $\psi(r1) = -a^{-2}r1$.
Therefore $\phi \circ \psi(r1) = \phi( -a^{-2} r1) = - a^2 \phi(r1) = - r_1$.

This proves that $\phi$ induces a $\Z$-algebra automorphism of $\GQ_n$,
that $\psi$ induces a $\Z$-algebra anti-automorphism of $\GQ_n$,
and that $\phi \circ \psi = \psi \circ \phi$ induces a $R$-algebra anti-automorphism
of $\GQ_n$.

\subsection{A Gröbner basis with positive words}

Using the GBNP package of GAP4 (see \cite{GBNP}) on specialized rational values of $a,b,c$,
 we guess a (noncommutative) Gröbner basis for $\GQ_3$. The rewriting system corresponding to it is
the following one. Here and later on, we have used for concision the convention that the empty word is denoted $\emptyset$ and the Artin generators $s_i$, $s_i^{-1}$ are denoted $i$ and $\bar{i}$, respectively.

$$
\begin{array}{clcl}
(1) & 111 & \leadsto & \Sigma_1. 11 - \Sigma_2. 1 + \Sigma_3. \emptyset \\
(2) & 222 & \leadsto & \Sigma_1. 22 - \Sigma_2. 2 + \Sigma_3 .\emptyset \\
(3) & 212 & \leadsto & 121 \\
(4) & 1211 2 & \leadsto & 11 2 11 + a. 2 11 2 - a. 1 22 1
+ a^2 .22 1 - a^2. 2 11 \\ & & & + a^2. 1 22 - a^2. 112 
- a^3. 22 + a^3. 11 + a^4. 2 - a^4 .1 \\
(5) & 21121 & \leadsto & 11211 + a.2112 - a. 1221 + a^2. 221- a^2.211\\ & & & +a^2.122-a^2.112-a^3.22+a^3.11+a^4.2 - a^4.1 \\
(6) & 12211 & \leadsto & 11221 + a.2211 +(a^2+\Sigma_2)a^{-1}.1211 - a.1122 - (a^2+\Sigma_2)a^{-1}.1121 \\ & & & - a^2.221 - (a^2+\Sigma_2).211+a^2.122 + (a^2+\Sigma_2).112\\ & & & +(a^2+\Sigma_2)a.21 - (a^2+\Sigma_2)a.12 \\
(7) & 21122 & \leadsto & 11221 + \Sigma_1.2112 - \Sigma_2.1221 + \Sigma_2.221 - \Sigma_2.211 \\
(8) & 22112 & \leadsto & 12211 + \Sigma_1.2112 - \Sigma_1.1221 + \Sigma_2.122 - \Sigma_2.112
\end{array}
$$

Notice that, inside $B_3$, we have $21121 = 21212 = 12112$, whence the validity of (5) is equivalent
to the validity of (4). Using $\Phi_{H_3}$ we check that the relations (7) and (8) are actually true inside $H_3$.
Relation (6) is mapped inside $(a^2+bc)(a-b)(a-c)M_2(R)$, and its image inside $\mathcal{H}_3(b,c)$
is equal to $(a^2+bc)(a-b)(a-c)a^{-1}(s_1 s_2 - s_2 s_1)$. Therefore it is valid inside $\GQ_3$, and
actually could be taken as a defining relation, too.

We then check (by computer) that the set of (positive) words avoiding the patterns 
$$111, \ 222, \ 212, \ 12112, \ 21121, \ 12211, \ 21122, \ 22112$$
is finite, and has exactly 20 elements. We also check that it provides a generating set
of $\GQ_3$ (by using the rewriting system described above). Actually, this
proves that $\GQ_3$ is spanned by these 20 elements even if it
is defined over the smaller ring $\Z[b,c,a,a^{-1}]$ using the relation $r_2$.

This is remarkable because $H_3$ is \emph{not}
finitely generated if defined over $\Z[a,b,c]$ (see \cite{CYCLO}). 

$$
\begin{array}{|c|c|c|c|c|}
\hline
\emptyset & 1 & 2 & 11 & 12  \\ 
21 & 22 & 112 & 121 & 122  \\ 
211 & 221 & 1121 & 1122 & 1211  \\ 
1221 & 2112 & 2211 & 11211 & 11221  \\ 
\hline
\end{array}
$$

\subsection{Two Gröbner bases with signed words}

We now construct two rewriting systems on \emph{signed words}.
The procedure is similar, as we use GBNP to find Gröbner basis on specialisations,
and then the previously described algorithm to check the validity of the relations inside $\GQ_3$.
We use two different orderings on the signed generators to find the Gröbner basis. The first one is $1 < 2 < \bar{1} < \bar{2}$
and is described in table \ref{table:GQ3grobsigned1}, 
the second one is $1 < \bar{1} < 2 < \bar{2}$ 
and is described in table \ref{table:GQ3grobsigned2}. 
In these tables, we used the conventions $u = a+b+c$, $v = ab+ac+bc$ and $w = abc$.

For the first ordering, we avoid the patterns
$$
1\bar{1}, 22, 2\bar{2}, \bar{1}1, \bar{2}2, 212, 2\bar{1}\bar{2}, \bar{1}21, \
\bar{1}\bar{2}1, \bar{2}12, \bar{2}\bar{1}2, \bar{2}\bar{1}\bar{2}, 11, \bar{1\
}\bar{1}, \bar{2}\bar{2}, \bar{1}\bar{2}\bar{1}, \bar{2}1\bar{2}, 121\bar{2}, \
12\bar{1}2, 21\bar{2}1, 21\bar{2}\bar{1}, 2\bar{1}2\bar{1}, \bar{1}2\bar{1}2
$$
and we get the following basis
$$
\begin{array}{|c|c|c|c|c|}
\hline
\emptyset & 1 & 2&  \bar{1}&  \bar{2}\\
 12 & 1\bar{2} & 21 & 2\bar{1} & \bar{1}2 \\
 \bar{1}\bar{2} & \bar{2}1 & \bar{2}\bar{1} & 121 & 12\bar{1} \\
 1\bar{2}1 & 1\bar{2}\bar{1} & 21\bar{2} & 2\bar{1}2 & \bar{1}2\bar{1}  \\
 \hline
 \end{array}
$$

For the second ordering, we avoid the patterns
$$
1\bar{1}, \bar{1}1, 22, 2\bar{2}, \bar{2}2, \bar{1}\bar{2}1, 212, 21\bar{2}, 2\
\bar{1}\bar{2}, \bar{2}12, \bar{2}\bar{1}2, \bar{2}\bar{1}\bar{2}, 11, \bar{1}\
\bar{1}, \bar{2}\bar{2}, \bar{1}\bar{2}\bar{1}, \bar{2}1\bar{2}, 12\bar{1}2, \
\bar{1}2\bar{1}2, 2\bar{1}21, 2\bar{1}2\bar{1}
$$
and we get the following basis
$$
\begin{array}{|c|c|c|c|c|}
\hline
\emptyset & 1 & \bar{1} & 2 & \bar{2} \\ 12 & 1\bar{2} & \bar{1}2 & \bar{1}\bar{2} & 21 \\
 2\bar{1} & \bar{2}1 & \bar{2}\bar{1} & 121 & 12\bar{1} \\ 1\bar{2}1 & 1\bar{2}\bar{1}
  & \bar{1}21 & \bar{1}2\bar{1} & 2\bar{1}2 \\
  \hline
  \end{array}
$$ 
We notice that these two distinct collections of signed words represent exactly the same collection
of elements inside the braid group $B_3$, since $\bar{1}21 = 21\bar{2}$ in $B_3$. In particular they provide the same basis
of $\GQ_3$.

\begin{table}
$$
\begin{array}{|clcl|clcl|}
\hline
(1) & 1\bar{1} & \leadsto & \emptyset & (2) &
22 & \leadsto & w.\bar{2}+u.2-v.\emptyset \\ (3) &
2\bar{2} & \leadsto & \emptyset& 
(4) & \bar{1}1 & \leadsto & \emptyset \\ (5) &
\bar{2}2 & \leadsto & \emptyset & (6) & 
212 & \leadsto & 121 \\  
(7) & \bar{1}21 & \leadsto & 21\bar{2} & (8) &
\bar{1}\bar{2}1 & \leadsto & 2\bar{1}\bar{2} \\ (9) &
\bar{2}12 & \leadsto & 12\bar{1}&
(10) & \bar{2}\bar{1}2 & \leadsto & 1\bar{2}\bar{1} \\ (11) &
\bar{2}\bar{1}\bar{2} & \leadsto & \bar{1}\bar{2}\bar{1} & (12) &
11 & \leadsto & w.\bar{1}+u.1-v.\emptyset\\ 
(13) & \bar{1}\bar{1} & \leadsto & v/w.\bar{1}+w^{-1}.1-u/w.\emptyset & (14) &
\bar{2}\bar{2} & \leadsto & v/w.\bar{2}+w^{-1}.2-u/w.\emptyset \\ (15) & 
121\bar{2} & \leadsto & 21 &
(16) & 12\bar{1}2 & \leadsto & w.\bar{2}1\bar{2}+u.12\bar{1}-v.\bar{2}1 \\ 
(17) & 21\bar{2}1 & \leadsto & w.\bar{1}2\bar{1}+u.21\bar{2}-v.\bar{1}2 & 
(18) & 21\bar{2}\bar{1} & \leadsto & \bar{1}2\\ 
\hline
\end{array}
$$
{}$$
\begin{array}{|clcl|}
\hline
(19) & 2\bar{1}\bar{2} & \leadsto & a^{-2}.21\bar{2}+1\bar{2}\bar{1}-a^{-2}.12\bar{1}+(-a).\bar{2}\bar{1}+a^{-1}.\bar{2}1+a.\bar{1}\bar{2}-a^{-1}.\bar{1}2+a^{-1}.2\bar{1}-a^{-3}.21\\ & & & -a^{-1}.1\bar{2}+a^{-3}.12\\ 
(20) & \bar{1}\bar{2}\bar{1} & \leadsto & a^{-2}.\bar{1}2\bar{1}-(b+c)/w.1\bar{2}\bar{1}-(1/(wa)).1\bar{2}1+(b+c)/(wa^2).12\bar{1}\\ & & & +(1/(wa^3)).121+v/w.\bar{2}\bar{1}+(1/w).\bar{2}1+a^{-1}.\bar{1}\bar{2}-a^{-3}.\bar{1}2\\ & & & -v/(wa^2).2\bar{1}-(wa^2)^{-1}.21+u/(wa).1\bar{2}\\ & & & -u/(wa^3).12-(a^2+v)/(wa).\bar{2}\\ & & & +(a^2+v)/(wa^3).2\\ 
(21) & \bar{2}1\bar{2} & \leadsto & a^2.\bar{1}\bar{2}\bar{1}+(bc)^{-1}.2\bar{1}2+(a(b+c))/(bc).1\bar{2}\bar{1}-(b+c)/w.12\bar{1}-(wa)^{-1}.121\\ & & & -(va)/(bc).\bar{2}\bar{1}+v/w.\bar{2}1-a.\bar{1}\bar{2}-u/(bc).\bar{1}2\\ & & & -a/(bc).2\bar{1}+w^{-1}.21+a^{-1}.1\bar{2}+(u/(wa)).12\\ & & & +((a^2+v)/bc).\bar{1}-(a^2+v)/(wa).1\\ 
(22) & 2\bar{1}2\bar{1} & \leadsto & (bc)^{-1}.21\bar{2}1+a^2.1\bar{2}\bar{1}\bar{1}+(-1).12\bar{1}\bar{1}-a^3.\bar{2}\bar{1}\bar{1}
+(a(a^2+bc)).\bar{1}\bar{2}\bar{1}+(-(a^2+bc)/a).\bar{1}2\bar{1}\\
 & & & +a.2\bar{1}\bar{1}+a^{-1}.2\bar{1}2+(-(a^2+v)/w).21\bar{2}+(b+c-2a).1\bar{2}\bar{1}+a^{-1}.1\bar{2}1\\ 
& & & -(b+c-a)/a^2).12\bar{1}-a^{-3}.121+a^2.\bar{2}\bar{1} +(-1).\bar{2}1+(-(a^2+bc)).\bar{1}\bar{2} \\
 & & &  +((a^3(b+c)+2a^2bc+b^2c^2)/(wa)).\bar{1}2\\
 & & & +(((b+c)a^3+(b+c)^2a^2+2w(b+c)+(bc)^2)/(wa)).2\bar{1}\\ & & & +( (a^2 + bc+v)/(wa)).21\\ & & & -(b+c)a^{-1}.1\bar{2}+(u/a^3).12+a.\bar{2}-(a+b)(a+c)v/w.\bar{1}\\ & &  & -(a^4+2a^3(b+c)+a^2(b^2+5bc+c^2)+2w(b+c)+b^2c^2)/(wa^2).2\\ & & & +(-(a^2+v)/w).1+((a+b)^2(a+c)^2/(wa)).\emptyset\\ 
(23) & \bar{1}2\bar{1}2 & \leadsto & a.\bar{1}2\bar{1}+a^{-1}.2\bar{1}2-a^{-1}.21\bar{2}-a.1\bar{2}\bar{1}+a^2.\bar{2}\bar{1}\\ & & & +v.\bar{1}\bar{2}+(uv/w).\bar{1}2+(-1).2\bar{1}+a^{-2}.21+1\bar{2}\\ & & & +(u/w).12+(-(a^2+v)/a).\bar{2}+(-(a+b)(a+c)v/w).\bar{1}\\ & & & +(-u(a^2+v)/(wa)).2\\ & & & +(-(a^2+v)/w).1+((a+b)^2(a+c)^2/(wa)).\emptyset\\ 
\hline
\end{array}
$$
\caption{Rewriting system for $\GQ_3$ from $1 < 2 < \bar{1} < \bar{2}$}
\label{table:GQ3grobsigned1}
\end{table}

\begin{table}
$$
\begin{array}{|clcl|clcl|}
\hline
(1) & 1\bar{1} & \leadsto & \emptyset &
(2) & \bar{1}1 & \leadsto & \emptyset\\ 
(3) & 22 & \leadsto & (w).\bar{2}+(u).2+(-v).\emptyset & 
(4) & 2\bar{2} & \leadsto & \emptyset\\ 
(5) &\bar{2}2 & \leadsto & \emptyset &
(6) & 212 & \leadsto & 121\\ 
(7) & 21\bar{2} & \leadsto & \bar{1}21 & (8) & 2\bar{1}\bar{2} & \leadsto & \bar{1}\bar{2}1\\ 
(9) & \bar{2}12 & \leadsto & 12\bar{1} & (10) & 
\bar{2}\bar{1}2 & \leadsto & 1\bar{2}\bar{1}\\ 
(11) & \bar{2}\bar{1}\bar{2} & \leadsto & \bar{1}\bar{2}\bar{1} & (12) & 11 & \leadsto & (w).\bar{1}+(u).1+(-v).\emptyset\\ 
(13) & \bar{1}\bar{1} & \leadsto & (v/w).\bar{1}+(1/w).1+(-u/w).\emptyset & (14) & \bar{2}\bar{2} & \leadsto & (v/w).\bar{2}+(1/w).2+(-u/w).\emptyset\\ 
(15) & 12\bar{1}2 & \leadsto & (w).\bar{2}1\bar{2}+(u).12\bar{1}+(-v).\bar{2}1 & & & & \\
\hline
\end{array}
$$
{}
$$
\begin{array}{|clcl|}
\hline
 &  & & \ \\
(16) & \bar{1}\bar{2}1 & \leadsto & a^{-2}.\bar{1}21+1\bar{2}\bar{1}-a^{-2}.12\bar{1}-a.\bar{2}\bar{1}+a^{-1}.\bar{2}1+a^{-1}.2\bar{1}-a^{-3}.21+a.\bar{1}\bar{2} \\
& & & -a^{-1}.\bar{1} 2 - a^{-1}. 1 \bar{2} + a^{-3}.12\\
(17) & \bar{1}\bar{2}\bar{1} & \leadsto & a^{-2}.\bar{1}2\bar{1}-(b+c)w^{-1}.1\bar{2}\bar{1}-(wa)^{-1}.1\bar{2}1+((b+c)/(wa^2)).12\bar{1}\\
 & & & +(1/(wa^3)).121+(v/w).\bar{2}\bar{1}+w^{-1}.\bar{2}
1-(v/(wa^2)).2\bar{1}-(1/(wa^2)).21+a^{-1}.\bar{1}\bar{2}\\ & & & -a^{-3}.\bar{1}2+(
u/(wa)).1\bar{2}-(u/(wa^3)).12-((a^2+v)/(wa)).\bar{2}+((a^2+v)/(wa^3)).2\\ 
(18) & \bar{2}1\bar{2} & \leadsto & (bc)^{-1}.2\bar{1}2+\bar{1}2\bar{1}-(bc)^{-1}.1\bar{2}1+((a^2+v)/w).\bar{2}1-((a^2+v)/w).2\bar{1}\\
 & & &-((a^2+v)/w).\bar{1}2+((a^2+v)/w).1\bar{2}-((a^2+v)/(bc)).\bar{2}+((a^2+v)/(wa)).2
\\ & &  &+((a^2+v)/(bc)).\bar{1}-((a^2+v)/(wa)).1\\ 
(19) & \bar{1}2\bar{1}2 & \leadsto & -(bc).\bar{1}2\bar{1}\bar{1}+((bc)/a^2).12\bar{1}\bar{1}+a^{-1}.2\bar{1}2+((a^2+v)/a).2\bar{1}\bar{1}+((a^2+v)/a).\bar{1}2\bar{1}\\
& & & -a.1\bar{2}\bar{1}-(v/a^3).12\bar{1}-a^{-3}.121+a^2.\bar{2}\bar{1}\\
& & & -((a^3(b+c)+a^2(b^2+c^2+4bc) + 2w(b+c) + (bc)^2)/(wa)).2\bar{1}-(u/w).21
+v.\bar{1}\bar{2}\\ & & & +((u(b+c))/(bc)).\bar{1}2-((a^2+v)).\bar{1}\bar{1}+1\bar{2}+(u(a^2+bc)/(wa^2)).12-((a^2+v)/a).\bar{2}\\ & & & +((a^2+v)/a^2).\emptyset\\ 
(20) & 2\bar{1}21 & \leadsto & 12\bar{1}2+u.\bar{1}21-u.12\bar{1}+
v.\bar{2}1-v.1\bar{2}\\ 
(21) & 2\bar{1}2\bar{1} & \leadsto & \bar{1}2\bar{1}2+v.\bar{2}\bar{1}+((uv)/w).2\bar{1}+(u/w).21-v.\bar{1}\bar{2}-((uv)/w).\bar{1}2-(u/w).12 \\
\hline
\end{array}
{}
$$
\caption{Rewriting system for $\GQ_3$ from $1 < \bar{1} < 2 <   \bar{2}$}
\label{table:GQ3grobsigned2}
\end{table}

\subsection{$\GQ_3$ as a $\GQ_2$-bimodule} \label{sect:GQ3asGQ2bimodule}
An immediate consequence of the basis found above
is that$$
\GQ_3 = \sum_{k=0}^2 \GQ_2 s_2^k \GQ_2  + R s_2 s_1^{-1}s_2
$$
as a $R$-module.
Let $M_1 = \sum_{k=0}^2 \GQ_2 s_2^k \GQ_2  \subset \GQ_3$.
From the basis with signed words described above, it
is clear that the bimodule $\GQ_3 /M_1$ is a free $R$-module of rank $1$
spanned over $R$ by the element $s_2 s_1^{-1}s_2$, and that $M_1$ is a free $R$-module
of rank $19$.
Since $s_1.s_2s_1^{-1}s_2 = 
(s_1s_2s_1^{-1})s_2 = 
s_2^{-1}s_1s_2s_2 = 
s_2^{-1}s_1(s_2^2) = w.s_2^{-1}s_1s_2^{-1}+u.(s_2^{-1}s_1s_2) - v.s_2^{-1}s_1
= w.s_2^{-1}s_1s_2^{-1}+u.s_1s_2s_1^{-1} - v.s_2^{-1}s_1 \equiv w.s_2^{-1}s_1s_2^{-1} \mod M_1$.
 Now relation (21) of table \ref{table:GQ3grobsigned1} states that $s_2^{-1} s_1 s_2^{-1} \equiv (bc)^{-1} s_2 s_1^{-1} s_2 \mod M_1$
hence $s_1. s_2s_1^{-1}s_2  \equiv a .s_2s_1^{-1}s_2 \mod M_1$.
Similarly, $s_2s_1^{-1}s_2.s_1 = s_2(s_1^{-1}s_2s_1) = 
s_2^2s_1s_2^{-1} = w. s_2^{-1}s_1s_2^{-1}  + u. s_2s_1s_2^{-1}  - v .s_1s_2^{-1} 
\equiv  w. s_2^{-1}s_1s_2^{-1}  \equiv  w(bc)^{-1}. s_2s_1^{-1}s_2 \equiv a.s_2s_1^{-1}s_2 \mod M_1$.
If $S_x$ denotes the $\GQ_2$-module defined by $s_1 \mapsto x$, this proves that  $\GQ_3/M_1 \simeq S_a \otimes S_a$
as a $\GQ_2$-bimodule. As a by-product we get the following alternative
descriptions of $\GQ_3$ as a $R$-module :
$$
\GQ_3 = \sum_{k=0}^2 \GQ_2 s_2^k \GQ_2  + R s_2^{-1} s_1s_2^{-1}
$$

Let $M_2 = \GQ_2 \subset M_1$. We know from the basis description that $M_2$ is a free $R$-module of rank $3$
such that $M_1/M_2$ is a free $R$-module of rank $16$.
We consider $M_+ \subset M'_1 = M_1/M_2$ the bimodule generated by $s_2$.
It is spanned as a $R$-module by $2$, $12$, $\bar{1}2$, $2 \bar{1}$, $21$,
$121$, $12\bar{1}$, $\bar{1}2 1$, $\bar{1}2 \bar{1}$, and since they belong to the
basis of $\GQ_3$ and are disjoint from the part of it providing a basis of $M_2$,
we get that they provide a basis of $M_+$. Therefore, $M_+ \simeq \GQ_2 \otimes \GQ_2$ as a $\GQ_2$-bimodule. We then consider $M'_1/M_+$. It is a free
$R$-module of rank $7$, and is spanned by $\bar{2}$, $1\bar{2}$,
$\bar{1}\bar{2}$, $\bar{2} \bar{1}$, $\bar{2}1$, $1 \bar{2} 1$,  $1 \bar{2} \bar{1}$.
Clearly we have a natural surjective map $\GQ_2 \otimes \GQ_2 \to M'_1/M_+$,
defined by $x \otimes y \mapsto x \bar{2}y$. Inside the kernel we find
$\bar{1} \bar{2} 1 - 1 \bar{2}\bar{1} +a.\bar{2}\bar{1} - a^{-1}.\bar{2}1 - a.\bar{1}\bar{2} + a^{-1}.1 \bar{2}$
by relation (16) of table \ref{table:GQ3grobsigned2} and
$$\bar{1} \bar{2} \bar{1}+(b+c)w^{-1}. 1 \bar{2}\bar{1} + (wa)^{-1}. 1 \bar{2}1  - (v/w).\bar{2}\bar{1} -w^{-1}\bar{2}1 - a^{-1}\bar{1}\bar{2} - (u/(wa)).1\bar{2} + ((a^2+v)/(wa)) \bar{2}$$
by relation (17) of table \ref{table:GQ3grobsigned2}.

\subsection{Another basis and some special computations}
$$
\begin{array}{|c|c|c|c|c|}
\hline
\emptyset & \bar{2} & \bar{2}\bar{1} & \bar{2}1 & \bar{1} \\ 
\bar{1}\bar{2} & \bar{1}2 & 1 & 1\bar{2} & 12 \\ 
2 & 2\bar{1} & 21 & \bar{2}\bar{1}\bar{2} & 2\bar{1}\bar{2} \\ 
21\bar{2} & \bar{2}\bar{1}2 & 212 & \bar{2}1\bar{2} & \bar{1}2\bar{1} \\ 
\hline
\end{array}$$

We claim that the above list of 20 elements provides an alternative basis for $\GQ_3$. This can be checked using
$\Phi_{H_3}$ and express each element of this list as a linear combination over $K$ of one of the previously
obtained $R$-basis. Then, one checks that both the corresponding $K$-valued matrix and its inverse have all their
coefficients inside $R$, which proves our claim.

An interest of this basis is that the 19 first elements of the basis belong to
the $R$-submodule $u_1u_2+ \bar{2}u_1u_2 + u_2u_1\bar{2} + R. 121$, which will have
some importance in our computations. In particular, by expressing $\bar{2}1\bar{2}1\bar{2}$ in this
basis one gets the following identities 
\begin{equation} \label{eq2b12b12b}
\bar{2}1\bar{2}1\bar{2} \equiv \frac{bc}{a^2} .\bar{1}2\bar{1} + \frac{bc-a^2}{a^4bc}.121
\ \ \mod u_1u_2+ u_2u_1 + \bar{2}u_1u_2 + u_2u_1\bar{2} 
\end{equation}

\begin{equation}
2\bar{1}2\bar{1}2 \equiv a^2 .\bar{1}2\bar{1} \ \ \mod u_1u_2+ u_2u_1 \bar{2}u_1u_2 + u_2u_1\bar{2} + R. 121
\end{equation}

\subsection{Spanning sets for specific $\GQ_3$-modules}

\begin{lemma} \label{lem:quotA3dim5}
The quotient of $\GQ_3$ by the left ideal generated by $(2- a.\emptyset)1$ and $(2 - a.\emptyset)(\emptyset-a.\bar{1})$
is spanned as a $R$-module by $\emptyset, 1 , \bar{1}, 2, \bar{2}$.
\end{lemma}
\begin{proof}
Let us consider the left ideal $I$ generated by $(2- a.\emptyset)1$ and $(2 - a.\emptyset)(\emptyset-a.\bar{1})$,
and $V$ the $R$-submodule spanned by $\emptyset, 1 , \bar{1}, 2, \bar{2}$.
Since $\GQ_3$ is generated as a unital algebra by $1$ and $2$ it is sufficient to prove
$2.V \subset V + I$ and $1.V \subset V+I$. We start with the former.
We have $21 = (2 - a.\emptyset)1 + a.1 \in a.1 + I \subset V + I$.
Moreover 
$$2 \bar{1} = (2- a.\emptyset)\bar{1} + a.\bar{1}
= -a^{-1}(2- a.\emptyset)(\emptyset - a .\bar{1}) + a^{-1}.2 - \emptyset + a.\bar{1} \in V + I$$
hence $2.V \subset V+I$. Now, $$12 = \bar{2}212 = \bar{2}121 \in \bar{2}1(a.1 + I)
\subset \bar{2}(a.11 + I) \subset \bar{2}(V + I) \subset V+I.
$$
Finally, $1\bar{2} = \bar{2}21\bar{2} = \bar{2}\bar{1}21  \in \bar{2}\bar{1}(a.1 + I) = a.\bar{2} + I \subset V+I$
hence $1.V \subset V+I$ and this proves the claim.
\end{proof}

\section{Structure on 4 strands}

It was determined in \cite{LG} that $\dim \GQ_4 \otimes_R K  = 264$.
Moreover, there are models of all the irreducible representations
of $H_4$ which are defined over $R' = \Z[j][a^{\pm 1},b^{\pm 1},c^{\pm 1}]$,
and this provides an embedding $\Phi_{H_4}$ of $H_4$
inside a product of matrix algebras over $R'$. These models can be found in \cite{LG},
too.

\subsection{$\GQ_4$ as a $\GQ_3$-bimodule}

We denote $u_i = R + R s_i + R s_i^2 = R + R s_i + R s_i^2$ the $R$-subalgebra of $\GQ_4$ generated by $s_i$,
and $\GQ_3,\GQ_2$ the subalgebra of $\GQ_4$  generated by $s_1,s_2$ and $s_1$. Obviously $\GQ_2 = u_1$.
We first basically use the decompositions 
$$\GQ_3 = u_1 u_2 u_1 + u_2 u_1 u_2
= 
u_1 u_2 u_1 + R s_2 s_1^{-1} s_2 
= u_1 u_2 u_1 + R s_2^{-1} s_1 s_2^{-1}.$$
which follow from theorem \ref{theo:structmodGQ3}.

We set $\GQ_4^{(1)} = \GQ_3 u_3 \GQ_3$ and $\GQ_4^{(i+1)} = \GQ_4^{(i)} u_3 \GQ_3$.

\begin{lemme} \label{lemGQ4L1} {\ }
\begin{enumerate}
\item $u_3 u_2 u_3 \GQ_3 u_3 \subset \GQ_4^{(2)}$
\item $u_3 \GQ_3 u_3u_2 u_3  \subset \GQ_4^{(2)}$
\end{enumerate}
\end{lemme}
\begin{proof}
We first show (1). We have $\GQ_3 = u_1 u_2 u_1 + u_2 u_1 u_2$
hence $u_3 u_2 u_3 \GQ_3 u_3 \subset u_3 u_2 u_3 u_1 u_2 u_1 u_3 + u_3 u_2 u_3 u_2 u_1 u_2 u_3$.
From $u_3 u_2u_3 u_2 = u_3 u_2u_3 + u_2 u_3 u_2$ we get
$(u_3 u_2u_3 u_2) u_1 u_2 u_3 \subset u_3 u_2u_3 u_1 u_2 u_3 + \GQ_4^{(2)}$.
Since $u_3 u_2u_3 u_1 u_2 u_1 u_3 = u_3 u_2 u_3 u_1 u_2 u_3 u_1 \subset u_3 u_2u_3 u_1 u_2 u_3 \GQ_3$
it is sufficient to show that $u_3 u_2 u_3 u_1 u_2 u_3 \subset \GQ_4^{(2)}$. From $u_3 = R + R s_3 + R s_3^{-1}$ we
deduce $u_3 u_2 u_3 u_1 u_2 u_3 \subset \GQ_4^{(2)} + \sum_{\alpha \in \{-1,1 \}} s_3^{\alpha} u_2 u_3 u_1 u_2 u_3$.
But $s_3^{\alpha} u_2 u_3 u_1 u_2 u_3 = s_3^{\alpha}u_2 u_1 u_3 u_2 u_3 
\subset  s_3^{\alpha}u_2 u_1 s_3^{-\alpha}s_2^{\alpha}s_3^{-\alpha} + s_3^{\alpha} u_2 u_1 u_2 u_3 u_2
\subset (s_3^{\alpha} u_2 s_3^{-\alpha}) u_1 s_2^{\alpha} s_3^{-\alpha} + \GQ_4^{(2)} \subset \GQ_4^{(2)}$.
This proves (1), and (2) can be deduced from it, or can be proven similarly.

\end{proof}

\begin{proposition} \label{propGQ4P1} $\GQ_4 = \GQ_4^{(2)}$.
\end{proposition}
\begin{proof}
It is sufficient to show $u_3 \GQ_3 u_3 \GQ_3 u_3 \subset \GQ_4^{(2)}$, hence that
$s_3^{\alpha} \GQ_3 s_3^{\beta} \GQ_3 u_3 \subset \GQ_4^{(2)}$ pour $\alpha,\beta
\in \{-1,1\}$. We have $\GQ_3 = R s_2^{\alpha}s_1^{-\alpha}s_2^{\alpha} + u_1 u_2 u_1
=R s_2^{-\alpha}s_1^{\alpha}s_2^{-\alpha} + u_1 u_2 u_1$, hence
$$
\begin{array}{lclr}
s_3^{\alpha} \GQ_3 s_3^{\beta} \GQ_3 u_3 &\subset &s_3^{\alpha}s_2^{\alpha}s_1^{-\alpha}s_2^{\alpha}s_3^{\beta} \GQ_3 u_3 
+ s_3^{\alpha} u_1 u_2 u_1 s_3^{\beta} \GQ_3 u_3\\
 &\subset &s_3^{\alpha}s_2^{\alpha}s_1^{-\alpha}s_2^{\alpha}s_3^{\beta} \GQ_3 u_3 
+ \GQ_3s_3^{\alpha}  u_2  s_3^{\beta} \GQ_3 u_3\\
 &\subset &s_3^{\alpha}s_2^{\alpha}s_1^{-\alpha}s_2^{\alpha}s_3^{\beta} \GQ_3 u_3 
+ \GQ_4^{(2)}& \mbox{ by lemma \ref{lemGQ4L1}}\\
 &\subset &s_3^{\alpha}s_2^{\alpha}s_1^{-\alpha}s_2^{\alpha}s_3^{\beta} (s_2^{-\alpha}s_1^{\alpha}s_2^{-\alpha}) u_3 
+ \GQ_4^{(2)}& \mbox{ again after lemma \ref{lemGQ4L1}}\\
 &\subset &s_3^{\alpha}s_2^{\alpha}s_1^{-\alpha}(s_2^{\alpha}s_3^{\beta} s_2^{-\alpha})s_1^{\alpha}s_2^{-\alpha} u_3 
+ \GQ_4^{(2)}\\
 &\subset &s_3^{\alpha}s_2^{\alpha}s_1^{-\alpha}s_3^{-\alpha}s_2^{\beta} s_3^{\alpha})s_1^{\alpha}s_2^{-\alpha} u_3 
+ \GQ_4^{(2)}\\
 &\subset &(s_3^{\alpha}s_2^{\alpha}s_3^{-\alpha})s_1^{-\alpha}s_2^{\beta} s_1^{\alpha}s_3^{\alpha} s_2^{-\alpha} u_3 
+ \GQ_4^{(2)}\\
 &\subset &s_2^{-\alpha}s_3^{\alpha}(s_2^{\alpha}s_1^{-\alpha}s_2^{\beta} s_1^{\alpha})s_3^{\alpha} s_2^{-\alpha} u_3 
+ \GQ_4^{(2)}\\
 &\subset &\GQ_3u_3\GQ_3u_3 u_2 u_3 
+ \GQ_4^{(2)}\\
 &\subset &
 \GQ_4^{(2)}& \mbox{ by lemma \ref{lemGQ4L1}}\\
\end{array}
$$
and this proves the claim.
\end{proof}

\begin{lemme} {\ }\label{lemGQ4L2}
\begin{enumerate}
\item $s_3 \GQ_3 s_3^{-1} \subset \GQ_4^{(1)} + \GQ_3 s_3 s_2^{-1} s_3 \GQ_3$.
\item $s_3^{-1} \GQ_3 s_3 \subset \GQ_4^{(1)} + \GQ_3 s_3 s_2^{-1} s_3 \GQ_3$.
\end{enumerate}
\end{lemme}
\begin{proof}
We prove  (1). From $\GQ_3 = u_1 u_2 u_1 + R s_2 s_1^{-1} s_2$ we get
$s_3 \GQ_3 s_3^{-1} \subset u_1 s_3 u_2 s_3^{-1} u_1 + R s_3 s_2 s_1^{-1} s_2 s_3^{-1}
\subset \GQ_4^{(1)} + u_1 s_3 s_2^{-1} s_3 u_1 + R s_3 s_2 s_1^{-1} s_2 s_3^{-1}$,
hence it is sufficient to prove $s_3 s_2 s_1^{-1} s_2 s_3^{-1} \in \GQ_4^{(1)} + \GQ_3 s_3 s_2^{-1} s_3 \GQ_3$.
But $$ \begin{array}{llclcl}
 &s_3 s_2 s_1^{-1} s_2 s_3^{-1} &=& 
s_3 s_2 s_1^{-1} (s_2 s_3^{-1} s_2^{-1}) s_2 &=& 
s_3 s_2 s_1^{-1} s_3^{-1} s_2^{-1} s_3 s_2 \\ = &
(s_3 s_2 s_3^{-1})s_1^{-1}  s_2^{-1} s_3 s_2 &= &
s_2^{-1} s_3 (s_2 s_1^{-1}  s_2^{-1}) s_3 s_2 &=& 
s_2^{-1} s_3 s_1^{-1} s_2^{-1}  s_1 s_3 s_2 \\ = & 
s_2^{-1} s_1^{-1}s_3  s_2^{-1}  s_3  s_1s_2 &\in& \GQ_3 (s_3 s_2^{-1} s_3)\GQ_3\\
\end{array} $$
and this proves (1). The proof of (2) is similar.
\end{proof}

We set $w_0 = s_3 s_2 s_1^2 s_2 s_3$, $w_+ = s_3 s_2 s_1^{-1} s_2 s_3$,
$w_- = s_3^{-1} s_2^{-1} s_1 s_2^{-1} s_3^{-1}$.

\begin{lemme} \label{lemGQ4L3}
\begin{enumerate}
\item $w_0 \in R^{\times} w_+ + u_1 u_3 u_2 u_3 u_1$
\item $w_0^{-1} \in R^{\times} w_- + u_1 u_3 u_2 u_3 u_1$
\item $\forall x \in \GQ_3 \ \ w_+ x \in x w_+ + \GQ_3 u_3 u_2 u_3 \GQ_3.$
\item $\forall x \in \GQ_3 \ \ w_- x \in x w_- + \GQ_3 u_3 u_2 u_3 \GQ_3.$
\item $w_+ s_3, w_+ s_1 \in R w_+ + \GQ_3 u_3 u_2 u_3 \GQ_3$.
\item $w_- s_3, w_- s_1 \in R w_- + \GQ_3 u_3 u_2 u_3 \GQ_3$.
\end{enumerate}
\end{lemme}
\begin{proof}
(1) is a consequence of $s_1^2 \in R^{\times} s_1^{-1} + R s_1 + R$ and of the braid relations.
(2) is similar. (3) is deduced from  (1) and (2), and of the fact that $w_0$ commutes with $s_1$ et $s_2$. (4)
is similar. From $s_3^2 \in u_3 = R s_3 + R s_3^{-1} + R$
we deduce $w_+ s_3 \in R w^+ + \GQ_4^{(1)} + R s_3 s_2 s_1^{-1} s_2 s_3^{-1} $ and (5) is a consequence of lemma \ref{lemGQ4L2} (1).
(6) is similar.

\end{proof}

We now use another aspect of the defining relation of $\GQ_3$, under the form
$$
\begin{array}{lcl}
s_1^{-1} s_2 s_1 &=& 
\frac{-1}{a}  s_1 s_2  
+ a s_1 s_2^{-1}  + a s_1^{-1} s_2  - 
a^3  s_1^{-1} s_2^{-1}  +a^{-1} s_2 s_1  - a s_2 s_1^{-1}  -
a s_2^{-1} s_1  +  a^3 s_2^{-1} s_1^{-1} \\
& &  + a^2 s_1^{-1} 
 s_2^{-1} s_1  - a^2 s_1 s_2^{-1} s_1^{-1}  +  s_1 s_2 s_1^{-1}
 \end{array}
$$

In particular,
$
s_1^{-1} s_2 s_1 \equiv  s_1 s_2 s_1^{-1}  +  a^2 (s_1^{-1} 
 s_2^{-1} s_1  -  s_1 s_2^{-1} s_1^{-1})   \mod u_1 u_2 + u_2 u_1.
$

\begin{lemme} \label{lemGQ4L4}
We have $w_+ u_2 \subset R w_+ +\GQ_3 u_3 u_2 u_3 \GQ_3$,
and $w_- u_2  \subset R w_- +\GQ_3 u_3 u_2 u_3 \GQ_3$.
\end{lemme}
\begin{proof}
The second claim is deduced from the first one through the natural automorphisms, hence we can limit ourselves to considering the first one.
Since  $u_2$ is generated as a $R$-algebra by $s_2^{-1}$,
it sufficient to prove $w_+ s_2^{-1} \in R w_+ + \GQ_3 u_3 u_2 u_3 \GQ_3$.
Using the obvious shift morphism $u_1 u_2 u_1 \mapsto u_2 u_3 u_2$ characterized by $s_1 \mapsto s_2$, $s_2 \mapsto s_3$, we deduce from the preceding relation (inside $u_1 u_2 u_1$)
that
$
s_2^{-1} s_3 s_2 \equiv  s_2 s_3 s_2^{-1}  +  a^2 (s_2^{-1} 
 s_3^{-1} s_2  -  s_2 s_3^{-1} s_2^{-1})   \mod u_2 u_3 + u_3 u_2.
$
hence
$s_2 s_3 s_2^{-1} \in s_2^{-1} s_3 s_2 - a^2 
(s_2^{-1} 
 s_3^{-1} s_2  -  s_2 s_3^{-1} s_2^{-1}) + u_2 u_3 + u_3 u_2$.
We deduce from this
$w_+ s_2^{-1} = s_3 s_2 s_1^{-1} s_2 s_3 s_2^{-1} \in s_3 s_2 s_1^{-1}s_2^{-1} s_3 s_2 - a^2 
(s_3 s_2 s_1^{-1}s_2^{-1} 
 s_3^{-1} s_2  -  s_3 s_2 s_1^{-1}s_2 s_3^{-1} s_2^{-1}) + s_3 s_2 s_1^{-1}u_2 u_3 + s_3 s_2 s_1^{-1}u_3 u_2$.
We have $s_3 s_2 s_1^{-1} u_2 u_3 \subset R s_3 s_2 s_1^{-1} s_2 s_3 + u_3 u_1 u_2 u_1 u_3 + s_3 \GQ_3 s_3^{-1} + \GQ_4^{(1)}
 \subset R w_+ + \GQ_3 u_3 u_2 u_3 \GQ_3 $ after lemma \ref{lemGQ4L2}; 
Clearly $s_3 s_2 s_1^{-1} u_3 u_2 \subset \GQ_3 u_3 u_2 u_3 \GQ_3$
 and $s_3 (s_2 s_1^{-1} s_2^{-1}) s_3 s_2
 =s_3 s_1^{-1} s_2^{-1} s_1 s_3 s_2
 = s_1^{-1}s_3 s_2^{-1} s_3  s_1s_2
 \in \GQ_3 u_3 u_2 u_3 \GQ_3$. Finally, 
 $s_3 s_2 s_1^{-1}s_2^{-1} 
 s_3^{-1} s_2 \in  \GQ_3 u_3 u_2 u_3 \GQ_3$ after lemma \ref{lemGQ4L2}. 
\end{proof}

\begin{proposition} \label{propGQ4P2}
We have
$$
\GQ_4 = \GQ_3 + \GQ_3 s_3 \GQ_3 + \GQ_3 s_3^{-1} \GQ_3  + \GQ_3 s_3 s_2^{-1} s_3 \GQ_3   + R w_0 + R w_0^{-1}
$$
\end{proposition}
\begin{proof}
By proposition \ref{propGQ4P1} we know that $\GQ_4 = \GQ_4^{(2)} = \GQ_3 u_3 \GQ_3 u_3 \GQ_3$.
But $\GQ_3 = u_2 u_1 u_2 + u_1 u_2 u_1$, hence
$\GQ_4 = \GQ_3 u_3 u_2 u_1 u_2 u_3 \GQ_3
+ \GQ_3 u_3 u_1 u_2 u_1 u_3 \GQ_3
= \GQ_3 u_3 u_2 u_1 u_2 u_3 \GQ_3
+ \GQ_3 u_3  u_2  u_3 \GQ_3$.
Since $u_3 u_2 u_3 \subset R s_3 s_2^{-1} s_3 + u_2 u_3 u_2$
we have $\GQ_3 u_3 u_2 u_3 \GQ_3 =\GQ_3 + \GQ_3 s_3 \GQ_3 + \GQ_3 s_3^{-1} \GQ_3 + \GQ_3 s_3 s_2^{-1} s_3 \GQ_3$.
On the other hand, $$
u_3 (u_2 u_1 u_2) u_3 \subset \GQ_4^{(1)} + \sum_{\alpha,\beta \in \{ -1,1 \}}  s_3^{\alpha} u_2 u_1 u_2 s_3^{\beta}$$
hence, after lemma \ref{lemGQ4L2}, we get
$$
u_3 (u_2 u_1 u_2) u_3 \subset \GQ_3 u_3 u_2 u_3 \GQ_3 + \sum_{\alpha \in \{-1,1\} }  s_3^{\alpha} u_2 u_1 u_2 s_3^{\alpha}.$$
From $u_2 u_1 u_2  \subset R s_2 s_1^{-1} s_2 + u_1 u_2 u_1$ and 
$u_2 u_1 u_2  \subset R s_2^{-1} s_1 s_2^{-1} + u_1 u_2 u_1$ 
we deduce that $u_3 u_2 u_1 u_2 u_3 \subset  \GQ_3 u_3 u_2 u_3 \GQ_3 + R w_+ + R w_- $.
From lemma \ref{lemGQ4L3} (3) and (4) we get
$\GQ_3 u_3 u_2 u_1 u_2 u_3\GQ_3  \subset  \GQ_3 u_3 u_2 u_3 \GQ_3 +  w_+\GQ_3 +  w_-\GQ_3 $.
Finally, from lemma \ref{lemGQ4L3} (5), (6) and lemma \ref{lemGQ4L4} we get
$\GQ_3 u_3 u_2 u_1 u_2 u_3\GQ_3  \subset  \GQ_3 u_3 u_2 u_3 \GQ_3 +  R w_+ + R w_- $,
and this proves the claim.
\end{proof}

\subsection{$\GQ_3 u_3 \GQ_3$ as a $R$-module}

\begin{proposition} \label{prop:A3s3pmA3}
$\GQ_3s_3^{\pm 1}\GQ_3 = \GQ_3.s_3^{\pm 1}.F_1 + R.E.\{s_3^{\pm 1}s_2^{-1}s_1^{-1},s_3^{\pm 1}s_2s_1^{-1}s_2\}$ where 
$$F_1 = \{ 1,s_2,s_2s_1,s_2s_1^{-1},s_2^{-1},s_2^{-1}s_1 \}$$
and $E = \{ \emptyset, 2, \bar{2}, 12, 1\bar{2}, \bar{1}2, \bar{1}\bar{2}, 2\bar{1}2 \}$. 
In particular, $\GQ_3 s_3^{\pm 1} \GQ_3$
is spanned as a $R$-module by $136$ elements.
\end{proposition}

From $\GQ_3 = u_1u_2u_1+Rs_2s_1^{-1}s_2$ we get
$\GQ_3s_3\GQ_3 = \GQ_3 s_3 u_1u_2u_1 + \GQ_3 s_2s_1^{-1}s_2
= \GQ_3 s_3 u_2u_1 + \GQ_3 s_2s_1^{-1}s_2$.
Now, $s_3u_2u_1 = R.s_3.F_1 + Rs_3s_2^{-1}s_1^{-1}$
hence $\GQ_3 s_3u_2u_1 \subset \GQ_3.s_3.F_1+\GQ_3 s_3s_2^{-1}s_1^{-1}$.

We use that the defining relation can be rewritten
$$
a^2 (s_1-a) s_2^{-1} s_1^{-1}= (s_1-a) s_2 s_1^{-1} + a(a s_1^{-1}-1) s_2^{-1} s_1  + (a^{-1} -s_1^{-1})s_2 s_1  
+ (a s_1^{-1}- a^{-1} s_1)s_2   + a (s_1-a^2s_1^{-1}) s_2^{-1} 
  $$
whence
$$
a^2(1-a.\emptyset ).3\bar{2}\bar{1} \in \GQ_3.321 + \GQ_3.32\bar{1} + \GQ_3.32 + \GQ_3.3\bar{2} + \GQ_3.3\bar{2}1 \subset \GQ_3.3.F_1
$$
From the $R$-bases for $\GQ_3$ obtained above, we know that $\GQ_3$ is spanned
as a right $u_1$-module by $E = \{ \emptyset, 2, \bar{2}, 12, 1\bar{2}, \bar{1}2, \bar{1}\bar{2}, 2\bar{1}2 \}$.
Therefore by the relation above we get that $\GQ_3 s_3 s_2^{-1}s_1^{-1} \subset R.E.3\bar{2}\bar{1} + \GQ_3.3.F_1$
hence $\GQ_3s_3u_2u_1 \subset \GQ_3.s_3.F_1+R.E.3\bar{2}\bar{1}$.

We now notice that $s_1 s_3 s_2 s_1^{-1}s_2 =s_3  (s_1 s_2 s_1^{-1})s_2 =s_3  s_2^{-1} s_1 s_2^2$.
Since $s_2^2 = us_2-v+ws_2^{-1}$ we get
$s_1 s_3 s_2 s_1^{-1}s_2 =u s_3  (s_2^{-1} s_1s_2) - v s_3  s_2^{-1} s_1 + w s_3  s_2^{-1} s_1s_2^{-1}
=u s_3  s_1 s_2s_1^{-1} - v s_3  s_2^{-1} s_1 + w s_3  s_2^{-1} s_1s_2^{-1}
=u s_1 s_3   s_2s_1^{-1} - v s_3  s_2^{-1} s_1 + w s_3  s_2^{-1} s_1s_2^{-1}
\in \GQ_3.s_3   s_2s_1^{-1}+\GQ_3s_3  s_2^{-1} s_1 + w s_3  s_2^{-1} s_1s_2^{-1}
$ and this proves $(s_1-w).s_3s_2s_1^{-1}s_2 \in \GQ_3.s_3   s_2s_1^{-1}+\GQ_3s_3  s_2^{-1} s_1 \subset \GQ_3.F_1$.
Therefore, $\GQ_3.s_3s_2s_1^{-1}s_2 \subset R.E.s_3s_2s_1^{-1}s_2 + \GQ_3.F_1$
and this proves the claim for $\GQ_3 s_3 \GQ_3$. The proof for $\GQ_3 s_3^{-1} \GQ_3$ is the same.

\begin{lemma} \label{lem:A3s3pmF}
$\GQ_3.s_3.F + \GQ_3 s_3^{-1}.F = \GQ_3.F_2 +  R.E'.\{\bar{3} \bar{2},\bar{3} \bar{2}1 \} + R.E_0.\bar{3}2\bar{1}$
with 
$$
F_2 = \{ s_3, s_3^{-1}, s_3 s_2, s_3 s_2^{-1}, s_3^{-1} s_2, s_3 s_2 s_1,
s_3^{-1} s_2 s_1, s_3 s_2 s_1^{-1}, s_3 s_2^{-1} s_1 \}
$$
and $E' = (s_1s_2s_1)E(s_1s_2s_1)^{-1} = \{ \emptyset, 1, \bar{1},21,2 \bar{1},\bar{2} 1, \bar{2}\bar{1}, 1 \bar{2} 1 \}$,
$E_0 = \{ \emptyset, 1 ,\bar{1},2,\bar{2} \}$. In particular,
$\GQ_3.s_3.F + \GQ_3 s_3^{-1}.F$ is spanned as a $R$-module by 201 elements.
\end{lemma}
\begin{proof}
By definition, $\GQ_3.s_3.F + \GQ_3 s_3^{-1}.F$ is the (left) $\GQ_3$-module
generated by $$ \{ s_3,s_3s_2,s_3s_2s_1,s_3s_2s_1^{-1},s_3s_2^{-1},s_3s_2^{-1}s_1,
s_3^{-1},s_3^{-1}s_2,s_3^{-1}s_2s_1,s_3^{-1}s_2s_1^{-1},s_3^{-1}s_2^{-1},s_3^{-1}s_2^{-1}s_1 \}$$
{}
$$
= F_2 \cup \{ 
s_3^{-1}s_2s_1^{-1},s_3^{-1}s_2^{-1},s_3^{-1}s_2^{-1}s_1 \}
$$
Notice that $\GQ_3.s_3.F_1 \subset \GQ_3.F_2$.
Using the defining relation as above (and taking its image by conjugation under $s_1 s_2 s_3$),
we have
$$
a^2 (s_2-a) s_3^{-1} s_2^{-1}
= (s_2-a) s_3 s_2^{-1} + a(a s_2^{-1}-\emptyset) s_3^{-1} s_2  + (a^{-1} -s_2^{-1})s_3 s_2  
+ (a s_2^{-1}- a^{-1} s_2)s_3   + a (s_2-a^2s_2^{-1}) s_3^{-1}.
  $$
  Since $\GQ_3$ is spanned as a right $u_1$-module by $E$, it is spanned as a right $u_2$-module
  by $E' = (s_1s_2s_1)E(s_1s_2s_1)^{-1}$ hence
$\GQ_3 s_3^{-1} s_2^{-1} \subset \GQ_3.F_2 + R.E'.s_3^{-1}s_2^{-1}$. 
Similarly, $\GQ_3. s_3^{-1} s_2^{-1} s_1 \subset \GQ_3.F_2 + R.E'.s_3^{-1}s_2^{-1}s_1$.

Now, $1.\bar{3}2\bar{1} = \bar{3}(12\bar{1}) = (\bar{3}\bar{2})12$. Using the defining relation
we get
$$
\begin{array}{lcl}
a^2(2 - a.\emptyset) 1.\bar{3}2\bar{1} &=& a^2(2 - a.\emptyset)(\bar{3}\bar{2})12 \\
&=& (2-a.\emptyset) 3\bar{2}12 + a(a \bar{2}-\emptyset) \bar{3}212  + (a^{-1}.\emptyset -\bar{2}) 32 12 
+ (a \bar{2}- a^{-1} 2) 3 12  \\ & & + a (2-a^2\bar{2}) \bar{3} 12 \\
&=& (2-a.\emptyset) 3\bar{2}12 + a(a \bar{2}-\emptyset) \bar{3}121  + (a^{-1}.\emptyset -\bar{2}) 3121
+ (a \bar{2}- a^{-1} 2) 13 2  \\ & & + a (2-a^2\bar{2}) 1\bar{3} 2 \\
&\in & (2 - a.\emptyset)3\bar{2}12 + \GQ_3.F_2 \\ & \subset & \GQ_3.F_2 \\
\end{array}
$$
We now use that $1 2 \bar{1} \in \bar{1}21 + u_1.2 + u_1.\bar{2} + R.21+ a.2\bar{1} + u_1.\bar{2}1-a^3.\bar{2}\bar{1}$ to get
$$
\bar{3}2\bar{1} = \bar{1}\bar{3} (12\bar{1}) \in 
u_1.\bar{3}21 + u_1.\bar{3}2 +  u_1.\bar{3}\bar{2} +    a.\bar{1}\bar{3}2\bar{1} 
+  u_1.\bar{3}\bar{2}1-a^3.\bar{3}\bar{1}\bar{2}\bar{1}
$$
that is
$$
(\emptyset -  a.\bar{1} )\bar{3}2\bar{1} \in u_1.\bar{3}21 + u_1.\bar{3}2 +  u_1.\bar{3}\bar{2}+  u_1.\bar{3}\bar{2}1-a^3.\bar{3}\bar{1}\bar{2}\bar{1}
$$
and $(\emptyset -  a.\bar{1} )\bar{3}2\bar{1} \in  \GQ_3.F_2 + \GQ_3.\{\bar{3} \bar{2},\bar{3} \bar{2}1 \} -a^3.\bar{3}\bar{1}\bar{2}\bar{1}$.
Now, 
$\bar{3}(\bar{1}\bar{2}\bar{1}) = \bar{3}\bar{2}\bar{1}\bar{2}$ and, by the defining relation,
$$
a^2(2-a.\emptyset) \bar{3}\bar{1}\bar{2}\bar{1}
= a^2(2-a.\emptyset) \bar{3}\bar{2}\bar{1}\bar{2} \in \GQ_3.3.\GQ_3 + u_2.\bar{3}2\bar{1}\bar{2} + u_2.\bar{3}\bar{2}
$$
and, since $ \bar{3}2\bar{1}\bar{2}=\bar{3}\bar{1}\bar{2}1=\bar{1}\bar{3}\bar{2}1 \in \GQ_3.\bar{3}\bar{2}1$
we get 
$a^2(2-a.\emptyset) \bar{3}\bar{1}\bar{2}\bar{1} \in \GQ_3.F_2 + R.E'.\{\bar{3} \bar{2},\bar{3} \bar{2}1 \}$.
Therefore,  $(2-a.\emptyset)(\emptyset - a.\bar{1}) \bar{3}2 \bar{1} \in \GQ_3.F_2 + \GQ_3.\{\bar{3} \bar{2},\bar{3} \bar{2}1 \}
= \GQ_3.F_2 + R.E'.\{\bar{3} \bar{2},\bar{3} \bar{2}1 \}$.
From lemma \ref{lem:quotA3dim5} we deduce that 
$$
\GQ_3. \bar{3}2 \bar{1} \subset \GQ_3.F_2 + R.E'.\{\bar{3} \bar{2},\bar{3} \bar{2}1 \} + R.E_0.\bar{3}2\bar{1}
$$
and this proves the claim.
\end{proof}

\begin{lemma}  \label{lem:spanA3s3pmA3} 
$\GQ_3s_3.\GQ_3 + \GQ_3 s_3^{-1}.\GQ_3 = \GQ_3.s_3.F + \GQ_3 . s_3^{-1}.F + R.E.3 \bar{2}\bar{1}  +R.E.32\bar{1}2+ R.\bar{3}\bar{2}\bar{1}+R. \bar{3} 2 \bar{1} 2. $
In particular, it is spanned by 219 elements, and $\GQ_3 u_3 \GQ_3 = \GQ_3 + \GQ_3s_3.\GQ_3 + \GQ_3 s_3^{-1}.\GQ_3$ is spanned by 239 elements.
\end{lemma}

\begin{proof}
Note that, according to proposition \ref{prop:A3s3pmA3}, we have $\GQ_3.s_3.\GQ_3 =  \GQ_3 .s_3 .F +R.E.3 \bar{2}\bar{1} + R.E.32\bar{1}2
\subset \GQ_3.s_3. F+  \GQ_3 . s_3^{-1}.F + R.E.3 \bar{2}\bar{1} + R.E.32\bar{1}2$. In particular $I = \GQ_3.s_3. F+  \GQ_3 . s_3^{-1}.F + R.E.3 \bar{2}\bar{1} + R.E.32\bar{1}2$
is a left $\GQ_3$-module.
From the defining relation we get
$$
a^2 (2-a.\emptyset) \bar{3}\bar{2}\bar{1}
= (2-a.\emptyset) 3\bar{2}\bar{1} + a(a .\bar{2}-1) \bar{3}2\bar{1}  + (a^{-1}.\emptyset -\bar{2})32\bar{1}  
 + (a \bar{2}- a^{-1} 2)3\bar{1}   + a (2-a^2\bar{2}) \bar{3}\bar{1}
$$
hence $a^2 (2-a.\emptyset) \bar{3}\bar{2}\bar{1} \in \GQ_3.3.\GQ_3 + \GQ_3. \bar{3}.F \subset I$.
{}
Now, in the proof of proposition \ref{prop:A3s3pmA3} we proved that $(1 -a.\emptyset) \bar{3}\bar{2}\bar{1} \in \GQ_3. \bar{3}.F_1 \subset I$. Since
the quotient of $\GQ_3$ by its left ideal generated by $1 -a.\emptyset$ and $2  -a.\emptyset$ is obviously spanned by $\emptyset$ we get that
$\GQ_3. \bar{3}\bar{2}\bar{1} \subset R.\bar{3}\bar{2}\bar{1}  + I$. In particular $J = R.\bar{3}\bar{2}\bar{1}  + I$ is a $\GQ_3$-submodule.

From the defining relation we get $a(a.\bar{2}-\emptyset) \bar{3}2 \in a^2(2- a.\emptyset). \bar{3}\bar{2} + \GQ_3.3.\GQ_3 + \GQ_3.\bar{3}$
hence $a(a.\bar{2}-\emptyset) \bar{3}2\bar{1}2 \in a^2(2- a.\emptyset). \bar{3}\bar{2}\bar{1}2 + \GQ_3.3.\GQ_3 + \GQ_3.\bar{3}\bar{1}2$
that is 
$$a(a.\bar{2}-\emptyset) \bar{3}2\bar{1}2 \in a^2(2- a.\emptyset). \bar{3}1\bar{2}\bar{1} + \GQ_3.3.\GQ_3 + \GQ_3.\bar{3}2 \subset \GQ_3. \bar{3}\bar{2}\bar{1}+I = J.
$$
Similarly, from $(1- a.\emptyset)2\bar{1} \in a^2(1-a.\emptyset)\bar{2}\bar{1} - a(a.\bar{1}-\emptyset) \bar{2}1 - (a^{-1}.\emptyset - \bar{1}) 21 + u_1 u_2$
we get 
$$
\begin{array}{lcl}
(1- a.\emptyset)\bar{3}2\bar{1}2 &\in &a^2(1-a.\emptyset)\bar{3}\bar{2}\bar{1}2 - a(a.\bar{1}-\emptyset) \bar{3}\bar{2}12 
- (a^{-1}.\emptyset - \bar{1}) \bar{3}212 + u_1 \bar{3} u_2 \\
 &\in &a^2(1-a.\emptyset)\bar{3} 1\bar{2}\bar{1} - a(a.\bar{1}-\emptyset) \bar{3}12\bar{1}
- (a^{-1}.\emptyset - \bar{1}) \bar{3}121 + u_1 \bar{3} u_2 \\
 &\in &u_1 \bar{3} \bar{2}\bar{1} - u_1 \bar{3}2\bar{1}
- u_1 \bar{3}21 + u_1 \bar{3} u_2 \\
&\in &\GQ_3 \bar{3} \bar{2}\bar{1} + \GQ_3.\bar{3}.F_1 \subset J \\
\end{array}
$$
Since the quotient of $\GQ_3$ by its left ideal generated by $(1- a.\emptyset)$ and $a(a.\bar{2}-\emptyset)$
is spanned by the image of $\emptyset$, we get $\GQ_3.\bar{3}2\bar{1}2 \in R.\bar{3}2\bar{1}2 + J$
and this proves the claim.

\end{proof}

\subsection{Spanning $\GQ_4/\GQ_3u_3\GQ_3$ as a $R$-module - preliminaries}

By proposition \ref{propGQ4P2} we know that $\GQ_4 = \GQ_3 u_3 \GQ_3 + R w_0 + R w_0^{-1} + \GQ_3 .3 \bar{2} 3.\GQ_3$.

\subsubsection{Step 1 : $\GQ_3 u_3 \GQ_3 + u_2 3 \bar{2}3 u_2 = \GQ_3 u_3 \GQ_3 + R.3 \bar{2} 3$} 

First note that $2.3\bar{2}3 = (23\bar{2})3 = \bar{3}233 \in u_2u_3u_2 + R. 3 \bar{2} 3$ by theorem \ref{theo:structmodGQ3}.
Since $\GQ_3$ is spanned as a right $u_2$-module by $E'$,
we get that $\GQ_3 u_3 \GQ_3 + \GQ_3. 3 \bar{2}3 = \GQ_3 u_3 \GQ_3  + R.E'.3 \bar{2}3$.
By a similar argument we get $3\bar{2}3.2 \in u_2u_3u_2 + R. 3 \bar{2} 3$
hence $\GQ_3 u_3 \GQ_3+\GQ_3.3\bar{2}3.u_2  = \GQ_3 u_3 \GQ_3+ \GQ_3.3 \bar{2} 3 = \GQ_3 u_3 \GQ_3  + R.E'.3 \bar{2}3$.

\paragraph{suite}
ble
\subsubsection{Step 2 : $\GQ_3 u_3 \GQ_3 + u_1 3 \bar{2}3 u_1 = \GQ_3 u_3 \GQ_3 + u_1.x + x. u_1 +R .1x1  + R. 1 x \bar{1}$, $x = 3 \bar{2}3$} { \ } \\

While studying $\GQ_3$ as a $\GQ_2$-module in section \ref{sect:GQ3asGQ2bimodule}, the quotient module of $\GQ_2 u_2 \GQ_2$
by its submodule $\GQ_2+\GQ_2 s_2 \GQ_2$ has been determined under the name $M'_1/M_+$, and it was proven
to be generated by (images of) the seven elements
$\bar{2}$, $1\bar{2}$,
$\bar{1}\bar{2}$, $\bar{2} \bar{1}$, $\bar{2}1$, $1 \bar{2} 1$,  $1 \bar{2} \bar{1}$. Therefore,
$\GQ_3 u_3 \GQ_3 + u_1 3 \bar{2} 3 u_1$ is spanned by the already determined $219+20 = 239$ elements
spanning $\GQ_3 u_3 \GQ_3$ plus the 7 elements $3\bar{2}3$, $13\bar{2}3$,
$\bar{1}3\bar{2}3$, $3\bar{2}3 \bar{1}$, $3\bar{2}31$, $1 3\bar{2}3 1$,  $1 3\bar{2}3 \bar{1}$

\subsubsection{Step 3 : $u_2 31 \bar{2} \bar{1}3 \subset \GQ_3 u_3 \GQ_3 + R.31 \bar{2} \bar{1}3+ u_2.3 1 \bar{2} 1 3 + u_2.13\bar{2}3$} { \ } \\

By the study at the end of section \ref{sect:GQ3asGQ2bimodule} we know that
$\bar{1} \bar{2} \bar{1}\equiv -(b+c)w^{-1}. 1 \bar{2}\bar{1} - (wa)^{-1}. 1 \bar{2}1  + (v/w).\bar{2}\bar{1} +w^{-1}\bar{2}1 + a^{-1}\bar{1}\bar{2} + (u/(wa)).1\bar{2} - ((a^2+v)/(wa)) \bar{2}$
modulo $\GQ_2 + \GQ_2 s_2 \GQ_2$. Therefore
$3\bar{1} \bar{2} \bar{1}3 \equiv -(b+c)w^{-1}. 31 \bar{2}\bar{1}3 - (wa)^{-1}. 31 \bar{2}13  + (v/w).3\bar{2}3\bar{1} +w^{-1}3\bar{2}31 + a^{-1}\bar{1}3\bar{2}3 + (u/(wa)).31\bar{2}3 - ((a^2+v)/(wa)) 3\bar{2}3$ modulo $\GQ_3 u_3 \GQ_3$
and $123\bar{1} \bar{2} \bar{1}3 \equiv -(b+c)w^{-1}. 1231 \bar{2}\bar{1}3 - (wa)^{-1}. 1231 \bar{2}13  + (v/w).123\bar{2}3\bar{1} +w^{-1}.123\bar{2}31 + a^{-1}.12\bar{1}3\bar{2}3 + (u/(wa)).1231\bar{2}3 - ((a^2+v)/(wa)) 123\bar{2}3$ modulo $\GQ_3 u_3 \GQ_3$.

Since $23\bar{2}3 \equiv a. 3 \bar{2} 3 \mod \GQ_3 u_3 \GQ_3$, we have 
$1231 \bar{2}\bar{1}3 = 1213 \bar{2}\bar{1}3 = 2123 \bar{2}3\bar{1} \equiv
a.213 \bar{2}3\bar{1}$, 
$1231 \bar{2}13 = 1213 \bar{2}31= 2123 \bar{2}31 \equiv a.213 \bar{2}31$,
 $123\bar{2}3\bar{1} \equiv a.13\bar{2}3\bar{1}$, 
 $123\bar{2}31 \equiv a.13\bar{2}31$, 
$12\bar{1}3\bar{2}3 = \bar{2}123\bar{2}3 \equiv a.\bar{2}13\bar{2}3 $ 
and $1231\bar{2}3 = 1213\bar{2}3 = 2123\bar{2}3 \equiv  a.213\bar{2}3$  
modulo $\GQ_3 u_3 \GQ_3$.

It follows that
$123\bar{1} \bar{2} \bar{1}3 \equiv -(b+c)w^{-1}. a.213 \bar{2}3\bar{1} - w^{-1}.213 \bar{2}31  + (v/w).a.13\bar{2}3\bar{1} +w^{-1}.a.13\bar{2}31 + \bar{2}13\bar{2}3 + (u/w).213\bar{2}3 - ((a^2+v)/w) 13\bar{2}3$ modulo $\GQ_3 u_3 \GQ_3$. On the other hand, we have $123\bar{1}\bar{2}\bar{1} 3 = 12\bar{1}3\bar{2}3\bar{1}  = 
\bar{2}123\bar{2}3\bar{1}  \equiv a.\bar{2}13\bar{2}3\bar{1}$ modulo $\GQ_3 u_3 \GQ_3$. This yields to
$\bar{2}13\bar{2}3\bar{1} \equiv 
-(b+c)w^{-1}. 213 \bar{2}3\bar{1} - (aw)^{-1}.213 \bar{2}31  + (v/w).13\bar{2}3\bar{1} +w^{-1}.13\bar{2}31 + a^{-1}.\bar{2}13\bar{2}3 + (u/(wa)).213\bar{2}3 - ((a^2+v)/(aw)) 13\bar{2}3$ modulo $\GQ_3 u_3 \GQ_3$, which can be rephrased as $(\bar{2} + (b+c)w^{-1}. 2 - (v/w).\emptyset) 13\bar{2}3\bar{1}$ being congruent to 
$
 - (aw)^{-1}.213 \bar{2}31+  w^{-1}.13\bar{2}31 + a^{-1}.\bar{2}13\bar{2}3 + (u/(wa)).213\bar{2}3 - ((a^2+v)/(aw)) 13\bar{2}3$ modulo $\GQ_3 u_3 \GQ_3$.
 Now, $\bar{2} + (b+c)w^{-1}. 2 - (v/w).\emptyset  = w^{-1}.2.(2 - a.\emptyset)$
 hence
 $$
 2(2- a.\emptyset)31\bar{2}\bar{1}3
 \equiv  - a^{-1}.213 \bar{2}31 + 13\bar{2}31 + wa^{-1}.\bar{2}13\bar{2}3 + (u/a).213\bar{2}3 - ((a^2+v)/a) .13\bar{2}3
 $$
 modulo $\GQ_3 u_3 \GQ_3$ and
\begin{equation} \label{eq21x1bA}
 (2- a.\emptyset)31\bar{2}\bar{1}3
 \equiv  - a^{-1}.13 \bar{2}31 + \bar{2}13\bar{2}31 + wa^{-1}.\bar{2}\bar{2}13\bar{2}3 + (u/a).13\bar{2}3 - ((a^2+v)/a) .\bar{2}13\bar{2}3
\end{equation}
 modulo $\GQ_3 u_3 \GQ_3$ and this proves in particular that $u_2 31 \bar{2} \bar{1}3 \subset \GQ_3 u_3 \GQ_3 + R.31 \bar{2} \bar{1}3+ u_2.3 1 \bar{2} 1 3 + u_2.13\bar{2}3$.
 Applying the usual automorphisms and the previous reductions we deduce that $31 \bar{2} \bar{1}3u_2 \subset \GQ_3 u_3 \GQ_3 +R.31 \bar{2} \bar{1}3+  3 1 \bar{2} 1 3.u_2 + 3\bar{2}31.u_2 + u_1.3\bar{2}3.u_1$

\subsection{$\GQ_4/\GQ_3u_3\GQ_3$ as a $R$-module - computational description} 
\label{sect:A4tildegen}
We concentrate our attention on the $\GQ_3$-bimodule $\tilde{A}_4 = \GQ_4/\GQ_3 u_3 \GQ_3$.

\subsubsection{A convenient basis}
\label{sect:A4tildebasis}
We introduce the following list of $25$ elements of $\tilde{A}_4$. They will turn out to provide
a $R$-basis. We set $x = 3 \bar{2}3$ and $y = 1 \bar{2}1$.

$$
\begin{array}{|lcllcllcllcl|}
\hline
e_{1} & = & 3\bar{2}3 & e_{2} & = & 13\bar{2}3 & e_{3} & = & \bar{1}3\bar{2}3 & e_{4} & = & 213\bar{2}3 \\ 
e_{5} & = & 2\bar{1}3\bar{2}3 & e_{6} & = & \bar{2}13\bar{2}3 & e_{7} & = & \bar{2}\bar{1}3\bar{2}3 & e_{
8} & = & 3\bar{2}31 \\ 
e_{9} & = & 3\bar{2}3\bar{1} & e_{10} & = & 3\bar{2}312 & e_{11} & = & 3\bar{2}31\bar{2} & e_{
12} & = & 3\bar{2}3\bar{1}2 \\ 
e_{13} & = & 3\bar{2}3\bar{1}\bar{2} & e_{14} & = & 213\bar{2}31 & e_{15} & = & \bar{2}13\bar{2}31 & e_{
16} & = & 13\bar{2}312 \\ 
e_{17} & = & 13\bar{2}31\bar{2} & e_{18} & = & 13\bar{2}3\bar{1} & e_{19} & = & 13\bar{2}31 & e_{
20} & = & 213\bar{2}312 \\ 
e_{21} & = & 213\bar{2}31\bar{2} & e_{22} & = & 32\bar{1}23 & e_{23} & = & \bar{3}\bar{2}1\bar{2}\bar{3} & e_{
24} & = & 3\bar{2}31\bar{2}1 \\ 
e_{25} & = & 1\bar{2}13\bar{2}3 & & & & &&& \\ 
\hline
\end{array}
$$

\subsubsection{Description of $f = \Phi \circ \Psi$ on $\tilde{A}_4$}

We let $f$ denote the $R$-module automorphism induced by $\Phi \circ \Psi = \Psi \circ \Phi$ on $\tilde{A}_4$.
We have $f(3\bar{2}3) = 3 \bar{2}3$, and an immediate verification shows that $f(e_i) = e_{\sigma_f(i)}$
for all $i \not\in \{ 18, 21 \}$, with
$$\sigma_f = (2,8)(3,9)(4,10)(5,12)(6,11)(7,13)(14,16)(15,17)(24,25) \in \mathfrak{S}_{25}.
$$
It remains to compute $f(e_{18})$ and $f(e_{21})$.

We have $f(e_{18}) = \bar{1}x1 = 3 \bar{1}\bar{2}13$. By section \ref{sect:GQ3asGQ2bimodule}
(or by relation (16) of table \ref{table:GQ3grobsigned2}) we know that $\bar{1}\bar{2} 1 \equiv 1 \bar{2}\bar{1} - a.\bar{2}\bar{1} + a^{-1}.\bar{2}1 + a. \bar{1}\bar{2}
- a^{-1}.1\bar{2}$ modulo $M_+$. This implies, as in step 2, that
$\bar{1}x1 \equiv 1x\bar{1} - a.x\bar{1} + a^{-1}.x1 + a.\bar{1}x - a^{-1}.1x$ (mod. $\GQ_3 u_3 \GQ_3$).
This proves that
$$f(e_{18}) = e_{18} +a.(e_3 - e_9) + a^{-1}.(e_8 - e_2).
$$
which completes the explicit determination of $f$, except for $f(e_{21}) = \bar{2}1x12$. We will determine in section \ref{sect:u2u1xu1u2}
that
\begin{equation} \label{eqfe21}
f(e_{21}) = e_{21} + a.(e_{15}-e_{17}) + a^{-1}.(e_{16} - e_{14})
\end{equation}

 Note that this description provides
a square matrix of size 25 with coefficients in the subring $\Z[a^{\pm 1}]$ of $R$.

\subsubsection{Description of $u_2u_1xu_1+u_1xu_1u_2$}

We have $\bar{1} x 12 = f(21x\bar{1})$ and, from equation (\ref{eq21x1bA}) we get
$(2-a.\emptyset).1x\bar{1} = -a^{-1}. e_{19} + e_{15} + w a^{-1}. \bar{2}\bar{2}1x + (u/a).e_2 - (a^2+v)a^{-1}.e_6$.
Since $\bar{2}\bar{2}1x = w^{-1}21.x - uw^{-1}.1x + vw^{-1}.\bar{2}1x
= w^{-1}. e_4 - uw^{-1}.e_2 + vw^{-1}.e_6$ we get
$(2 - a.\emptyset)1x\bar{1} = -a^{-1} e_{19} + e_{15} + a^{-1}.e_4 - a.e_6$,
that is 
\begin{equation} \label{eq21x1b}
2 1 x \bar{1} = a.e_{18}  -a^{-1} e_{19} + e_{15} + a^{-1}.e_4 - a.e_6
\end{equation}
Then $\bar{1}x12 = f(2 1 x \bar{1}) = a.f(e_{18})  -a^{-1} e_{19} + e_{17} + a^{-1}.e_{10} - a.e_{11}$

We want to compute $\bar{1} x \bar{1}$. Following the indications of step 2, we use the results of section \ref{sect:GQ3asGQ2bimodule} (in particular relation (17) of table \ref{table:GQ3grobsigned2}) to expand $\bar{1} x \bar{1} = 3 (\bar{1}\bar{2} \bar{1}) 2$ and get
\begin{equation} \label{eq1bx1b}
\bar{1} x \bar{1} = -\frac{b+c}{w} e_{18} - \frac{1}{wa} e_{19} + \frac{v}{w} e_{9} + w^{-1}. e_{8} + a^{-1}.e_3 + \frac{u}{wa} e_2 - \frac{a^2+v}{wa} e_1
\end{equation}
From this and the identity $2 x = x 2 = a. x$ one readily gets
\begin{equation}
\bar{1} x \bar{1}2 = -\frac{b+c}{w} .1x\bar{1}2 - \frac{1}{wa} e_{16} + \frac{v}{w} e_{12} + w^{-1}. e_{10} + e_3 + \frac{u}{w} e_2 - \frac{a^2+v}{w} e_1
\end{equation}
From the identity $\bar{1}x1 \equiv 1x\bar{1} - a.x\bar{1} + a^{-1}.x1 + a.\bar{1}x - a^{-1}.1x$ obtained above, 
we get
$2\bar{1}x1 = 21x\bar{1} - a.2x\bar{1} + a^{-1}.2x1 + a.2\bar{1}x - a^{-1}.21x
= 21x\bar{1} - a^2.x\bar{1} + x1 + a.2\bar{1}x - a^{-1}.21x$ hence
\begin{equation}
2\bar{1}x1 = 21x\bar{1} - a^2 e_{9}+e_8 + a.e_5 -a^{-1}.e_4
\end{equation}
and $21x\bar{1}$ is known by (\ref{eq21x1b}).
Similarly, we get
\begin{equation}
\bar{2}\bar{1}x1 = \bar{2}1x\bar{1} -  e_{9}+a^{-2}.e_8 + a.e_7 -a^{-1}.e_6
\end{equation}

But $\bar{2} 1 x \bar{1}$ is not known yet. We get it as follows. From the description of $2 1 x \bar{1}$ in (\ref{eq21x1b}) we get, after using the cubic relation
a couple of times, that
$221x\bar{1} = a.21x\bar{1} - a^{-1}.e_{14}+e_{19} +(u/a).e_4 - \frac{v+a^2}{a} .e_2 + (w/a).e_6$.
Expanding $22 = u.2 - v.\emptyset + w.\bar{2}$ we get from this that
\begin{equation} \label{eq2b1x1b}
\bar{2} 1 x \bar{1} = a^{-1}.e_{18} + \frac{u}{wa}.e_{19} + \frac{a-u}{w}.e_{15} + w^{-1}.e_4 + \frac{v}{w}.e_6 - \frac{a^2+v}{aw}.e_2 - \frac{1}{aw}.e_{14}
\end{equation}

\subsubsection{Description of $u_2u_1xu_1u_2$}
\label{sect:u2u1xu1u2}
Let us denote $e' = \bar{2}1x1\bar{2}$. We postpone for now the determination of its value. Note that $f(e')=e'$.
Multiplying (\ref{eq21x1b}) on the right by $2$ and using expansion of $\bar{2}\bar{2}$ by the cubic relation as well as $x2 = a.x$, we get that
\begin{equation} \label{eq21x1b2}
 21x\bar{1}2 = a.f(2\bar{1}x1) + f(e_{21}) + e_4 - a^2. e_6 -a^{-1}.e_{16}
\end{equation}
Similarly, multiplying (\ref{eq21x1b}) on the right by $\bar{2}$, one gets
\begin{equation}
 21x\bar{1}\bar{2} = a.f(\bar{2}\bar{1}x1) -a^{-1}.e_{17} + e' + a^{-2}.e_4 - e_6
\end{equation}
We start over the same computations, this time from (\ref{eq2b1x1b}), multiplying first by $2$ and then by $\bar{2}$ on the right. One gets
\begin{equation} \label{eq21x1b2}
\bar{2}1x\bar{1} 2= \frac{bc}{w}.f(2 \bar{1} x1) + \frac{u}{wa}.e_{16} + \frac{a-u}{w}.f(e_{21}) -\frac{1}{wa}.e_{20}+\frac{va}{w} . e_6 - \frac{v+a^2}{w}. e_2 + \frac{1}{bc}.e_4
\end{equation}
and
\begin{equation}
\bar{2} 1 x\bar{1} \bar{2} = \frac{bc}{w}.f(\bar{2}\bar{1}x1) + \frac{u}{wa}.e_{17} + \frac{a-u}{w}.e' + w^{-1}a^{-1}.e_4 + \frac{v}{wa}.e_6 - \frac{v+a^2}{a^2w}.e_2 - \frac{1}{aw}.e_{21}
\end{equation}

Note that the above four equations need an expression of $f(e_{21})$ and $e'$ to be expressable as a linear combination of the $(e_i)_{1 \leq i \leq 25}$.
We first compute $f(e_{21})=21x1\bar{2}$. Its expression as been given in (\ref{eqfe21}), but is not yet justified. We do it now.
Following rule (18) of table \ref{sect:GQ3asGQ2bimodule}, we can expand $1 \bar{2}1$ inside
$21x1\bar{2} = 231\bar{2}13\bar{2}$. This yields
\begin{equation} \label{eq21x12bA}
21x1\bar{2} \equiv 232\bar{1}23\bar{2} - (bc).\bar{3}213\bar{2}3\bar{2} + (a^2+v).x1\bar{2} + \frac{a^2+v}{a^2}.21x - (a^2+v).e_1
\end{equation}

Note that $232\bar{1}23\bar{2} = 2.w_+.\bar{2} = w_+$.
We now use rule (23) of table \ref{table:GQ3grobsigned1} (after applying the shift morphism $1 \mapsto 2$, $2 \mapsto 3$)
to expand $\bar{2}3\bar{2} 3$ inside $\bar{2}3\bar{2} 312\bar{3} = f(\bar{3}213\bar{2}3\bar{2})$.  Using in addition a few easy braid relations we get from this expansion that
$$
\bar{2}3\bar{2} 312\bar{3} \equiv a^{-1}.3\bar{2}312\bar{3} - \frac{a}{bc}.21x\bar{1} + \frac{a^2}{bc}.1x\bar{1} + \frac{v}{bc}.\bar{2}1x + (bc)^{-1}.21x - \frac{a^2+v}{w}.1x
$$
modulo $\GQ_3 u_3 \GQ_3$. Pluging this into (\ref{eq21x12bA}) we get
that $21x1\bar{2}$ is equal to 
$$
w_+ - \frac{bc}{a}.\bar{3}213\bar{2}3 + a.\bar{1}x12 - a^2.\bar{1}x1 + a^2.x1\bar{2} - x12 + \frac{a^2+v}{a}.x1 +\frac{a^2+v}{a^2}.21x - (a^2+v).x$$
modulo $\GQ_3 u_3 \GQ_3$. Now, $\bar{3}213\bar{2}3 = \bar{3}231\bar{2}3= 23\bar{2}1\bar{2}3$ and expanding $\bar{2}1\bar{2}$ by rule (18) of table \ref{sect:GQ3asGQ2bimodule},
we get that 
$$
(bc).23\bar{2}1\bar{2}3 = 2.w_+ - 21x1 + (a^2+v).x1 +\frac{a^2+v}{a}.21x - a(a^2+v).x
$$
Altogether, this yields
\begin{equation} \label{eq21x12bB}
21x1\bar{2} = (\emptyset - a^{-1}.2).w_+ + a^{-1}.21x1 + a.\bar{1}x12 - a^2.\bar{1}x1 + a^2.x1\bar{2} - x12
\end{equation}
Applying $f$ we get 
\begin{equation} \label{eq2b1x12}
\bar{2}1x12 = (\emptyset - a^{-1}.2).w_+ + a^{-1}.1x12 + a.21x\bar{1} - a^2.1x\bar{1} + a^2.\bar{2}1x - 21x
\end{equation}
since $w_+2 = 2w_+$. Therefore, we have
\begin{equation} \label{eq2b1x12Delta}
21x1\bar{2} - \bar{2}1x12 = a^{-1}.(21x1 - 1x12) + a.(\bar{1}x12 - 21x\bar{1}) + a^2.(1x\bar{1} - \bar{1}x1)+a^2.(x1\bar{2} - \bar{2}1x) + (21x-x12)
\end{equation}
and from this we get the expression of $f(e_{21})$ obtained above (\ref{eqfe21}).

We now compute $e' = \bar{2}1x1\bar{2}$.
Multiplying (\ref{eq2b1x12}) on the right by $2$ yields (after expanding $22$ inside $1x12$)
$$
\bar{2}1x122 = w_+.(\emptyset - a^{-1}.2)2 + a^{-1}u.1x12 - a^{-1}v.1x1 + a^{-1}w.1x1\bar{2} + a.21x\bar{1}2 - a^2.1x\bar{1}2 + a^3.\bar{2}1x - a.21x
$$
Using that $\bar{2} = w^{-1}.22 - uw^{-1}.2 + vw^{-1}.\emptyset$ we deduce from this that $\bar{2}1x1\bar{2}$ is equal to
$$
w^{-1}.w_+.(\emptyset - a^{-1}.2)2+\frac{u}{aw}.1x12 - \frac{v}{aw}.1x1 + \frac{1}{a}.1x1\bar{2} + \frac{a}{w}.21x\bar{1}2 - \frac{a^2}{w}.1x\bar{1}2
+ \frac{a^3}{w}.\bar{2}1x - \frac{a}{w}.21x - \frac{u}{w}.\bar{2}1x12 + \frac{v}{w}.\bar{2}1x1
$$
Now, we use that $1x12 = 2 \bar{2}1x12$. From (\ref{eq2b1x12}) 
this yields
$1x12  =2. w_+.(\emptyset- a^{-1}.2) + a^{-1}.21x12 + a.221x\bar{1} - a^2.21x\bar{1} + a^2.1x - 221x$.
Expanding $22$ twice, we get that
$w_+.(\emptyset- a^{-1}.2)2 = 2. w_+.(\emptyset- a^{-1}.2)$ is equal to
$$
1x12 - a^{-1}.21x12 - au.21x\bar{1} - av.1x\bar{1} - aw.\bar{2}1x\bar{1} + a^2.21x\bar{1} - (a^2+v).1x + u.21x + w.\bar{2}1x
$$
Pluging this into the former equation provides an expression for $\bar{2}1x1\bar{2}$, as
$$
\bar{2}1x1\bar{2} = \frac{a+u}{aw}.e_{16} - \frac{1}{aw}.e_{20} + \frac{a(a-u)}{w}.21x\bar{1} + \frac{av}{w}.e_{18} - a.\bar{2}1x\bar{1} - \frac{a^2+v}{w}.e_2 + \frac{u-a}{w}.e_4 - \frac{v}{aw}.e_{19}
$$
{}
$$
+ \frac{a}{w} . 21x\bar{1}2+a^{-1}.e_{17} - \frac{a^2}{w}.1x\bar{1} 2 + \left( \frac{a^3}{w}+1\right).e_6 - \frac{u}{w}.\bar{2}1x12 + \frac{v}{w}.e_{15}
 $$
 
 We shall need to compute $x \bar{1}2\bar{1}$. Using relation (18) of table \ref{table:GQ3grobsigned2} to expand $\bar{1}2 \bar{1}$, we get after an easy
 calculation, that
\begin{equation}  \label{eqx1b21b}
x\bar{1}2\bar{1} = \frac{v}{w}.e_{12} + (bc)^{-1}.e_{24} - \frac{au}{w}.e_{11}
\end{equation} 

 \subsection{$\GQ_4/\GQ_3u_3\GQ_3$ as a $\GQ_3$-bimodule - computational description} 
\label{sect:A4tildebimod}
\subsubsection{Left multiplication by $s_1$ inside $\tilde{A}_4$}
$$
\begin{array}{|lcl|lcl|lcl|}
\hline
e_1 & \mapsto & e_2 &  e_6 & \mapsto & e_{25}  & e_{11} & \mapsto & e_{17} \\
e_2 & \mapsto & -v.e_1 + u.e_2 + w.e_3  &  e_7 & \mapsto & a.e_{7}  & e_{12} & \mapsto & f(2 \bar{1} x1) \\
e_3 & \mapsto & e_1 &  e_8 & \mapsto & e_{19}  & e_{13} & \mapsto & f(\bar{2}\bar{1}x1) \\
e_4 & \mapsto & a.e_4 &  e_9 & \mapsto & e_{18}  & e_{14} & \mapsto & a.e_{14} \\
e_5 & \mapsto & a.e_6 &  e_{10} & \mapsto & e_{16}  & e_{15} & \mapsto & f(1.e_{24}) \\
\hline
\end{array}
$$
and
$$
\begin{array}{|lcl|lcl|}
\hline
e_{16} & \mapsto & u.e_{16}-v.e_{10} +w.f(21x\bar{1}) & e_{21} & \mapsto & a.e_{21} \\
e_{17} &  \mapsto & u.e_{17} - v.e_{11}+w. f(\bar{2}1x\bar{1}) & e_{22} &  \mapsto & (\ref{eq1wplus}) \\
e_{23} &  \mapsto & (\ref{eq1wmoins}) \& (\ref{eq1bx1b}) \\
e_{18} &  \mapsto & u.e_{18} - v.e_9 + w. \bar{1}x\bar{1} & e_{23} &  \mapsto & (\ref{eq1wmoins}) \\ 
e_{19} & \mapsto & u.e_{19} - v.e_8 + w.f(e_{18}) & e_{24} &  \mapsto & (\ref{eq1xy}) \\
e_{20} & \mapsto & a.e_{20} & e_{25} &  \mapsto &  u.e_{25} - v.e_6+wa^{-1}.e_5\\
\hline
\end{array}
$$
We now consider $1 w_+ = 3(12\bar{1}2)3$. Using rule (15) in table \ref{table:GQ3grobsigned2} we
get that $1 w_+ = w. 3 \bar{2} 1 \bar{2} 3 + u.312\bar{1}3 - v. e_8$. We have $312\bar{1}3 = 1323\bar{1} = 1232\bar{1} \equiv 0$,
hence $1 w_+ \equiv w. 3 \bar{2} 1 \bar{2} 3  - v. e_8$.
Now, using rule (18) in table  \ref{table:GQ3grobsigned2} we get
\begin{equation}
3 \bar{2}1\bar{2} 3 \equiv (bc)^{-1}.e_{22} - (bc)^{-1}.e_{19} + \frac{a^2+v}{w}. e_8 + \frac{a^2+v}{w} . e_2 - \frac{a^2+v}{bc}. e_1
\end{equation}
hence 
\begin{equation} \label{eq1wplus}
1w_+ \equiv a.e_{22} - a. e_{19} + a^2.e_8 + (a^2+v)e_2 - a(a^2+v).e_1.
\end{equation}
In particular we get the following potentially useful property
\begin{equation} \label{eq1wpmodu1xu1}
1.w_+ \equiv a.w_+ \mod \GQ_4^{(1)}+u_1xu_1.
\end{equation}

We now want to compute $1 w_- = w_- 1 = 1.e_{23}$. For this we first compute
$\bar{1}w_- = \bar{1} \bar{3}\bar{2} 1 \bar{2} \bar{3}
= \bar{3}( \bar{1}\bar{2} 1) \bar{2} \bar{3}
= \bar{3} 2\bar{1} (\bar{2} \bar{2}) \bar{3}
= w^{-1} \bar{3} 2\bar{1} 2 \bar{3} - uw^{-1}.\bar{3} 2\bar{1}  \bar{3} + vw^{-1}.\bar{3} 2\bar{1} \bar{2}  \bar{3}$.
Since $\bar{3} 2\bar{1} \bar{2}  \bar{3} = \bar{3} \bar{1} \bar{2} 1 \bar{3}=  \bar{1} \bar{3}\bar{2}  \bar{3}1   \bar{1} \bar{2}\bar{3}  \bar{3}1  \equiv 0$
we get
$\bar{1}.w_- = w^{-1}.\bar{3} 2\bar{1} 2 \bar{3} - uw^{-1}(bc)^{-1}.x\bar{1}$.

Then using relation (18) of table \ref{table:GQ3grobsigned2}  to replace $2 \bar{1} 2$, we get after a straightforward computation that
\begin{equation}
\bar{3} 2 \bar{1} 2 \bar{3} \equiv (bc) e_{23} - \bar{1} x \bar{1} + \frac{a^2+v}{w} e_9 + \frac{a^2+v}{w} e_3 - \frac{a^2 + v}{wa} e_1
\end{equation}
Now, $\bar{1}\bar{1}.w_- =  w^{-1}.\bar{1}.\bar{3} 2\bar{1} 2 \bar{3} - uw^{-1}(bc)^{-1}.\bar{1}x\bar{1}$. From the expression of
$\bar{3} 2 \bar{1} 2 \bar{3} $ given above, this yields after a again straightforward computation that
$$
\bar{1}\bar{1}.w_- = w^{-2}(bc)^2.e_{23} + \frac{a^2+v}{w^2a}.e_9 + \frac{v(a^2+v)}{w^3}.e_3 - \frac{a^2+v}{w^3a}(bc+ua).e_1 - w^{-2}.e_{18}
-\frac{v}{w^2}.\bar{1}x\bar{1} + \frac{a^2+v}{w^3}.e_2
$$

Now, $1.w_- = u.w_- - v .\bar{1}.w_- + w. \bar{1}\bar{1} . w_-$. From this we easily get
\begin{equation} \label{eq1wmoins}
1.w_- = a.e_{23} + \frac{a}{w}.e_9 - \frac{a(a^2+v)}{w^2}.e_1 - w^{-1}.e_{18}+\frac{a^2+v}{w^2}.e_2
\end{equation}

We now want to compute $1.e_{24} = 1xy$. We use the identity 
 $1x\bar{1} = \bar{1}x1 +  a.x\bar{1} -a^{-1}.x1 -a.\bar{1}x+a^{-1}.1x$ (mod $\GQ_3 u_3 \GQ_3$).
 Multiplying on the right by $2\bar{1}$ we get (through a couple of braid relations and $x2 = ax$) that
 $$
 1x\bar{1}2\bar{1} = a^{-1}.f(21x\bar{1}) + a. x \bar{1}2\bar{1} - a^{-2}.e_{10}-a^2.\bar{1}x\bar{1}+e_{18}
 $$
 which provides an explicit description of $ 1x\bar{1}2\bar{1}$ thanks to (\ref{eqx1b21b}).

Using relation (18) of table \ref{table:GQ3grobsigned2}  to replace $y = 1 \bar{2} 1$, we get after a straightforward computation
making use of $x2 = ax$ and $x\bar{2} = a^{-1} x$ that 
\begin{equation} \label{eq1xy}
1.xy = u.e_{17} - \frac{v}{a} f(2 \bar{1} x1) + (bc).1x\bar{1}2\bar{1}
\end{equation}

Now, $1.e_{15} = 1 \bar{2} 1 3 \bar{2} 3 1 = yx1 = f(1xy)$
and this completes the table of left multiplication by $s_1$.

\subsubsection{Left multiplication by $s_2$ inside $\tilde{A}_4$}
$$
\begin{array}{|lcl|lcl|lcl|}
\hline
e_1 & \mapsto & a.e_1  &  e_6 & \mapsto & e_2  & e_{11} & \mapsto & a.e_{11} \\
e_2 & \mapsto & e_4  &  e_7 & \mapsto & e_3  & e_{12} & \mapsto & a.e_{12} \\
e_3 & \mapsto & e_5 &  e_8 & \mapsto & a.e_8 & e_{13} & \mapsto & a.e_{13} \\
e_4 & \mapsto & u.e_4 - v.e_2+w.e_6 &  e_9 & \mapsto & a.e_9  & e_{14} & \mapsto & u.e_{14}-v.e_{19}+w.e_{15} \\
e_5 & \mapsto & u.e_5-v.e_3+w.e_7 &  e_{10} & \mapsto &  a.e_{10} & e_{15} & \mapsto & e_{19} \\
\hline
\end{array}
$$
and 
$$
\begin{array}{|lcl|lcl|}
\hline
e_{16} & \mapsto & e_{20} & e_{21} & \mapsto & u.e_{21}-v.e_{17}+w.e' \\
e_{17} &  \mapsto & e_{21} & e_{22} &  \mapsto & (\ref{de22}) \\
e_{18} &  \mapsto & 21x\bar{1} \ \ \ (\ref{eq21x1b}) & e_{23} &  \mapsto & (\ref{eqde23}) \\
e_{19} & \mapsto & e_{14} & e_{24} &  \mapsto & a.e_{24} \\
e_{20} & \mapsto & u.e_{20}-v.e_{16}+w.f(e_{21}) & e_{25} &  \mapsto &  (\ref{eqde25}) \\
\hline
\end{array}
$$
where $e' = \bar{2}1x1\bar{2}$ has been computed in section \ref{sect:u2u1xu1u2}.

Every entry in the table is straightforward to compute, except for 3 of them. We need to compute $2.e_{22} = 2.w_+$, $2.e_{23} = 2.w_-$ and $2.e_{25}$. We
start with $2.e_{25} = 2yx = 21\bar{2}1x = \bar{1}211x = u.\bar{1}21x - v.\bar{1}2.x + w.\bar{1}2\bar{1}x = u.21\bar{2}x - av.\bar{1}.x + w.f(x\bar{1}2\bar{1})$ 
hence
\begin{equation} \label{eqde25}
2.e_{25} = \frac{u}{a}.e_4 - av.e_3 + w.f(x\bar{1}2\bar{1})
\end{equation}
Now, from (\ref{eq21x12bB}) one easily gets that
\begin{equation} \label{de22}
2.e_{22} = 2.w_+ = -a.e_{21} + a.e_{22} + e_{14} + a^2.\bar{1}x12 - a^3.f(e_{18}) + a^3.e_{11} - a.e_{10}
\end{equation}
We now compute $2.w_- = 2 \bar{3}\bar{2}1\bar{2}\bar{3}$.
Using rule (16) of table \ref{table:GQ3grobsigned2}, we get $2 \bar{3}\bar{2}
=\bar{2}\bar{3}2-a^{-2}.\bar{2}32 +a^{-2}.23\bar{2}+a.\bar{3}\bar{2}-a^{-1}.\bar{3}2
-a^{-1}.3\bar{2}+a^{-3}.32-a.\bar{2}\bar{3}+a^{-1}.\bar{2}3+a^{-1}.2\bar{3} - a^{-3}.23$.
Multiplying on the right by $1\bar{2}\bar{3}$, we get after a straightforward computation that
$$
2 .w_- \equiv (bc)^{-1}.\bar{2}\bar{1}x1 + a^{-2}.23\bar{2}1\bar{2}\bar{3} + a.e_{23} -w^{-1}.f(e_{18})-a^{-1}.3\bar{2}1\bar{2}\bar{3}
$$
Since $\bar{2}1 \bar{2} \equiv (bc)^{-1}.2\bar{1}2 \mod u_1u_2u_1$ (see section \ref{sect:GQ3asGQ2bimodule}),
we have $23\bar{2}1\bar{2}\bar{3} \equiv (bc)^{-1}.232\bar{1}2\bar{3}$ and $3\bar{2}1\bar{2}\bar{3} \equiv (bc)^{-1}.32\bar{1}2\bar{3}$ modulo $\GQ_3 u_3 \GQ_3$.
Now, in the proof of lemma \ref{lemGQ4L2} we checked that $3 2 \bar{1} 2 \bar{3} = \bar{2}\bar{1}x12$, whence
$23 2 \bar{1} 2 \bar{3} = \bar{1}x12$. Altogether, this yields
\begin{equation} \label{eqde23}
2.e_{23} = a.e_{23}+(bc)^{-1}.\bar{2}\bar{1}x1 + \frac{1}{aw}.f(21x\bar{1}) -w^{-1}.f(e_{18})-w^{-1}.f(21x\bar{1}\bar{2})
\end{equation}
\subsection{Freeness of $\GQ_4$ as a $R$-module} 

Let $\mathcal{B}_{(0)}$ be a basis of $\GQ_3$. We recall from lemmas \ref{lem:A3s3pmF} and \ref{lem:spanA3s3pmA3}  that $\GQ_3u_3\GQ_3 = \GQ_3+\GQ_3 s_3 \GQ_3 + \GQ_3 s_3^{-1}\GQ_3$ is
spanned by $$
\mathcal{B}_{(1)} = \mathcal{B}_{(0)} \sqcup 
\mathcal{B}_{(0)} \times F_2 \sqcup E' \times \{ \bar{3}\bar{2}, \bar{3}\bar{2} 1 \} \sqcup E_0 \times \{ \bar{3}2\bar{1} \} 
\sqcup E \times \{ 3 \bar{2}\bar{1},32\bar{1}2 \}
\sqcup \{ \bar{3}\bar{2}\bar{1}, \bar{3} 2 \bar{1} 2 \}
$$
where 
\begin{itemize}
\item $E = \{ \emptyset, 2, \bar{2}, 12, 1\bar{2}, \bar{1}2, \bar{1}\bar{2}, 2\bar{1}2 \}$ 
\item $F = \{ \emptyset, 2, 21, 2\bar{1}, \bar{2}, \bar{2}1 \}$.
\item $F_2 = \{ 3, \bar{3}, 32,3\bar{2}, \bar{3} 2, 321, \bar{3}21, 32\bar{1},3\bar{2}1 \}$
\item $E_0 =\{ \emptyset, 1 ,\bar{1},2,\bar{2} \} $
\item $E' = \{ \emptyset, 1, \bar{1},21,2 \bar{1},\bar{2} 1, \bar{2}\bar{1}, 1 \bar{2} 1 \}$
\end{itemize}
Since $\mathcal{B}_{(0)}$ has 20 elements, $\mathcal{B}_{(1)}$
has $239$ elements.
We now add to $\mathcal{B}_{(1)}$ the 25 elements described by the words $(e_i)_{1 \leq i \leq 25}$ of section \ref{sect:A4tildebasis}
to get a collection $\mathcal{B}_{(2)}$ of 264 elements.

\begin{theorem} \label{theo:A4libre} $\GQ_4$ is a free $R$-module of rank 264, and $\mathcal{B}_{(2)}$ is a basis.
\end{theorem}
\begin{proof} We know that $\GQ_4 = \GQ_3u_3\GQ_3 + \GQ_3 x \GQ_3+R w_+ + R w_-$, where $x = 3\bar{2}3$,
by proposition \ref{propGQ4P2} and lemma \ref{lemGQ4L3},. We want to prove that $\mathcal{B}_{(2)}$ spans $\GQ_4$ as a $R$-module.
We recalled that $\mathcal{B}_{(1)}$ spans $\GQ_3 u_3 \GQ_3$, so it is sufficient
to prove that the image of $\mathcal{B}_{(2)} \setminus \mathcal{B}_{(1)}$ span $\GQ_4/\GQ_3u_3\GQ_3$.
We know that $\GQ_4/\GQ_3 u_3 \GQ_3$ is generated as a $\GQ_3$-bimodule by $x,w_+,w_-$. Let $M$ denote
the $R$-submodule of $\GQ_4/\GQ_3 u_3 \GQ_3$ spanned by the image of $\mathcal{B}_{(2)} \setminus \mathcal{B}_{(1)}$.
It contains $x,w_+,w_-$, and it is stable by left multiplication by $\GQ_3$, by section \ref{sect:A4tildebimod}. Morever, it is also stable by the antimorphism
of $\GQ_3$-module $f$, hence it is stable by left and right multiplication by $\GQ_3$. This proves $M = \GQ_4/\GQ_3u_3\GQ_3$,
hence $\GQ_4$ is spanned by $\mathcal{B}_{(2)}$. Since $\GQ_4 \otimes K$ has dimension 264, this
proves that $\GQ_4$ is a free $R$-module with basis $\mathcal{B}_{(2)}$.
\end{proof}

An immediate corollary is the following one.

\begin{corollary}
The natural morphism $Q_3 \to Q_4$ is injective.
\end{corollary}

The matrix of left and right multiplication by $s_1,s_2$ on the basis $\mathcal{B}_{(2)} \setminus \mathcal{B}_{(1)}$
of $\tilde{A}_4$ can be found in the file \verb+A4tilde.gap+
at \url{http://www.lamfa.u-picardie.fr/marin/data/A4tilde.gap}.

In \cite{LG} we described an explicit isomorphism
$$
\Phi_{4}^K : \GQ_4 \otimes \tilde{K} \to \tilde{K}^3 \times M_2(\tilde{K})^2 \times M_3(\tilde{K})^5 \times M_6(\tilde{K})^4 \times M_8(\tilde{K})
$$
from the explicit matrix models of the irreducible representations of the semisimple algebra $H_4 \otimes \tilde{K}$,
where $\tilde{K} = \Q(\zeta_3,a,b,c)$. We denote $\Phi_4$ the composite of $\Phi_4^K$ by the natural $R$-algebra
morphism $\GQ_4 \to \GQ_4 \otimes \tilde{K}$. By theorem \ref{theo:A4libre} we know that $\Phi_4$ is injective,
and therefore $\Phi_4$ can be used for explicit computations inside $\GQ_4$.

Computing the images of the elements of the relevant bases, together with their images
by left and right multiplications by the generators and their inverses, we could in principle get in this way
the structure constants of $\GQ_4$ on the basis $\mathcal{B}_{(2)}$.
However, because the coefficients of the equations belong to the field $\Q(a,b,c)$, this linear algebra matter
is computationally nontrivial (even after having reduced the problem to $a=1$).

\bigskip

\bigskip

\section{Structure on 5 strands}

\subsection{General properties}

We denote $\GQ_{n+1}^{(1)} = \GQ_n.1$ the image of $\GQ_n$ inside $\GQ_{n+1}$ under the natural map.
A collection of $\GQ_n$-subbimodules of $\GQ_{n+1}$ is defined inductively by the formula
$\GQ_{n+1}^{(k+1)} = \GQ_{n+1}^{(k)}u_n \GQ_n$. In other terms,
$$
\GQ_{n+1}^{(k)} = \underbrace{\GQ_n u_n \GQ_n u_n \dots u_n \GQ_n u_n \GQ_n}_{k+1 \mbox{ terms}}
$$
We know that $\GQ_3^{(2)} = \GQ_3$ and $\GQ_4^{(2)} = \GQ_4$.

Every element $x$ of the braid group $B_{n+1}$ either belongs to $B_n$ (that is, the image of $B_n$ under the usual map $B_n \to B_{n+1}$
of adding one strand `on the right'), or can be written either as $x_1 s_n x_2s_n \dots s_n x_{k+1}$ for some $k$,
or as $x_1 s_n^{-1} x_2s_n^{-1} \dots s_n^{-1} x_{k+1}$ for some $k$. The process to convert any given braid to one of these
forms is called by Dehornoy `handle reduction', and is at the origin of his ordering on the braid group $B_{n+1}$. A nice reference
for this is \cite{DEHORDER}, ch. 3. The basic `handle reduction' has the following form. If $a \in B_{n+1}$ is written
$$
a = s_n s_{n-1} a_1 s_{n-1} a_2 \dots s_{n-1} a_k s_{n-1} s_n^{-1}
$$
with $a_i \in B_{n-1}$, then it can be rewritten as
$$
a = s_{n-1}^{-1} \left( s_n  (^{s_{n-1}} a_1) s_n  (^{s_{n-1}} a_2) \dots s_n (^{s_{n-1}} a_k) s_n \right) s_{n-1}
$$

\begin{figure}
\begin{minipage}{\textwidth}
\begin{minipage}[c][9cm][c]{\dimexpr0.5\textwidth-0.5\Colsep\relax}
\begin{center}
\begin{tikzpicture}[scale=0.7]
\braid[braid colour=black,strands=5,braid start={(0,0)}]%
{ \dsigma _4 \dsigma_3 \dsigma_2 \dsigma_3 \dsigma_2 \dsigma_3 \dsigma_3 \dsigma_3 \dsigma_3 \dsigma_2 \dsigma_3\dsigma_4^{-1}   }
\fill[blue] (0.5,-2)-- (3.5,-2) -- (3.5,-3) -- (0.5,-3)
 -- cycle;
\fill[blue] (0.5,-4)-- (3.5,-4) -- (3.5,-5) -- (0.5,-5) -- cycle;
\fill[white] (0.5,-6)-- (4.5,-6) -- (4.5,-8) -- (0.5,-8) -- cycle;
\draw (2,-7) node {$\dots$};
\fill[blue] (0.5,-9)-- (3.5,-9) -- (3.5,-10) -- (0.5,-10) -- cycle;

\end{tikzpicture}
\end{center}

\end{minipage}\hfill
\begin{minipage}[c][9cm][c]{\dimexpr0.5\textwidth-0.5\Colsep\relax}
\begin{center}

\begin{tikzpicture}[scale=0.5]
\braid[braid colour=black,strands=5,braid start={(0,0)}]%
{ \dsigma_3^{-1} \dsigma_4 \dsigma_3 \dsigma_2 \dsigma_3^{-1} \dsigma_4 \dsigma_3 \dsigma_2 \dsigma_3^{-1}  \dsigma_1 \dsigma_1 \dsigma_4 \dsigma_3 \dsigma_2 \dsigma_3^{-1} \dsigma_4 \dsigma_3 }
\fill[blue] (0.5,-3)-- (3.5,-3) -- (3.5,-4) -- (0.5,-4) -- cycle;
\fill[blue] (0.5,-7)-- (3.5,-7) -- (3.5,-8) -- (0.5,-8) -- cycle;
\fill[white] (0.5,-9)-- (4.5,-9) -- (4.5,-11) -- (0.5,-11) -- cycle;
\draw (2,-10) node {$\dots$};
\fill[blue] (0.5,-13)-- (3.5,-13) -- (3.5,-14) -- (0.5,-14) -- cycle;

\end{tikzpicture}

\end{center}

\end{minipage}
\end{minipage}
\caption{Handle reduction} 
\label{fig:handlered}
\end{figure}
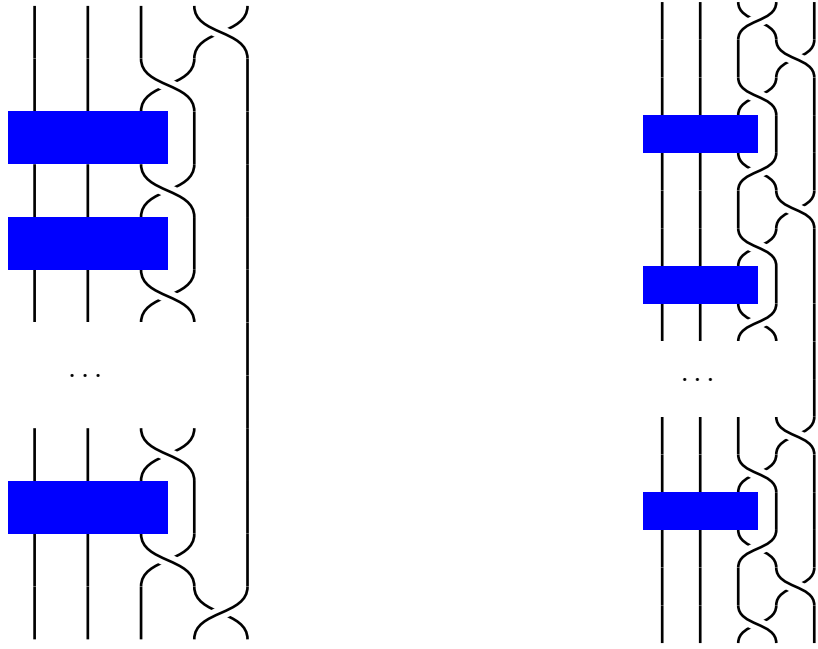

An iterated application of handle reduction proves the following identities

\begin{equation} \label{eq:whandleA}
s_n^{-1}(s_{n-1}^{-1}\dots s_2^{-1}s_1s_2^{-1}\dots s_{n-1}^{-1})s_n 
=(s_{n-1}\dots s_1)(s_n^{-1}\dots s_3^{-1}s_2 s_3^{-1}\dots s_n^{-1}) (s_{n-1}\dots s_1)^{-1}
\end{equation}
and its mirror image
\begin{equation} \label{eq:whandleB}
s_n(s_{n-1}\dots s_2s_1^{-1}s_2\dots s_{n-1})s_n^{-1}
=(s_{n-1}^{-1}\dots s_1^{-1})(s_n\dots s_3s_2^{-1} s_3\dots s_n) (s_{n-1}^{-1}\dots s_1^{-1})^{-1}
\end{equation}

Let us define $\GQ_{n+1}^{(1)+} =\GQ_{n+1}^{(1)-} = \GQ_n.1$ and 
$\GQ_{n+1}^{(k+1)+} = \GQ_{n+1}^{(k)+}s_n \GQ_n$,
$\GQ_{n+1}^{(k+1)-} = \GQ_{n+1}^{(k)-}s_n^{-1} \GQ_n$. In other terms,
$$
\GQ_{n+1}^{(k)+} = \underbrace{\GQ_n s_n \GQ_n s_n \dots s_n \GQ_n s_n \GQ_n}_{k+1 \mbox{ terms}}
\ \ \ \ 
\GQ_{n+1}^{(k)-} = \underbrace{\GQ_n s_n^{-1} \GQ_n s_n^{-1} \dots s_n^{-1} \GQ_n s_n^{-1} \GQ_n}_{k+1 \mbox{ terms}}
$$
Let
$$
\GQ_{n+1}^+ = \sum_{k \geq 1} \GQ_{n+1}^{(k)+}
\ \ \ \ 
\GQ_{n+1}^- = \sum_{k \geq 1} \GQ_{n+1}^{(k)-}
$$

The immediate consequence of 
Handle reduction in our context can be stated in the
following form, although it is the explicit process recalled above (and illustrated in figure \ref{fig:handlered}) that will be useful to us.

\begin{proposition} For all $n$ we have $\GQ_n = \GQ_n^+ + \GQ_n-$.
\end{proposition}

\subsection{$\GQ_5^{(2)}$ as a $\GQ_4$-bimodule}

We know by proposition \ref{propGQ4P2} that
$\GQ_4 = \GQ_4^{(2)} = \GQ_3u_3\GQ_3 + \GQ_3.3\bar{2}3.\GQ_3 + \GQ_3.32\bar{1}23
+ \GQ_3.\bar{3}\bar{2}1\bar{2}\bar{3}$. Since $\bar{4}3\bar{4} \in R^{\times} .4\bar{3}4 + u_3u_4u_3$
and, by handle reduction, $4X\bar{4}$ for $X \in \{3\bar{2}3,32\bar{1}23,\bar{3}\bar{2}1\bar{2}\bar{3}\}$
can be written as $X4Y4Z$ or $X\bar{4}Y\bar{4}Z$ for $X,Z \in B_4$ and $Y \in \{3, \bar{3},3\bar{2}3,\bar{3}2\bar{3} \}$,
this implies
that $\GQ_5^{(2)}$
is generated as a $\GQ_4$-bimodule by $\emptyset, 4, \bar{4}, 4\bar{3}4, 43\bar{2}34,
\bar{4}\bar{3}2\bar{3}\bar{4}$, $432\bar{1}234$, $\bar{4}32\bar{1}23\bar{4}$, $4\bar{3}\bar{2}1\bar{2}\bar{3}4$,
$\bar{4}\bar{3}\bar{2}1\bar{2}\bar{3}\bar{4}$.

Let us introduce $\GQ_5^{(1.5)}$ the $\GQ_4$-subbimodule of $\GQ_5$ generated by $\emptyset,4,\bar{4},4\bar{3}4, 43\bar{2}34,
\bar{4}\bar{3}2\bar{3}\bar{4}$, and $\GQ_5^{(1.3)}$ the $\GQ_4$-subbimodule of $\GQ_5$ generated by $\emptyset,4,\bar{4},4\bar{3}4$.
Notice that both are stable under $F=\Phi \circ \Psi$.

We use the defining relation under the form
\begin{equation} \label{eq:343rel}
(3-a.\emptyset).4\bar{3}=a^2(3-a.\emptyset).\bar{4}\bar{3}+a(\emptyset-a.\bar{3}).\bar{4}3+(\bar{3}-a^{-1}.\emptyset)43 + (a^{-1}.3 - a.\bar{3}).4
+a(a^2.\bar{3}-3).\bar{4}
\end{equation}
and in particular $(3-a.\emptyset).4\bar{3} \in u_3 \bar{4}u_3+u_3. 4 + u_3.43$. Since $43\bar{2}1\bar{2}\bar{3}4
= \bar{2}\bar{1}.4\bar{3} 2\bar{3}4.12$
this implies 
$$
(3 - a.\emptyset).4w_-4 \in u_3.\bar{4}.u_3 \bar{2}1\bar{2}\bar{3}4 + \GQ_4.4\bar{3}4 + u_3u_2u_1.4\bar{3}2\bar{3}4.u_1u_2 \subset u_3.\bar{4}.\GQ_4.4+\GQ_4.4\bar{3}4+ \GQ_4.4\bar{3}2\bar{3}4.\GQ_4
$$

\begin{proposition} {\ }
\begin{enumerate} 
\item $4.\GQ_4.\bar{4} \subset \GQ_5^{(1.5)}$ and $\bar{4}.\GQ_4.4 \subset \GQ_5^{(1.5)}$
\item $\GQ_5^{(2)} = \GQ_5^{(1.5)} + R.4w_+4 + R.\bar{4}w_-\bar{4} + R.4w_-4 + R.\bar{4}w_+\bar{4}$. Moreover, 
for all $X \in \{ 4w_+4 ,\bar{4}w_-\bar{4} ,4w_-4 ,\bar{4}w_+\bar{4} \}$ and $\la \in \GQ_4$, we have $\la.X \equiv X.\la \equiv \eps(\la)X \mod \GQ_5^{(1.5)}$,
where $\eps : \GQ_4 \to R$
is induced by $s_i \mapsto a$.
\end{enumerate}
\end{proposition}
\begin{proof} 
We first prove (1). Since $\GQ_5^{(1.5)}$ is stable under $F$ we only need to
prove $4.\GQ_4.\bar{4} \subset \GQ_5^{(1.5)}$.
Since $\GQ_4$ us generated as a $\GQ_3$-bimodule by $S= \{ \emptyset,3,\bar{3},3\bar{2}3,32\bar{1}23,\bar{3}\bar{2}1\bar{2}\bar{3} \}$, and $4$ commutes with $\GQ_3$, we need to prove $4.X.\bar{4} \in \GQ_5^{(1.5)}$ for all $X \in S$.
Clearly $4 X \bar{4} \in \GQ_5^{(1)} \subset \GQ_5^{(1.5)}$ for $X \in \{ \emptyset,3,\bar{3} \}$. By the proof of lemma \ref{lemGQ4L2},
or handle reduction,we know that $32\bar{1}2\bar{3} \in \GQ_3.3\bar{2}3.\GQ_3$,
whence $4.3\bar{2}3.\bar{4}  = s(32\bar{1}2\bar{3}) \in s(\GQ_3.3\bar{2}3.\GQ_3) \subset \GQ_4.4\bar{3}3.\GQ_4 \subset \GQ_5^{(1.3)}$.
It remains to consider $X \in \{w_+,w_- \}$. By handle reduction (see (\ref{eq:whandleB})) we have $4 w_+ \bar{4} = \bar{3}\bar{2}\bar{1}.43\bar{2}34.123 \in \GQ_5^{(1.5)}$.
Similarly, using (\ref{eq:whandleA}) we get $\bar{4}w_- 4 = 321.\bar{4}\bar{3}2\bar{3}\bar{4}.\bar{1}\bar{2}\bar{3} \in \GQ_5^{(1.5)}$. Applying $F$ we
get $4w_-\bar{4} \in \GQ_5^{(1.5)}$ and this concludes (1).

We now prove (2). We compute modulo the $\GQ_4$-bimodule $\GQ_5^{(1.5)}$. We proved that $(3 - a.\emptyset).4w_-4 \in u_3. \bar{4}.\GQ_4.4 + \GQ_4.4 \bar{3}4$
hence $(3 - a.\emptyset).4w_-4 \equiv 0$ by (1). By the computations of section \ref{sect:A4tildebimod} we know that $(2- a.\emptyset).w_{\pm} \in \GQ_3.3\bar{2}3.\GQ_3$
and $(1- a.\emptyset).w_{\pm} \in \GQ_3.3\bar{2}3.\GQ_3$. It follows that $\la.4 w_-4 \equiv \eps(\la).4 w_- 4$ for all $\la \in \GQ_4$.

From relation (\ref{eq:343rel}) we get similarly that $a.(\emptyset - a.\bar{3}).\bar{4}3 \in u_3 4 u_4 + a^2.(a.\emptyset-3).\bar{4}\bar{3} + u_3. \bar{4}$
whence $a.(\emptyset - a.\bar{3})\bar{4}w_+\bar{4} \in \GQ_5^{(1.5)}$. This can be rewritten as $\bar{3}.\bar{4}w_+\bar{4} \equiv \eps(\bar{3})$.
Since $\GQ_4$ is generated by $1,2,\bar{3}$ this yields $\la.\bar{4}w_+\bar{4} \equiv \eps(\la).\bar{4}w_+\bar{4}$ for all $\la \in \GQ_4$.

Again by (\ref{eq:343rel}) we get that $(\bar{3}-a^{-1}.\emptyset).43 \in (3-a.\emptyset).4\bar{3} + u_3.\bar{4}.u_3 + u_3 u_4$
hence $(\bar{3}-a^{-1}.\emptyset).4w_+4 \in (3-a.\emptyset).4\bar{3}2\bar{2}234 + u_3.\bar{4}.\GQ_4.4.u_3 + u_3 u_4.\GQ_4.u_4$
and this yields $(\bar{3}-a^{-1}.\emptyset).4w_+4  \in \GQ_5^{(1.5)}$ since $4\bar{3}2\bar{1}234 = 21.43\bar{2}34.\bar{1}\bar{2}$
by handle reduction. Since we already know $\la.4w_+4  = 4(\la.w_+)4 \equiv \eps(\la)4w_+4$ for $\la \in \GQ_3$,
this implies $\la.4w_+4 \equiv \eps(\la) 4 w_+4$. The case $\la.\bar{4}w_-\bar{4} \equiv \eps(\la).\bar{4}w_-\bar{4}$
is similar and left to the reader.

We have now proved that $\la.X \equiv \eps(\la)X$ for all $\la \in \GQ_4$ and $X \in  \{ 4w_+4 ,\bar{4}w_-\bar{4} ,4w_-4 ,\bar{4}w_+\bar{4} \}$.
Since $F(X) = X$ this implies $X.\la \equiv \eps(\la)X$ for all $\la,X$, and (2).

\end{proof}

We now claim that $\GQ_5 \neq \GQ_5^{(2)}$, that is $\GQ_5^{(2)} \subsetneq \GQ_5^{(3)}$. In order to prove this, one just need to check this on one
specialization over a field $\kk$, by comparing the dimensions of the two. Let us consider the specialization at $\{a,b,c \} = \mu_3(\kk)$ with $|\mu_3(\kk)| = 3$, in which case $\GQ_5$ is a quotient of the group
algebra $\kk \Gamma_5$, with $\Gamma_5 = B_5/s_1^3$. If $car. \kk \not\in \{2,3,5 \}$ then the algebra $\kk \Gamma_5$ is split semisimple (of dimension 155520), and one has a description of its irreducible
representation. Therefore, one can identify (this specialization of) $\GQ_5$ to a sum of matrix algebras. By computer means, we get that $\GQ_5^{(2)}$ has dimension $6489$
(over $\kk = \F_{103}$) while $\GQ_5=\GQ_5^{(3)}$ has dimension $6490$. Similarly, we check that $\GQ_5^{(1.5)}$ has dimension $6485 = 6489-4$, thus the equality
$\GQ_5^{(2)} = \GQ_5^{(1.5)} + R.4w_+4 + R.\bar{4}w_-\bar{4} + R.4w_-4 + R.\bar{4}w_+\bar{4}$ is sharp.
\bigskip

We now consider $\GQ_5^{(3)}$.

\subsection{$\GQ_5^{(3)}/\GQ_5^{(2)}$ as a $\GQ_4$-bimodule}

We first need to prove the following lemma.

\begin{lemma} {\ }
\begin{enumerate}
\item $sh(\GQ_4) \subset \GQ_5^{(2)}$
\item $u_4 \GQ_4 u_4u_3u_4 \subset \GQ_5^{(2)}$.
\item $u_4u_3u_4 \GQ_4 u_4 \subset \GQ_5^{(2)}$.
\item $u_4(u_3u_2u_3)(u_2u_1u_2)u_4(u_3u_2u_3)u_4 = u_4(u_3u_2u_3)u_4(u_2u_1u_2)(u_3u_2u_3)u_4 \subset \GQ_5^{(2)}$
\end{enumerate}
\end{lemma} 
Actually part (4) easily implies the first items, but the first items will be easier to use in the sequel.

\begin{proof}
From $\GQ_4 = \GQ_4^{(2)}$ we get (1). For (2) we need to prove that $s_4^{\alpha} \GQ_4 u_4u_3u_4 \subset \GQ_5^{(2)}$ for $\alpha \in \{-1,1\}$.
We have $u_4u_3u_4 \subset sh^2(\GQ_3) = R.s_4^{\alpha}s_3^{-\alpha}s_4^{\alpha} + u_3u_4u_3$ hence
$s_4^{\alpha} \GQ_4 u_4u_3u_4\subset \GQ_5^{(2)}$ iff $s_4^{\alpha} \GQ_4 s_4^{\alpha}s_3^{-\alpha}s_4^{\alpha} \subset \GQ_5^{(2)}$.

Let $A \in \GQ_3 u_3 \GQ_3$. We prove that $s_4^{\alpha} A s_4^{\alpha}s_3^{-\alpha}s_4^{\alpha} \in \GQ_5^{(2)}$.
We have $s_4^{\alpha} A s_4^{\alpha}s_3^{-\alpha}s_4^{\alpha} \in 
s_4^{\alpha} \GQ_3 u_3 \GQ_3 s_4^{\alpha}s_3^{-\alpha}s_4^{\alpha}=
\GQ_3s_4^{\alpha}  u_3 \GQ_3 s_4^{\alpha}s_3^{-\alpha}s_4^{\alpha}$.
Now $s_4^{\alpha}s_3^{-\alpha}s_4^{\alpha} \in R.s_4^{-\alpha}s_3^{\alpha}s_4^{-\alpha} + u_3u_4u_3$
hence $s_4^{\alpha}  u_3 \GQ_3 s_4^{\alpha}s_3^{-\alpha}s_4^{\alpha}  \subset
s_4^{\alpha}  u_3 \GQ_3 s_4^{-\alpha}s_3^{\alpha}s_4^{-\alpha} + \GQ_5^{(2)}$.
Since $s_4^{\alpha}  u_3 \GQ_3 s_4^{-\alpha}s_3^{\alpha}s_4^{-\alpha} = (s_4^{\alpha}  u_3 s_4^{-\alpha})\GQ_3 s_3^{\alpha}s_4^{-\alpha} \subset \GQ_5^{(2)}$
we get $s_4^{\alpha} A s_4^{\alpha}s_3^{-\alpha}s_4^{\alpha} \in \GQ_5^{(2)}$ for $A \in \GQ_3 u_3 \GQ_3$.

Now, we note that we can assume $\alpha = 1$, up to applying $\Phi$. Let us assume $A \in \GQ_3.3 \bar{2}3.\GQ_3$.
We need to prove that $4 3 \bar{2}3 \GQ_3 4\bar{3}4 \in \GQ_5^{(2)}$. We have $\GQ_3 = u_1u_2u_1 + u_2u_1u_2$.
Then, $4 3 \bar{2}3 (u_2u_1u_2) 4\bar{3}4 = 4 (3 \bar{2}3 u_2)u_1u_2 4\bar{3}4$ and we know that
$(3 \bar{2}3 u_2) \subset R.3 \bar{2}3  + u_2 u_3 u_2$ hence 
$4 (3 \bar{2}3 u_2)u_1u_2 4\bar{3}4 \subset 4 3 \bar{2}3u_1u_2 4\bar{3}4 + u_24 u_3u_2u_1u_2 4\bar{3}4$.
Since we already proved $u_24 u_3u_2u_1u_2 4\bar{3}4 \subset 4 \GQ_3u_3\GQ_3 4\bar{3}4 \subset \GQ_5^{(2)}$
it thus remains to prove that $4 3 \bar{2}3u_1u_2 u_14\bar{3}4 \subset \GQ_5^{(2)}$. It follows from $4 \bar{3}4 \subset
R.\bar{4}3\bar{4} + u_3u_4u_3$, as depicted in figure \ref{fig:A53fig1}. There we depict words in the generators as music notes on the stave : bullets correspond to Artin generators, with white/black coloring corresponds to $\pm$
power signs, grey coloring corresponds to indeterminate power signes, and the height of the bullet determines the position of the strand.

We now assume $A \in R.w_+$. Then $4 A4\bar{3}4 \in R.4 w_+ 4\bar{3}4 = R.\bar{3}4 w_+ 44  \subset \GQ_5^{(2)}$.
Now assume $A \in R.w_-$. Then $4 A4\bar{3}4 \in R.4 w_- 4\bar{3}4$
and $4 w_- 4\bar{3}4 = 4 \bar{3}\bar{2}1\bar{2}(\bar{3} 4\bar{3}4) \in
a^{-1}.4 \bar{3}\bar{2}1\bar{2}4\bar{3}4 + \GQ_5^{(2)}$.
But $4 \bar{3}\bar{2}1\bar{2}4\bar{3}4 \in 4 u_3\GQ_34\bar{3}4 \subset \GQ_5^{(2)}$ hence  $4 A4\bar{3}4 \in \GQ_5^{(2)}$.

Since $\GQ_4 = \GQ_3 u_3 \GQ_3 + \GQ_3 .3\bar{2}3.\GQ_3 + R.w_+ + R.w_-$ this proves (2). (3) follows by applying $F$.

We now prove (4). First note that $u_3u_2u_3u_2 = sh(\GQ_3) = u_2u_3u_2u_3$,
whence 
$$\begin{array}{lcl}
u_4(u_3u_2u_3)(u_2u_1u_2)u_4(u_3u_2u_3)u_4 &=& u_4(u_3u_2u_3u_2)u_1u_4(u_2u_3u_2u_3)u_4 \\
&=&  u_2u_4(u_3u_2u_3)u_1u_4(u_3u_2u_3)u_4u_2
\end{array}
$$ and we need to prove $u_4(u_3u_2u_3)u_1u_4(u_3u_2u_3)u_4 \subset \GQ_5^{(2)}$.
Now, since 
$$
s_4^{\alpha} u_3u_2u_3s_4^{-\alpha} = sh(s_3^{\alpha} u_2u_1u_2s_3^{-\alpha}) \subset \GQ_5^{(1)} + \GQ_4 u_4u_3u_4$$
by lemma \ref{lemGQ4L2}, we only need to prove
$s_4^{\alpha}(u_3u_2u_3)u_1s_4^{\alpha}(u_3u_2u_3)s_4^{\alpha} \subset \GQ_5^{(2)}$
for $\alpha \in \{-1,1 \}$. Using $\Phi$, we can assume $\alpha = 1$. We then use
that $u_3u_2u_3 \subset R.3\bar{2}3 + u_2u_3u_2$ and that
$s_4(u_2u_3u_2)u_1s_4(u_3u_2u_3)s_4 = u_2s_4u_3s_4u_2u_1(u_3u_2u_3)s_4 \subset \GQ_5^{(1)}$ by (2),
so we only need to prove $43\bar{2}3.u_1.43\bar{2}34 \subset \GQ_5^{(2)}$.
But $43\bar{2}3.u_1.43\bar{2}34 = 43\bar{2}.u_1.343\bar{2}34= 43\bar{2}.u_1.434\bar{2}34= 434\bar{2}.u_1.3\bar{2}434
= 343\bar{2}.u_1.3\bar{2}343 \subset \GQ_5^{(2)}$ as in figure \ref{fig:A53fig2}, and this concludes the proof of (4).

\end{proof}

We then claim that $4.3\bar{2}3.\GQ_3.4.3\bar{2}3.4 \subset R.4.3\bar{2}3.1 \bar{2}1.4.3\bar{2}3.4 + \GQ_5^{(2)}$. This is depicted in figure \ref{fig:A53fig2},
as well as the fact that $4 xu_2u_1 4 u_2 x4 \subset \GQ_5^{(2)}$.

\begin{lemma} \label{lem:4xy4x4} {\ } For all $i \in \{1,2,3 \}$ we have $s_i.4xy4x4 \equiv 4xy4x4.s_i \equiv a.4xy4x4 \mod \GQ_5^{(2)}$.
\end{lemma}
\begin{proof}
From the computations in section \ref{sect:A4tildegen} we get that $s_i.xy \equiv a.xy \mod \GQ_3 u_3 \GQ_3 + u_2u_1 x u_1 u_2$
for $i \in \{1,2 \}$ (notice that the image of $u_2u_1 x u_1 u_2$ in $\tilde{A}_4$ is spanned by the $e_i$ for $i < 22$).
Since $4u_2u_1 x u_1 u_24x4 = u_2u_1 4x u_1 u_24x4 \subset \GQ_5^{(2)}$ this implies
$s_i.4xy4x4 \equiv a.4xy4x4 \mod \GQ_5^{(2)}$ for $i \in \{1,2 \}$.

Since $\GQ_5^{(2)}$ is stable under $F = \Phi \circ \Psi$ and $F(4xy4x4) = 4x4yx4 = 4xy4x4$
this implies $4xy4x4.s_i \equiv a.4xy4x4 \mod \GQ_5^{(2)}$ for $i \in \{1,2 \}$. For the same reason, the statement 
$3.4xy4x4 \equiv a.4xy4x4$ is equivalent to the statement $4xy4x4.3 \equiv a.4xy4x4$.

We use equation (\ref{eq:343rel}) under the form $(\bar{3}-a^{-1}.\emptyset).43 \in (3 -a.\emptyset).4\bar{3}+ u_3\bar{4}u_3 + u_3u_4$
to get $(\bar{3}-a^{-1}.\emptyset).4xy4x4 = (\bar{3}-a^{-1}.\emptyset).43\bar{2}3y4x4 \in \GQ_5^{(2)}$ since
\begin{itemize}
\item $(3 -a.\emptyset).4\bar{3}\bar{2}3y4x4 = (3 -a.\emptyset).42\bar{3}\bar{2}y4x4 \in(3 -a.\emptyset)2u_3.4\bar{3}4\bar{2}yx4 \subset \GQ_5^{(2)}$
\item $u_3\bar{4}u_3\bar{2}3y4x4\subset u_3\bar{4}u_3\bar{2}34yx4$ and $\bar{4}u_3\bar{2}34 = R.\bar{4}\bar{2}34+ R.\bar{4}3\bar{2}34 + R.\bar{4}\bar{3}\bar{2}34
= R.\bar{2}(\bar{4}34)+ R.(\bar{4}34)\bar{2}(\bar{4}34) + R.\bar{4}2\bar{3}\bar{2}4 \subset \GQ_5^{(1)} + R.34\bar{3}\bar{2}3)4\bar{3} \subset \GQ_5^{(1)} + u_3u_24u_34u_2u_3$
whence
$\bar{4}u_3\bar{2}34yx4 \subset \GQ_5^{(1)}\GQ_4u_4+u_3u_24u_34u_2u_3yx4 \subset \GQ_5^{(2)}$.
\item $u_3u_4\bar{2}3y4x4=u_3\bar{2}u_434yx4 \subset \GQ_5^{(2)}$ 
\end{itemize}
Therefore $s_3^{-1}.4xy4x4 \equiv a^{-1}.4xy4x4$ whence $s_3.4xy4x4 \equiv a.4xy4x4$ and this completes the proof of the lemma.
\end{proof}

\begin{figure}
\begin{tikzpicture} [scale=.3]
\fill[color=yellow] (0,0) rectangle (11,0.7*4+0.7);
\draw[color=black] (0,0.7*1) -- (11,0.7*1);
\draw[color=black] (0,0.7*2) -- (11,0.7*2);
\draw[color=black] (0,0.7*3) -- (11,0.7*3);
\draw[color=black] (0,0.7*4) -- (11,0.7*4);
\draw [fill=white] (1,0.7*4) circle (0.4); 
\draw [fill=white] (2,0.7*3) circle (0.4); 
\draw [fill=black] (3,0.7*2) circle (0.4); 
\draw [fill=white] (4,0.7*3) circle (0.4); 
\draw [fill=gray] (5,0.7*1) circle (0.4); 
\draw [fill=gray] (6,0.7*2) circle (0.4); 
\draw [fill=gray] (7,0.7*1) circle (0.4); 
\draw [fill=white] (8,0.7*4) circle (0.4); 
\draw [fill=black] (9,0.7*3) circle (0.4); 
\draw [fill=white] (10,0.7*4) circle (0.4); 
\end{tikzpicture} 

\begin{tikzpicture} [scale=.3] 
\fill[color=yellow] (0,0) rectangle (11,0.7*4+0.7);
\draw[color=black] (0,0.7*1) -- (11,0.7*1);
\draw[color=black] (0,0.7*2) -- (11,0.7*2);
\draw[color=black] (0,0.7*3) -- (11,0.7*3);
\draw[color=black] (0,0.7*4) -- (11,0.7*4);
\draw [fill=white] (1,0.7*4) circle (0.4); 
\draw [fill=white] (2,0.7*3) circle (0.4); 
\draw [fill=black] (3,0.7*2) circle (0.4); 
\draw [fill=white] (4,0.7*3) circle (0.4); 
\draw [fill=gray] (5,0.7*1) circle (0.4); 
\draw [fill=gray] (6,0.7*2) circle (0.4); 
\draw [fill=gray] (7,0.7*1) circle (0.4); 
\draw [fill=black] (8,0.7*4) circle (0.4); 
\draw [fill=white] (9,0.7*3) circle (0.4); 
\draw [fill=black] (10,0.7*4) circle (0.4); 
\end{tikzpicture}

\begin{tikzpicture} [scale=.3] 
\fill[color=yellow] (0,0) rectangle (11,0.7*4+0.7);
\draw[color=black] (0,0.7*1) -- (11,0.7*1);
\draw[color=black] (0,0.7*2) -- (11,0.7*2);
\draw[color=black] (0,0.7*3) -- (11,0.7*3);
\draw[color=black] (0,0.7*4) -- (11,0.7*4);
\draw [fill=white] (1,0.7*4) circle (0.4); 
\draw [fill=white] (2,0.7*3) circle (0.4); 
\draw [fill=black] (3,0.7*2) circle (0.4); 
\draw [fill=white] (4,0.7*3) circle (0.4); 
\draw [fill=black] (5,0.7*4) circle (0.4); 
\draw [fill=gray] (6,0.7*1) circle (0.4); 
\draw [fill=gray] (7,0.7*2) circle (0.4); 
\draw [fill=gray] (8,0.7*1) circle (0.4); 
\draw [fill=white] (9,0.7*3) circle (0.4); 
\draw [fill=black] (10,0.7*4) circle (0.4); 
\draw (0.5,0.9) rectangle (5.5,3.5);
\end{tikzpicture} 

\begin{tikzpicture} [scale=.3] 
\fill[color=yellow] (0,0) rectangle (11,0.7*4+0.7);
\draw[color=black] (0,0.7*1) -- (11,0.7*1);
\draw[color=black] (0,0.7*2) -- (11,0.7*2);
\draw[color=black] (0,0.7*3) -- (11,0.7*3);
\draw[color=black] (0,0.7*4) -- (11,0.7*4);
\draw [fill=gray] (1,0.7*2) circle (0.4); 
\draw [fill=black] (2,0.7*3) circle (0.4); 
\draw [fill=gray] (2,0.7*1) circle (0.4); 
\draw [fill=white] (3,0.7*4) circle (0.4); 
\draw [fill=black] (4,0.7*3) circle (0.4); 
\draw [fill=gray] (4,0.7*1) circle (0.4); 
\draw [fill=gray] (5,0.7*2) circle (0.4); 
\draw [fill=gray] (6,0.7*1) circle (0.4); 
\draw [fill=gray] (7,0.7*2) circle (0.4); 
\draw [fill=gray] (8,0.7*1) circle (0.4); 
\draw [fill=white] (9,0.7*3) circle (0.4); 
\draw [fill=black] (10,0.7*4) circle (0.4); 
\end{tikzpicture} 
\caption{$4.3\bar{2}3.(u_1u_2u_1).4.\bar{3}. 4 \subset \GQ_5^{(2)}$}
\label{fig:A53fig1}.
\end{figure}
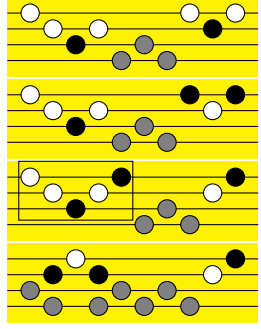

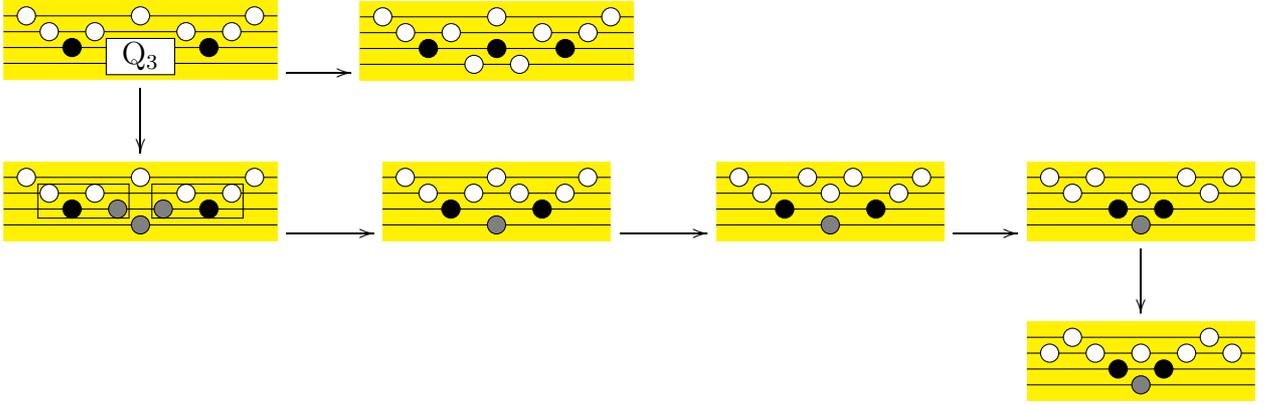
\begin{figure}
\xymatrix{
\begin{tikzpicture} [scale=.3] 
\fill[color=yellow] (0,0) rectangle (12,0.7*4+0.7);
\draw[color=black] (0,0.7*1) -- (12,0.7*1);
\draw[color=black] (0,0.7*2) -- (12,0.7*2);
\draw[color=black] (0,0.7*3) -- (12,0.7*3);
\draw[color=black] (0,0.7*4) -- (12,0.7*4);
\draw [fill=white] (1,0.7*4) circle (0.4); 
\draw [fill=white] (2,0.7*3) circle (0.4); 
\draw [fill=black] (3,0.7*2) circle (0.4); 
\draw [fill=white] (4,0.7*3) circle (0.4); 
\draw [fill=gray] (5,0.7*1) circle (0.4); 
\draw [fill=white] (6,0.7*4) circle (0.4); 
\draw [fill=gray] (6,0.7*2) circle (0.4); 
\draw [fill=gray] (7,0.7*1) circle (0.4); 
\draw [fill=white] (8,0.7*3) circle (0.4); 
\draw [fill=black] (9,0.7*2) circle (0.4); 
\draw [fill=white] (10,0.7*3) circle (0.4); 
\draw [fill=white] (11,0.7*4) circle (0.4); 
\draw [fill=white] (4.5,0.2) rectangle (7.5,1.8); 
\draw (6,1) node {\large $\GQ_3$};
\end{tikzpicture} \ar[r] \ar[d] & \begin{tikzpicture} [scale=.3] 
\fill[color=yellow] (0,0) rectangle (12,0.7*4+0.7);
\draw[color=black] (0,0.7*1) -- (12,0.7*1);
\draw[color=black] (0,0.7*2) -- (12,0.7*2);
\draw[color=black] (0,0.7*3) -- (12,0.7*3);
\draw[color=black] (0,0.7*4) -- (12,0.7*4);
\draw [fill=white] (1,0.7*4) circle (0.4); 
\draw [fill=white] (2,0.7*3) circle (0.4); 
\draw [fill=black] (3,0.7*2) circle (0.4); 
\draw [fill=white] (4,0.7*3) circle (0.4); 
\draw [fill=white] (5,0.7*1) circle (0.4); 
\draw [fill=white] (6,0.7*4) circle (0.4); 
\draw [fill=black] (6,0.7*2) circle (0.4); 
\draw [fill=white] (7,0.7*1) circle (0.4); 
\draw [fill=white] (8,0.7*3) circle (0.4); 
\draw [fill=black] (9,0.7*2) circle (0.4); 
\draw [fill=white] (10,0.7*3) circle (0.4); 
\draw [fill=white] (11,0.7*4) circle (0.4); 
\end{tikzpicture}\\
\begin{tikzpicture} [scale=.3] 
\fill[color=yellow] (0,0) rectangle (12,0.7*4+0.7);
\draw[color=black] (0,0.7*1) -- (12,0.7*1);
\draw[color=black] (0,0.7*2) -- (12,0.7*2);
\draw[color=black] (0,0.7*3) -- (12,0.7*3);
\draw[color=black] (0,0.7*4) -- (12,0.7*4);
\draw [fill=white] (1,0.7*4) circle (0.4); 
\draw [fill=white] (2,0.7*3) circle (0.4); 
\draw [fill=black] (3,0.7*2) circle (0.4); 
\draw [fill=white] (4,0.7*3) circle (0.4); 
\draw [fill=gray] (5,0.7*2) circle (0.4); 
\draw [fill=white] (6,0.7*4) circle (0.4); 
\draw [fill=gray] (6,0.7*1) circle (0.4); 
\draw [fill=gray] (7,0.7*2) circle (0.4); 
\draw [fill=white] (8,0.7*3) circle (0.4); 
\draw [fill=black] (9,0.7*2) circle (0.4); 
\draw [fill=white] (10,0.7*3) circle (0.4); 
\draw [fill=white] (11,0.7*4) circle (0.4); 
\draw (1.5,1) rectangle (5.5,2.5);
\draw (6.5,1) rectangle (10.5,2.5);
\end{tikzpicture} \ar[r]  & \begin{tikzpicture} [scale=.3] 
\fill[color=yellow] (0,0) rectangle (10,0.7*4+0.7);
\draw[color=black] (0,0.7*1) -- (10,0.7*1);
\draw[color=black] (0,0.7*2) -- (10,0.7*2);
\draw[color=black] (0,0.7*3) -- (10,0.7*3);
\draw[color=black] (0,0.7*4) -- (10,0.7*4);
\draw [fill=white] (1,0.7*4) circle (0.4); 
\draw [fill=white] (2,0.7*3) circle (0.4); 
\draw [fill=black] (3,0.7*2) circle (0.4); 
\draw [fill=white] (4,0.7*3) circle (0.4); 
\draw [fill=white] (5,0.7*4) circle (0.4); 
\draw [fill=gray] (5,0.7*1) circle (0.4); 
\draw [fill=white] (6,0.7*3) circle (0.4); 
\draw [fill=black] (7,0.7*2) circle (0.4); 
\draw [fill=white] (8,0.7*3) circle (0.4); 
\draw [fill=white] (9,0.7*4) circle (0.4); 
\end{tikzpicture}\ar[r]  & 
\begin{tikzpicture} [scale=.3] 
\fill[color=yellow] (0,0) rectangle (10,0.7*4+0.7);
\draw[color=black] (0,0.7*1) -- (10,0.7*1);
\draw[color=black] (0,0.7*2) -- (10,0.7*2);
\draw[color=black] (0,0.7*3) -- (10,0.7*3);
\draw[color=black] (0,0.7*4) -- (10,0.7*4);
\draw [fill=white] (1,0.7*4) circle (0.4); 
\draw [fill=white] (2,0.7*3) circle (0.4); 
\draw [fill=black] (3,0.7*2) circle (0.4); 
\draw [fill=white] (4,0.7*4) circle (0.4); 
\draw [fill=white] (5,0.7*3) circle (0.4); 
\draw [fill=gray] (5,0.7*1) circle (0.4); 
\draw [fill=white] (6,0.7*4) circle (0.4); 
\draw [fill=black] (7,0.7*2) circle (0.4); 
\draw [fill=white] (8,0.7*3) circle (0.4); 
\draw [fill=white] (9,0.7*4) circle (0.4); 
\end{tikzpicture}\ar[r]  & 
 \begin{tikzpicture} [scale=.3] 
\fill[color=yellow] (0,0) rectangle (10,0.7*4+0.7);
\draw[color=black] (0,0.7*1) -- (10,0.7*1);
\draw[color=black] (0,0.7*2) -- (10,0.7*2);
\draw[color=black] (0,0.7*3) -- (10,0.7*3);
\draw[color=black] (0,0.7*4) -- (10,0.7*4);
\draw [fill=white] (1,0.7*4) circle (0.4); 
\draw [fill=white] (2,0.7*3) circle (0.4); 
\draw [fill=white] (3,0.7*4) circle (0.4); 
\draw [fill=black] (4,0.7*2) circle (0.4); 
\draw [fill=white] (5,0.7*3) circle (0.4); 
\draw [fill=gray] (5,0.7*1) circle (0.4); 
\draw [fill=black] (6,0.7*2) circle (0.4); 
\draw [fill=white] (7,0.7*4) circle (0.4); 
\draw [fill=white] (8,0.7*3) circle (0.4); 
\draw [fill=white] (9,0.7*4) circle (0.4); 
\end{tikzpicture}\ar[d] \\  & & & \begin{tikzpicture} [scale=.3] 
\fill[color=yellow] (0,0) rectangle (10,0.7*4+0.7);
\draw[color=black] (0,0.7*1) -- (10,0.7*1);
\draw[color=black] (0,0.7*2) -- (10,0.7*2);
\draw[color=black] (0,0.7*3) -- (10,0.7*3);
\draw[color=black] (0,0.7*4) -- (10,0.7*4);
\draw [fill=white] (1,0.7*3) circle (0.4); 
\draw [fill=white] (2,0.7*4) circle (0.4); 
\draw [fill=white] (3,0.7*3) circle (0.4); 
\draw [fill=black] (4,0.7*2) circle (0.4); 
\draw [fill=white] (5,0.7*3) circle (0.4); 
\draw [fill=gray] (5,0.7*1) circle (0.4); 
\draw [fill=black] (6,0.7*2) circle (0.4); 
\draw [fill=white] (7,0.7*3) circle (0.4); 
\draw [fill=white] (8,0.7*4) circle (0.4); 
\draw [fill=white] (9,0.7*3) circle (0.4); 
\end{tikzpicture} 
}
\caption{$4.3\bar{2}3.\GQ_3.4.3\bar{2}3.4 \subset R.4.3\bar{2}3.1 \bar{2}1.4.3\bar{2}3.4 + \GQ_5^{(2)}$}
\label{fig:A53fig2}
\end{figure}

Note that $4x\bar{4} = \bar{3}\bar{2}.4\bar{3}4.23$ hence $s_4^{\alpha}xys_4^{\beta}xs_4^{\gamma} \in \GQ_5^{(2)}$ whenever $\alpha,\beta,\gamma \in \{-1, 1\}$ with
$\# \{\alpha,\beta,\gamma\} > 1$.

\begin{lemma} {\ } \label{lem:4w4w4to4xy4x}
\begin{enumerate}
\item $\bar{4} w_+ \bar{4} w_+ \bar{4} \equiv -(a^2/w)^3.4xy4x4 \mod \GQ_5^{(2)}$.
\item $4w_-4w_-4 \equiv - (a^6/w^5). 4xy4x4  \mod \GQ_5^{(2)}$.
\end{enumerate}
\end{lemma}
\begin{proof}
We first prove (1).
We use that $4 \bar{3}4 \equiv (bc).\bar{4}3\bar{4} + 3\bar{4}3 \mod u_3 4\bar{3} + u_4u_3 + u_3u_4$.
Since
\begin{itemize}
\item $\bar{4}32\bar{1}2(u_34\bar{3})2\bar{1}23\bar{4} = \bar{4}32\bar{1}2u_34(\bar{3}2\bar{1}23)\bar{4} = 
 \bar{4}32\bar{1}2u_34 21.3\bar{2}3.\bar{1}\bar{2}.\bar{4} = \bar{4}32\bar{1}2u_3 21.(43\bar{2}3\bar{4}) .\bar{1}\bar{2}
= \bar{4}32\bar{1}2u_3 21.\bar{3}\bar{2}.4\bar{3}4 .23.\bar{1}\bar{2} \subset \GQ_5^{(2)}$
\item $\bar{4}32\bar{1}2u_3u_42\bar{1}23\bar{4} = \bar{4}32\bar{1}2u_32\bar{1}2u_43\bar{4} \subset \GQ_5^{(2)}$
and $\bar{4}32\bar{1}2u_4u_32\bar{1}23\bar{4} = \bar{4}3u_42\bar{1}2u_32\bar{1}23\bar{4} 
 \subset \GQ_5^{(2)}$
 \item $\bar{4}32\bar{1}2(4\bar{3}4)2\bar{1}23\bar{4} =\bar{4}342\bar{1}2\bar{3}2\bar{1}243\bar{4} =34\bar{3}2\bar{1}2\bar{3}2\bar{1}2\bar{3}43 \in \GQ_5^{(2)}$
\end{itemize}
we get that $\bar{4}32\bar{1}2(3\bar{4}3)2\bar{1}23\bar{4}\equiv -(bc). \bar{4}32\bar{1}2\bar{4}3\bar{4}2\bar{1}23\bar{4} \mod \GQ_5^{(2)}$.
Now $\bar{4}32\bar{1}2\bar{4}3\bar{4}2\bar{1}23\bar{4} = (\bar{4}3\bar{4})2\bar{1}232\bar{1}2\bar{4}3\bar{4} \equiv (bc)^{-1}4\bar{3}42\bar{1}232\bar{1}2(\bar{4}3\bar{4})
\equiv (bc)^{-2}.4\bar{3}42\bar{1}232\bar{1}24\bar{3}4 \mod \GQ_5^{(2)}$.
Now $4\bar{3}42\bar{1}232\bar{1}24\bar{3}4 =4\bar{3}2\bar{1}2(434)2\bar{1}2\bar{3}4 =4(\bar{3}2\bar{1}23)4(32\bar{1}2\bar{3})4 =
4.21.x.\bar{1}\bar{2}.4.\bar{2}\bar{1}.x.12.4 = 21.4.x.\bar{1}\bar{2}\bar{2}\bar{1}.4.x.4 .12$. We use that
$\bar{1}\bar{2}\bar{2}\bar{1} \equiv (wbc)^{-1}.1\bar{2}1 \mod u_2 u_1 u_2$ to get
$21.4.x.\bar{1}\bar{2}\bar{2}\bar{1}.4.x.4 .12 \equiv (wbc)^{-1}.21.4xy4x4 .12 \mod \GQ_5^{(2)}$.
Now, by lemma \ref{lem:4xy4x4} we know $21.4xy4x4 .12 \equiv a^4.4xy4x4  \mod \GQ_5^{(2)}$. Altogether this proves (1).

We now prove (2),
We note that $\bar{4}w_+\bar{4}w_+\bar{4} = \bar{4}3 2\bar{1}23\bar{4}3 2\bar{1}23\bar{4} = 3(\bar{3}\bar{4}3) 2\bar{1}23\bar{4}3 2\bar{1}2(3\bar{4}\bar{3})3 
= 34\bar{3}\bar{4} 2\bar{1}23\bar{4}3 2\bar{1}2\bar{4}\bar{3}43 
= 34\bar{3} 2\bar{1}2(\bar{4}3\bar{4}3\bar{4}) 2\bar{1}2\bar{3}43$. Now,
\begin{itemize}
\item $4\bar{3} 2\bar{1}2(u_3u_4) 2\bar{1}2\bar{3}4 = 4\bar{3} 2\bar{1}2u_3 2\bar{1}2u_4\bar{3}4  \subset \GQ_5^{(2)}$
\item $4\bar{3} 2\bar{1}2(u_4u_3) 2\bar{1}2\bar{3}4 = 4\bar{3} u_42\bar{1}2u_3 2\bar{1}2\bar{3}4 \subset \GQ_5^{(2)}$
\item $4\bar{3} 2\bar{1}2(\bar{4}u_3u_4) 2\bar{1}2\bar{3}4 = (4\bar{3}\bar{4}) 2\bar{1}2u_3 2\bar{1}2u_4\bar{3}4 = \bar{3}\bar{4}3 2\bar{1}2u_3 2\bar{1}2u_4\bar{3}4 
\subset \GQ_5^{(2)} $
\item $4\bar{3} 2\bar{1}2(u_4u_3\bar{4}) 2\bar{1}2\bar{3}4 = 4\bar{3} u_42\bar{1}2u_3 2\bar{1}2(\bar{4}\bar{3}4)= 4\bar{3} u_42\bar{1}2u_3 2\bar{1}23\bar{4}\bar{3}
\subset \GQ_5^{(2)} $
\item $4\bar{3} 2\bar{1}2(343) 2\bar{1}2\bar{3}4 = 4\bar{3} 2\bar{1}234(3 2\bar{1}2\bar{3})4 = 4\bar{3} 2\bar{1}23\bar{2} \bar{1}43\bar{2}3412$ by handle reduction,
and this is equal to $4(\bar{3} 2\bar{1}23)\bar{2} \bar{1}4x412=214 3\bar{2}3\bar{1}\bar{2}\bar{2} \bar{1}4x412=214 x\bar{1}\bar{2}\bar{2} \bar{1}4x412$.
Expanding $\bar{2}\bar{2}$ by the cubic relation, we get $\bar{1}\bar{2}\bar{2} \bar{1} \equiv w^{-1}.\bar{1}2\bar{1} \mod u_2u_1u_2$, hence
$214 x\bar{1}\bar{2}\bar{2} \bar{1}4x412 \equiv w^{-1}.214 x\bar{1}2\bar{1}4x412 \mod \GQ_5^{(2)}$. Now $ \bar{1}2\bar{1}  \equiv (bc)^{-1} \mod u_2u_1u_2$
hence $214 x\bar{1}\bar{2}\bar{2} \bar{1}4x412 \equiv (wbc)^{-1}.214 x1\bar{2}14x412\equiv (wbc)^{-1}.21.4 xy4x4.12\equiv (a^4/wbc).4 xy4x4 \mod \GQ_5^{(2)}$.
\item $4\bar{3} 2\bar{1}2(\bar{3}4\bar{3}) 2\bar{1}2\bar{3}4 = 4\bar{3} (2\bar{1}2)\bar{3}4\bar{3} (2\bar{1}2)\bar{3}4$.
We have $2\bar{1}2 \equiv (bc).\bar{2}1\bar{2} \mod u_1u_2u_1$ hence
$4\bar{3} (2\bar{1}2)\bar{3}4 \equiv (bc).4\bar{3} \bar{2}1\bar{2}\bar{3}4 \mod u_14\bar{3} u_2\bar{3}4u_1$. Now
$u_14\bar{3} u_2\bar{3}4u_1 \subset u_14u_2u_3u_24u_1 + u_143 \bar{2}34u_1$
whence $4\bar{3} (2\bar{1}2)\bar{3}4\bar{3} (2\bar{1}2)\bar{3}4 \equiv
(bc).4\bar{3} (2\bar{1}2)\bar{3}4\bar{3} \bar{2}1\bar{2}\bar{3}4 \mod 
4\bar{3} (2\bar{1}2)\bar{3}u_14u_2u_3u_24u_1 +
4\bar{3} (2\bar{1}2)\bar{3}u_143 \bar{2}34u_1$. Since
$4\bar{3} (2\bar{1}2)\bar{3}u_14u_2u_3u_24u_1 = 4\bar{3} (2\bar{1}2)\bar{3}u_1u_24u_34u_2u_1 \subset \GQ_5^{(2)}$
and $4\bar{3} (2\bar{1}2)\bar{3}u_143 \bar{2}34u_1 = 4\bar{3} (2\bar{1}2)u_1(\bar{3}43) \bar{2}34u_1= 4\bar{3} (2\bar{1}2)u_143\bar{4} \bar{2}34u_1
= 4\bar{3} (2\bar{1}2)u_143 \bar{2}(\bar{4}34)u_1 \subset \GQ_5^{(2)}$
this implies $4\bar{3} (2\bar{1}2)\bar{3}4\bar{3} (2\bar{1}2)\bar{3}4 \equiv
(bc).4\bar{3} (2\bar{1}2)\bar{3}4\bar{3} \bar{2}1\bar{2}\bar{3}4 \mod \GQ_5^{(2)}$. 
Now  $4\bar{3} (2\bar{1}2)\bar{3}4\bar{3} \bar{2}1\bar{2}\bar{3}4 \equiv (bc).4\bar{3} \bar{2}1\bar{2}\bar{3}4\bar{3} \bar{2}1\bar{2}\bar{3}4
\mod 4\bar{3} u_1u_2u_1\bar{3}4\bar{3} \bar{2}1\bar{2}\bar{3}4$
and $4\bar{3} u_1u_2u_1\bar{3}4\bar{3} \bar{2}1\bar{2}\bar{3}4=u_14\bar{3} u_2\bar{3}4\bar{3} u_1\bar{2}1\bar{2}\bar{3}4$
with $4\bar{3} u_2\bar{3}4\bar{3} u_1\bar{2}1\bar{2}\bar{3}4 = 4\bar{3} \bar{3}4\bar{3} u_1\bar{2}1\bar{2}\bar{3}4 + 4\bar{3} \bar{2}\bar{3}4\bar{3} u_1\bar{2}1\bar{2}\bar{3}4
+ 4\bar{3} 2\bar{3}4\bar{3} u_1\bar{2}1\bar{2}\bar{3}4 \subset \GQ_5^{(2)} + 4\bar{3} 2\bar{3}4\bar{3} u_1\bar{2}1\bar{2}\bar{3}4$.
But $4(\bar{3} 2\bar{3})4\bar{3} u_1\bar{2}1\bar{2}\bar{3}4 \equiv (bc)^{-1}43\bar{2}3 4\bar{3} u_1\bar{2}1\bar{2}\bar{3}4 \mod \GQ_5^{(2)}$
and $43\bar{2}3 4\bar{3} u_1\bar{2}1\bar{2}\bar{3}4 = 43\bar{2}\bar{4}3 4 u_1\bar{2}1\bar{2}\bar{3}4= (43\bar{4})\bar{2}3  u_1\bar{2}1\bar{2}4\bar{3}4
= \bar{3}43\bar{2}3  u_1\bar{2}1\bar{2}4\bar{3}4 \in \GQ_5^{(2)}$. This proves $4\bar{3} (2\bar{1}2)\bar{3}4\bar{3} \bar{2}1\bar{2}\bar{3}4 \equiv (bc).4w_-4w_-4
\mod \GQ_5^{(2)}$ hence $4\bar{3} (2\bar{1}2)\bar{3}4\bar{3} (2\bar{1}2)\bar{3}4 \equiv (bc)^2.4w_-4w_-4$.
\end{itemize}
Therefore, by (\ref{eq2b12b12b}) we get
$\bar{4}w_+\bar{4}w_+\bar{4} \equiv  \frac{bc-a^2}{a^4bc}.(a^4/wbc).34 xy4x43+ \frac{bc}{a^2}. (bc)^2.34w_-4w_-43\mod \GQ_5^{(2)}$
that is 
$\bar{4}w_+\bar{4}w_+\bar{4} \equiv  \frac{bc-a^2}{bc}.(a^2/wbc).4 xy4x4+  (bc)^3.4w_-4w_-4 \mod \GQ_5^{(2)}$
hence
$$
4w_-4w_-4 \equiv - (a^6/w^5). 4xy4x4  \mod \GQ_5^{(2)}
$$
and this proves (2).
\end{proof}

\begin{proposition} {\ } 
\begin{enumerate} 
\item $\GQ_5^{(3)} = \GQ_5^{(2)} + R.4 xy 4x 4 $.
\item $\GQ_5^{(3)} = \GQ_5^{(2)} + R.4 w_- 4w_- 4 =  \GQ_5^{(2)}+ R. \bar{4}w_+\bar{4}w_+\bar{4}$
\item Modulo $\GQ_5^{(2)}$, $\la.4 w_- 4w_- 4 \equiv 4 w_- 4w_- 4.\la \equiv     \eps(\la).4 w_- 4w_- 4$ and
$\la.\bar{4}w_+\bar{4}w_+\bar{4} \equiv \bar{4}w_+\bar{4}w_+\bar{4}.\la \equiv     \eps(\la).\bar{4}w_+\bar{4}w_+\bar{4}$
for all $\la \in \GQ_4$.
\end{enumerate}
\end{proposition}
\begin{proof}
By lemma \ref{lem:4w4w4to4xy4x} we know that
$\GQ_5^{(2)} + R.4 xy 4x 4 =
 \GQ_5^{(2)} + R.4 w_- 4w_- 4 =  \GQ_5^{(2)}+ R. \bar{4}w_+\bar{4}w_+\bar{4}$. We denote $H$ this $R$-module, and
 aim to show that $H = \GQ_5^{(3)}$ Clearly $H \subset \GQ_5^{(3)}$, and $\Phi(H) = H$ since $\Phi(w_{\pm}) = w_{\mp}$.
By handle reduction it is thus sufficient to check that $s_4 \alpha s_4 \beta s_4 \in H$ for all $\alpha,\beta \in \GQ_4$.
Recall that $\GQ_4 = \GQ_3 u_3 \GQ_3 + \GQ_3 x \GQ_3 + R.w_+ + R.w_-$. It is thus sufficient to check
$s_4 \alpha s_4 \beta s_4 \in H$ for all $\alpha,\beta \in \{ \GQ_3u_3 \GQ_3, \GQ_3x\GQ_3, w_+,w_- \}$. 
We first consider the following cases.
\begin{itemize}
\item if $\alpha$ or $\beta$ belong to $\GQ_3u_3\GQ_3$ we have $s_4 \alpha s_4 \beta s_4 \in \GQ_5^{(2)} \subset H$.
\item if $\alpha$ and $\beta$ belong to $\GQ_3x\GQ_3$
 we have that $s_4 \alpha s_4 \beta s_4 \in \GQ_3 s_4 x \gamma s_4 x s_4 \GQ_3$
 with $\gamma \in \GQ_3 = R.y + u_2u_1u_2$ and
 $s_4 x \gamma s_4 x s_4 \in R.s_4 x y s_4 x s_4 + s_4 x  u_2u_1u_2 \subset  R.s_4 x y s_4 x s_4 + \GQ_5^{(2)} = H$.
 \end{itemize}
 We now notice that the cubic relation almost immediatly implies
 $$
 \left\lbrace \begin{array}{lcl}
 w_+ &\in & w^{-1}.w_0 + \GQ_3 u_3 \GQ_3 +\GQ_3.x \\
 w_- &\in & w.w_0^{-1} + \GQ_3 u_3 \GQ_3 +\GQ_3.x \\
 \end{array} \right.
 $$
where $w_0 = 321123$ centralizes $s_1$ (and $s_2$). By the above cases we know this implies
that $s_4 x \GQ_3 s_4 w_{\pm} s_4 \subset s_4 x s_4 w_{\pm} s_4 \GQ_3 + H$.
Therefore,
\begin{itemize}
\item if $\alpha \in \GQ_3 x \GQ_3$ and $\beta = w_+$, we have $s_4 \alpha s_4 \beta s_4 \in H$
iff $s_4 x s_4 w_+ s_4 \in H$. But $x$ commutes with $s_4 w_+ s_4$ whence $s_4 x s_4 w_+ s_4 = s_4 s_4 w_+ s_4x \in \GQ_5^{(2)} \subset H$
and this solves the case;
\item if $\alpha \in \GQ_3 x \GQ_3$ and $\beta = w_-$, we have $s_4 \alpha s_4 \beta s_4 \in H$
iff $s_4 x s_4 w_- s_4 \in H$. But $s_4 x s_4 w_- s_4 = 43\bar{2}(34\bar{3})\bar{2}1\bar{2}\bar{3}4
= 43\bar{2}\bar{4}34\bar{2}1\bar{2}\bar{3}4 = (43\bar{4})\bar{2}3\bar{2}1\bar{2}4\bar{3}4= \bar{3}43\bar{2}3\bar{2}1\bar{2}4\bar{3}4 \in \GQ_5^{(2)} \subset H$
and this solves the case.
\end{itemize}
The cases where the roles of $\alpha$ and $\beta$ are exchanged are deduced from these ones by applying $F$. Therefore,
we are reduced to considering $\alpha, \beta \in \{w_+,w_- \}$. If $\alpha = \beta$ this is clear, so we can assume $\alpha \neq \beta$,
and via $F$ there is only one case to consider, namely $4w_+4w_-4 = 432\bar{1}2(34\bar{3})\bar{2}1\bar{2}\bar{3}4
= 432\bar{1}2\bar{4}34\bar{2}1\bar{2}\bar{3}4= (43\bar{4})2\bar{1}23\bar{2}1\bar{2}4\bar{3}4= \bar{3}432\bar{1}23\bar{2}1\bar{2}4\bar{3}4 \in \GQ_5^{(2)} \subset H$,
and this proves the claim that $H = \GQ_5^{(3)}$, which implies (1) and (2). Then (3) is an immediate consequence of lemmas \ref{lem:4xy4x4} and \ref{lem:4w4w4to4xy4x}.

\end{proof}

\begin{theorem} $\GQ_5$ is a $R$-module of finite rank, and $\GQ_5 = \GQ_5^{(3)}$.
\end{theorem}
\begin{proof}
It is sufficient to show that $\GQ_5^{(4)} = \GQ_5^{(3)}$. But
$\GQ_5^{(4)} = \GQ_5^{(3)}u_4 \GQ_4 =\GQ_5^{(2)}u_4\GQ_4 + 4xy4x4u_4\GQ_4$.
and $\GQ_5^{(2)}u_4\GQ_4 = \GQ_5^{(3)}$ while $4xy4x4u_4\GQ_4 = 4xy4xu_4\GQ_4 \subset \GQ_5^{(3)}$,
whence the claim.
\end{proof}

\end{document}